\documentclass[a4paper,11pt]{amsbook}
\usepackage{amssymb}
\usepackage{amscd}
\usepackage{tikz}
\usepackage{amsmath}
\usepackage{graphicx}
\usepackage{caption}
\usepackage{subcaption}
\usepackage{calligra}
\DeclareMathAlphabet{\mathcalligra}{T1}{calligra}{m}{n}
\newtheorem{lemma}{Lemma}[chapter]
\newtheorem{prop}{Proposition}[chapter]
\newtheorem{thm}{Theorem}[chapter]
\newtheorem{cor}{Corollary}[section]
\theoremstyle{definition}
\newtheorem{defn}{Definition}[chapter]
\theoremstyle{remark}
\newtheorem{rem}{\bf Remark}[chapter]
\newtheorem{ex}{\bf Example}[chapter]

\newtheorem{conj}{\bf Conjecture}[chapter]

\newtheorem{exercise}{\bf Exercise}[chapter]
\newcounter{numl}
\newcommand{\labelnuml}{\textup{(\roman{numl})}}
\newenvironment{numlist}{\begin{list}{\labelnuml}%
{\usecounter{numl}\setlength{\leftmargin}{0pt}%
\setlength{\itemindent}{2\parindent}%
\setlength{\itemsep}{\smallskipamount}\def
\makelabel ##1{\hss \llap {\upshape ##1}}}}{\end{list}}

\DeclareSymbolFont{script}{U}{eus}{m}{n}
\DeclareSymbolFontAlphabet{\mathscr}{script}
\DeclareMathSymbol{\Wedge}{0}{script}{"5E}
\DeclareMathAlphabet{\mathrmsl}{OT1}{cmr}{m}{sl}
\newcommand{\R}{{\mathbb R}}
\newcommand{\Sph}{{\mathbb S}}
\newcommand{\C}{{\mathbb C}}
\newcommand{\Z}{{\mathbb Z}}

\newcommand{\Q}{{\mathbb Q}}
\newcommand{\T}{{\mathbb T}}
\newcommand{\cH}{{\mathcal H}}
\newcommand{\cC}{{\mathcal C}}
\newcommand{\cF}{{\mathcal F}}

\newcommand{\cS}{{\mathcal S}}
\newcommand{\cE}{{\mathcal E}}
\newcommand{\cO}{{\mathcal O}}

\newcommand{\trace}{\mathop{\mathrm{tr}}\nolimits}

\newcommand{\tor}{{\mathfrak t}}

\newcommand{\Hess}{\mathop{\mathrm{Hess}}}
\newcommand{\eps}{\varepsilon}

\newcommand{\Proj}{\mathrm P}

\begin{document}

\author[V. Apostolov]{Vestislav Apostolov}
\email{apostolov.vestislav@uqam.ca}

\title{The K\"ahler geometry of toric manifolds}

\maketitle
\tableofcontents

\chapter*{Preface} These notes are written for a mini-course I gave at the CIRM in Luminy  in 2019.  Their intended purpose was to present, in the context of smooth toric varieties,  a relatively self-contained and elementary introduction to the theory of extremal K\"ahler metrics pioneered by E. Calabi in the 1980's and extensively developed in recent years. The framework of toric varieties,  used in both symplectic and algebraic geometry,  offers a fertile testing ground for the general theory of extremal K\"ahler metrics and provides an important class of smooth complex varieties for which the existence theory is now understood in terms of a stability condition of the corresponding Delzant polytope. The notes do not contain any original material nor do they take into account  some more recent developments, such as the non-Archimedean approach to the Calabi problem.  I'm making them available on the arXiv because I continue to get questions about how they can be cited.

\section*{Acknowledgements} 
I am grateful to \'Eveline Legendre, Alexandro Lep\'e, Simon Jubert, Lars Sektnan, Isaque Viza de Souza, and Yicao Wang for their careful reading of previous versions of the manuscript, spotting and correcting typos and errors, and suggesting  improvements. I have benefited from numerous discussions with Paul Gauduchon,  \'Eveline Legendre and Lars Sektnan during the preparation of the notes.  Special thanks are due to David Calderbank for his collaborations on the topic  and for kindly allowing me to reproduce his elegant (unpublished) proof of Proposition~3.3. 

\chapter{Delzant Theory}\label{ch:Delzant}

\section{Hamiltonian group actions on symplectic manifolds} We start with some basic definitions in symplectic geometry.

\smallskip
Let $(M^{2m}, \omega)$ be a {\it symplectic} manifold of real dimension $2m$, i.e. a smooth manifold endowed with a closed $2$-form $\omega$  which is {\it non-degenerate} at each point $p\in M$ meaning that $\omega_p : T_pM \to T^*_pM$ is an isomorphism. A simple linear algebra shows that the latter condition is also equivalent to $\omega_p^{\wedge m} \neq 0\in \wedge^{2m}(T^*_pM).$ 

{\tiny 
\begin{ex}\label{flat-symplectic}
The basic example of a symplectic manifold is $({\mathbb R}^{2m}, \omega_0)$ where, using the identification $$\R^{2m} \cong \C^m =\{z=(z_1, \cdots, z_m) : z_i=x_i + \sqrt{-1} y_i, i=1, \ldots m\}, $$
the symplectic form is
$$\omega_0 := \sum_{i=1}^{m} dx_i \wedge dy_i = \frac{\sqrt{-1}}{2} \sum_{i=1}^n dz_i \wedge d\bar z_i.$$
\end{ex}}

{\tiny 
\begin{ex}\label{e:S2} Let $M= \Sph^2 \subset \R^3$ be the unit sphere endowed with  the atlas of $\Sph^2$ given by the stereographic projections from the north  and the south poles,  $N=(0, 0, 1)$ and $S=(0,0, -1)$, respectively.  We denote by  
\begin{equation*}
\begin{split} 
x & = \frac{2u}{1 + u^2 + v^2}, y= \frac{2v}{1+ u^2 + v^2}, z= -\frac{1-u^2- v^2}{1+ u^2 + v^2},  \ (u, v) \in \R^2, \\
x &=  \frac{2\tilde u}{1 + \tilde u^2 + \tilde v^2}, y= \frac{2\tilde v}{1+ \tilde u^2 + \tilde v^2}, z= \frac{1-\tilde u^2- \tilde v^2}{1+ \tilde u^2 + \tilde v^2},  \ (\tilde u, \tilde v) \in \R^2
\end{split}
\end{equation*}
the corresponding chart patches on $\Sph^2$. Recall that on $\Sph^2\setminus (\{N\} \cup \{ S\})$,  the transition between the coordiantes $(u,v)$ and $(\tilde u, \tilde v)$ is  given by the diffeomorphism
$\tilde u = \frac{u}{u^2 + v^2} , \tilde v = \frac{v}{u^2 + v^2}$ of $\R^2 \setminus \{(0,0)\}$. 
Then the $2$-form
$$\omega_{\Sph^2} := \frac{4 du \wedge dv}{(1+ u^2 + v^2)^2} =- \frac{4 d\tilde u \wedge d\tilde v}{(1+ \tilde u^2 + \tilde v^2)^2}$$
introduces a symplectic structure on $\Sph^2$.
\end{ex}}

\begin{defn} Let $M$ be a smooth manifold and $G$ a Lie group. An {\it action} of $G$ on $M$ is a group homomorphism
$$\rho : G \to {\rm Diff}(M),$$
where ${\rm Diff}(M)$ stands for the diffeomorphism group of $M$. The action of $G$ is {\it smooth} if the evaluation map
$${\rm ev} : G \times M \to M, \ \ {\rm ev} (g, p):= \rho(g)(p), $$
is a smooth map between manifolds.

Let $(M, \omega)$ be a symplectic manifold and $G$ a Lie group acting smoothly on $M$. We say that $G$  acts {\it symplectically} on $(M, \omega)$ if 
$$\rho(g)^*(\omega)= \omega.$$
\end{defn}

{\tiny 
\begin{ex}\label{flat-action}  Consider $(M, \omega)=(\R^{2m}, \omega_0)$ as in Example~\ref{flat-symplectic}. The action of the $m$-dimensional torus 
$$\T^m= (e^{\sqrt{-1}t_1}, \ldots, e^{\sqrt{-1} t_m}),$$ given by
$$\rho(e^{\sqrt{-1}t_1}, \ldots, e^{\sqrt{-1} t_m})(z) := (e^{\sqrt{-1}t_1} z_1, \ldots, e^{\sqrt{-1} t_m} z_m)$$
is symplectic.
\end{ex}}

{\tiny 
\begin{ex}[{\bf Hamiltonian flows}] \label{hamiltonian-flow} Let $f$ be a smooth function on $(M,\omega)$. Using the non-degeneracy of $\omega$, we define a vector field by
$$X_f := - \omega^{-1}(df), $$
called {\it the Hamiltonian} vector field of $f$.  Suppose that  $X_f$ is complete, i.e. its flow $\varphi_t$ is defined for all $t\in \R$ (this always holds if $M$ is compact).
Then,  $\rho(t):=\varphi_t$ defines a symplectic action of $\R$ on $(M, \omega)$. Indeed, we have
\begin{equation*}
\begin{split}
\frac{d}{dt}_{|_{t=s}} (\varphi_t^*\omega) &= \varphi_s^*\Big(\frac{d}{dt}_{|_{t=0}}  (\varphi^*_t\omega)\Big )  \\
                                                        &= \varphi_s^*\Big({\mathcal L}_{X_f} \omega \Big) \\
                                                        &= \varphi_s^*\Big(d \imath _{X_f} + \imath_{X_f} d) (\omega)\Big) \\
                                                        &= \varphi_s^*\Big(d (- d f)\Big) =0,\\
                            \varphi_0^*(\omega) &= \omega.
                            \end{split}
                            \end{equation*}  
                            \end{ex} }                          
The previous two examples show that different groups (in particular $\T^m$ and $\R^m$) can  possibly give rise to ``equivalent'' symplectic actions, in the sense that their images in ${\rm Diff}(M)$ are the same. A way to normalize the situation is to consider
\begin{defn} A symplectic action of a Lie group $G$ on $(M, \omega)$ is {\it effective} if the homomorphism
$\rho: G \to {\rm Diff}(M)$
has trivial kernel.
\end{defn}
Thus, the symplectic action of $\R^m=\{(t_1, \ldots, t_m)\}$ on $(\R^{2m}, \omega_0)$ defined in Example~\ref{flat-action} is not effective,  whereas the action of $\T^m$ on the same manifold is. In the sequel, we shall be interested in effective symplectic actions.

\smallskip Let $G$ be a Lie group acting smoothly on $M$. Denote by $\mathfrak{g} = {\rm Lie}(G)$ its Lie algebra,  i.e.  $\mathfrak{g}$ is the vector space of left-invariant vector fields on $G$. For any $\xi\in \mathfrak{g}$, we denote by  ${\rm exp}(t\xi)$ the corresponding 1-parameter subgroup of $G,$ and by $X_{\xi}$ the vector field on $M$ induced by the one parameter subgroup $\rho(\exp(t\xi))$ in ${\rm Diff}(M)$, i.e.
$$X_{\xi}(p) := \frac{d}{dt}_{|_{t=0}} \Big( \rho(\exp(t\xi))(p)\Big).$$
 The vector field $X_{\xi}$ is called {\it fundamental} vector field of $\xi \in \mathfrak{g}$. 
 
 \smallskip
 We recall that $G$  also acts on itself by conjugation
 $$\rho(g) (h):= g h g^{-1}, \ \ g\in G, h\in G,$$
 fixing the unitary element $e\in G$ and thus inducing a linear action on the vector space $T_eG \cong \mathfrak{g}$, called the {\it adjoint action} 
 $${\rm Ad} : G \to {\bf GL}(\mathfrak{g}).$$
 The induced linear action on the dual vector space $T^*_eG \cong \mathfrak{g}^*$  is given by ${\rm Ad}^*(g) (\alpha):=  \alpha \circ {\rm Ad}(g^{-1})$  and is called the {\it co-adjoint action}.  We are now ready to give the definition of a hamiltonian action on $(M, \omega)$.

\begin{defn}\label{d:hamiltonian} Let $\rho: G \to {\rm Diff}(M)$ be a  smooth action of a Lie group $G$ on $(M, \omega)$. It is said to be   {\it hamiltonian}  if there exists a 
 smooth map $$\mu : M \to \mathfrak{g}^*,$$ called {\it a momentum map}, which satisfies the  following two conditions.
\begin{enumerate}
\item[\rm (i)]  For any $\xi \in \mathfrak{g}$, the fundamental vector field $X_{\xi}$ satisfies
\begin{equation*}\label{pre-hamiltonian}
\omega(X_{\xi}, \cdot ) = - d\langle \mu, \xi \rangle, 
\end{equation*}
where $\langle \cdot, \cdot \rangle$ denotes the natural pairing between $\mathfrak{g}$ and its dual $\mathfrak{g}^*$.
\item[\rm (ii)]  $\mu$ is equivariant with respect to the action $\rho$ of $G$ on $M$ and the co-adjoint action ${\rm Ad}^*$ of $G$ on $\mathfrak{g}^*$, i.e.
\begin{equation*}\label{equivariant}
\mu\big(\rho(g)(p)\big)= {\rm Ad}^*(g)\big(\mu(p)\big).
\end{equation*}

\end{enumerate}
\end{defn}
\begin{rem}\label{r:1} (a) Notice that the condition (i) of Definition~\ref{d:hamiltonian} implies that any fundamental vector field $X_{\xi}$ is hamiltonian (see Example~\ref{hamiltonian-flow}) with respect to the smooth function 
\begin{equation}\label{mu-xi}
\mu_{\xi}(p) := \langle \mu(p), \xi \rangle.
\end{equation}
In particular, hamiltonian actions are symplectic.

(b) By a  standard argument in Lie theory, the condition (ii) of Definition~\ref{d:hamiltonian}  can be equivalently expressed as
\begin{equation*}
d\mu_{\xi}(X_{\eta})= -\mu_{[\eta, \xi]}.
\end{equation*}
For any two smooth functions $f, h$ on $(M, \omega)$ 
$$\{f, h\}_{\omega} : = \omega(X_{f}, X_{h})$$
is the  so-called  {\it Poisson  bracket}. Thus,  the condition (ii) of Definition~\ref{d:hamiltonian} is equivalent to
\begin{equation} \label{equivariant-bracket}
\{\mu_{\xi}, \mu_{\eta}\}_{\omega}  = -\mu_{[\xi, \eta]}.
\end{equation}
\end{rem}

An important example is  the following
{\tiny \begin{ex}\label{e:kirilov} Suppose $G$ is a semi-simple compact Lie group with Lie algebra $\mathfrak{g}$,  and $\T \subset G$ is  a maximal torus with Lie algebra $\mathfrak{t} \subset \mathfrak{g}$.  Complexifying, we also have a semi-simple complex Lie group $G^{\C}$ with Lie algebra $\mathfrak{g}^c{\C}:= \mathfrak{g}\otimes \C$,   and   a Cartan sub-algebra $\mathfrak{h}:=\mathfrak{t}\otimes \C$. A specific example which one can bear in  mind is $G={\bf SU}(n)$, $\T =\{ {\rm diag}(e^{\sqrt{-1} t_1}, \ldots, e^{\sqrt{-1} t_m}) : t_1 + \cdots + t_m=0\}$. Then $G^{\C} = {\bf SL}(n, \C)$.

The general theory of semi-simple Lie algebras  yields that $\mathfrak{g}^{\C}$ admits a {\it root decomposition}
$$\mathfrak{g}^{\C} = \mathfrak{h} \oplus \bigoplus_{\alpha \in R_+}(\mathfrak{g}^{\C}_{-\alpha} \oplus \mathfrak{g}^{\C}_{\alpha}), $$
where $R \subset \mathfrak{h}^*$ denotes the finite subspace of roots of $\mathfrak{g}^{\C}$ and $R= R_+ \cup (-R_+)$ with $R_+ \cap (-R_+)=\emptyset$  is the choice of a subset $R_+$ of {\it positive roots}. For $\alpha \in R$, we have set
$$\mathfrak{g}^{\C}_{\alpha} := \{u \in \mathfrak{g}^{\C} : {\rm ad}_x (u) = \alpha(x) u,  \ \forall x \in \mathfrak{h}\}.$$

We denote by $M$ an orbit in $\mathfrak{g}$ for the adjoint action ${\rm Ad}: G \to {\bf GL}(\mathfrak{g})$ and assume that $M$ is not a point. The general theory tells us that $M$ is aways the orbit of a non-zero element $x\in \mathfrak{t}$, and that such $x$ is unique if we require moreover  that $-i\alpha(x) \ge 0$ for all $\alpha \in R_+$. If  the unique $x \in \mathfrak{t}$ determined as above also satisfies $-i\alpha(x) > 0$ for all $\alpha \in R_+$, then the orbit $M\cong G/\T$ is {\it principal} for the adjoint action of $G$, and $M$ is called a {\it flag manifold}. In general, $M \cong G/G_x$ where $G_x$ is the subgroup of $G$ fixing $x$ (and whose Lie algebra is $\mathfrak{g}_x = {\rm Ker}({\rm ad}_x)$).

In the case $G={\bf SU}(n),$ this corresponds to the basic fact that any hermitian matrix $a\in \sqrt{-1} {\bf su}(n, \C)$ can be diagonalized by a conjugation with an element of ${\bf SU}(n)$.  The diagonal matrix is uniquely determined up to a permutation of its (real) eigenvalues, so  we can normalize it by ordering the eigenvalues in  decreasing order.  A hermitian matrix $a$ determines a regular orbit  for  the action of $G={\bf SU}(n)$ by conjugation precisely when  $a$ has simple spectrum.

Each orbit $M$ admits a natural symplectic form $\omega$,  called the {\it Kirillov--Kostant--Souriau form}, whose  definition we now recall.  As the tangent space $T_yM$  for any $y \in M$ is generated by the fundamental vector fields for the adjoint action of $G$, $T_yM$  can be identified with the image $\mathfrak{m}_y \subset \mathfrak{g}$ of the map
\begin{equation}\label{tangent-map}
u \to \frac{d}{dt}_{|_{t=0}} {\rm Ad}_{\exp(tu)}(y) = -{\rm ad}_y(u).
\end{equation}
For $u,v \in \mathfrak{g}$, we denote by
$$\langle u, v \rangle := {\rm tr} (ad_u \circ ad_v )$$
the Killing form of $\mathfrak{g}$. By assumption $\langle \cdot, \cdot \rangle$ is negative definite ($G$ being compact and semi-simple). 
We then set 
$$\omega_y(u,v) := -\langle y, [u,v]\rangle,$$
where $u,v \in T_y M \cong {\rm Im}({\rm ad}_y)$. Then $\omega$  gives rise to a symplectic structure on $M= G/G_x$.
The main fact is
\begin{prop}[\bf Kirillov--Kostant--Souriau]\label{Souriau} The 2-form  $\omega$ defines a symplectic structure on $M= G/G_x$, the adjoint action of $G$ on $(M, \omega)$ is hamiltonian with momentum  map
$\mu : M \to \mathfrak{g^*}$ identified with the inclusion $\imath: M \subset \mathfrak{g}$ composed with the Killing form $\langle \cdot, \cdot \rangle : \mathfrak{g} \to \mathfrak{g}^*.$
\end{prop}
\end{ex}}

\section{Hamiltonian actions of tori} We now specialize to the case when the Lie group  $G=\T^k$ is a $k$-dimensional torus.
\begin{lemma} \label{abelian-hamiltonian-action} Suppose $G=\T^k$ acts symplectically on $(M, \omega)$. Then,  the action is hamiltonian if and only if for any $\xi \in \mathfrak{g}$, there exists a smooth function $\mu_{\xi}$ on $M$,  such that  $\imath_{X_{\xi}} \omega = - d\mu_{\xi}$. In this case, the momentum map $\mu: M \to \mathfrak{g}^*$ is determined up to the addition of a vector in $\mathfrak{g}^*$.
\end{lemma}
\begin{proof} If the action of $G$ is hamiltonian, the existence of $\mu_{\xi}$ follows from Remark~\ref{r:1}(a). 

Conversely, suppose that each fundamental vector field $X_{\xi}$ is hamiltonian with respect to a smooth function $\mu_{\xi}$ on $M$. Choose a basis $\{\xi_1, \ldots, \xi_k\}$ of $\mathfrak{g}$ and let $X_1, \ldots, X_k$ and $\mu_1, \ldots, \mu_{k}$ be the corresponding fundamental vector fields  and hamiltonian function, respectively. We claim that for $\xi = \sum_{i=1}^k a_i \xi_i$, the function 
\begin{equation}\label{mu-xi-abelian}
\mu_{\xi}:= \sum_{i=1}^k a_i \mu_i
\end{equation}
is a hamiltonian of the induced fundamental vector field $X_{\xi}$. Indeed, since $G$ is abelian, $\exp(t\xi)= \exp(ta_1\xi_1) \circ \cdots \circ \exp(ta_k\xi_k)$ and,  therefore, the induced fundamental vector field  is $X_{\xi}= \sum_{i=1}^k a_i X_{\xi_i}$. The claim then follows trivially. Thus,  we can define $\mu : M \to \mathfrak{g}^*$ by letting
$$\langle \mu(p), \xi \rangle := \mu_{\xi}.$$

It remains to show that the condition \eqref{equivariant-bracket} holds. As the co-adjoint action of an abelian group $G$ is trivial,  this condition reduces to show that  $\mu_{\xi}$ is an invariant function under the action of $G$. Equivalently, by Remark~\ref{r:1}(b), we have to show that for any fundamental vector fields $X_{\xi}, X_{\eta}$, the smooth function $\{\mu_{\xi}, \mu_{\eta}\}_{\omega} =\omega(X_{\xi}, X_{\eta})$ identically vanishes on $M$. Since the action of $G$ is symplectic, for any fundamental vector field $X_{\zeta}, $ $\mathcal{L}_{X_{\zeta}}\omega =0$. It follows (by using that  $G$ is abelian again) that $\mathcal{L}_{X_{\zeta}} (\omega(X_{\xi}, X_{\eta}))=0$. Thus,
$\omega(X_{\xi}, X_{\eta})$ is a constant function on each orbit $\mathcal{O}\subset M$ for the action of $G$ on $M$. It is a standard fact that $\mathcal{O}$ is  a homogeneous manifold $G/G_{p},$ where $p\in \mathcal{O}$ and $G_p$ is the stabilizer of $p$ in $G$, see e.g. \cite{bredon}.  As $G$ is compact, $\mathcal{O}$ is a compact manifold (see \cite{bredon}, Chapter 1, Corollary 1.3), and therefore the restriction of $\mu_{\xi}$  to $\mathcal{O}$ has a critical point. At this point, $\omega(X_{\xi}, X_{\eta})=-d\mu_{\xi}(X_{\eta})=0$, thus showing that the function $\omega(X_{\xi}, X_{\eta})$ is identically zero on $\mathcal{O}$, and hence on $M$. 

The last statement trivially follows as we can  take instead of  $\mu_i$  the hamiltonian  $\mu_i + \lambda_i$   with $\lambda_i \in \R$ or, equivalently,  we can replace in \eqref{mu-xi-abelian} $\mu_{\xi}$ by $\mu_{\xi} + \langle \xi,  \lambda\rangle$ for $\lambda \in \mathfrak{g}^*$. \end{proof}

For the next result, we use the facts that if a compact Lie group $G$ acts effectively on $M$,  then each orbit $\mathcal{O}=G(p)$ is a compact homogeneous manifold of dimension $\le \dim G$ and, furthermore,  that there exist an open dense subset $M^0 \subset M$ of points whose orbits are of dimension $=\dim G$ (called principal orbits of $G$). We refer to \cite{bredon}  for  a proof of these facts.
\begin{lemma}\label{l:toric} Suppose  that $(M^{2m}, \omega)$ admits an effective hamiltonian action of $G=\T^k$. Then $k \le m$.
\end{lemma}
\begin{proof} By the argument in the proof of Lemma~\ref{abelian-hamiltonian-action}, the tangent space $T_p\mathcal{O}$ of a principal orbit $\mathcal{O}= G(p)$ is $\omega_p$ isotropic subspace of $T_pM$. Its dimension ($=\dim G$), therefore, is $\le m$. \end{proof} 

{\tiny \begin{ex}\label{flat-hamiltonian-action} We return again to Example~\ref{flat-action} and will now show that the symplectic action of $\T^m$ on $(\R^{2m}, \omega_0)$ is hamiltonian. We shall work on the dense open subset  $(\C^m)^0=\{(z_1, \ldots, z_m) : z_i\neq 0, \ \forall i=1, \ldots m\}$ of $\C^m \cong \R^{2m}$ on which we introduce polar coordinates $z_i = r_i e^{\sqrt{-1}\varphi_i}, i=1, \ldots, m$. The symplectic $2$-form $\omega_0$ then becomes
$$\omega_0 = \sum_{i=1}^m r_i dr_i\wedge d\varphi_i,$$
whereas the fundamental vector fields  $\{X_1, \ldots, X_{m}\}$ associated to the standard basis of $\R^m = \{(t_1, \ldots, t_m)\} \cong {\rm Lie}(\T^m)$ are
$X_i = \frac{\partial}{\partial \varphi_i}.$
It  follows that $\imath_{X_i} \omega_0 = -r_i dr_i = -d(\frac{1}{2}r_i^2).$ 
In other words,  the smooth map $\mu: \C^m \to \R^m \cong \Big({\rm Lie}(\T^m)\Big)^*$ defined by
\begin{equation}\label{flat-momentum}
\mu(z):= \frac{1}{2}\big(|z_1|^2, \ldots, |z_m|^2\big)
\end{equation}
is a momentum map for the action of $\T^m$ on $(\C^m)^0$,  and hence on $\C^m$ (by continuity). We also notice that
\begin{equation}\label{flat-polytope} 
{\rm Im}(\mu)=C_m:=\{(x_1, \ldots, x_m) \in \R^m : x_i\ge 0\}.
\end{equation}
\end{ex}}
{\tiny
\begin{exercise} Consider the $\Sph^1$-action on $(\Sph^2, \omega_{\Sph^2})$ corresponding to the rotation around the $z$-axis of $\R^3$, see Example~\ref{e:S2}.  Show that  this is hamiltonian with momentum map given by the $z$-ccordinate.
\end{exercise} }
\bigskip
The central result in the theory is the following convexity result
\begin{thm}[\bf Atiyah~\cite{Atiyah}, Guillemin--Stenberg~\cite{GS}]\label{convexity} Suppose $\T^k$ acts in a hamiltonian way on a  connected compact symplectic manifold $(M, \omega)$, with momentum map $\mu : M \mapsto \R^k$. Then,  
\begin{enumerate}
\item[\rm (i)]  The image of $\mu$ is  the convex hull $\Delta \subset \R^k$ of the images of the fixed points for the $\T^k$-action on $M$.
\item[\rm (ii)] The pre-image of any point of $\Delta$ is connected.
\end{enumerate}
\end{thm}

{\tiny
\begin{ex}[{\bf Sch\"ur's Theorem}]\label{adjoint-orbit}  This is the main application of Theorem~\ref{convexity} given  in \cite{Atiyah}.  In the general setting of Example~\ref{e:kirilov}, let  us take $G={\bf SU}(n)$ and  consider the orbit $M= M_{a}$ of all  conjugated hermitian  matrices  to   a given hermitian matrix $a$ by elements of $G$. We thus have a hamiltonian action on $(M, \omega)$ of the $(n-1)$-dimensional torus $\T\subset {\bf SU}(n)$ of diagonal matrices in ${\bf SU}(n)$. Furthermore, it easily  follows by Proposition~\ref{Souriau}  that the momentum map for this action is 
 $$\mu(b) = {\rm diag}(b_{11}, \ldots, b_{nn}), $$
 where $b$ is any hermitian matrix in $M_a$ and $(b_{11}, \ldots, b_{nn})$ are the diagonal elements  of $b$ (a vector in $\R^n$). Furthermore, the fixed points for the action of $\T$ on $M_a$ are precisely the diagonal hermitian matrices in $M_a$. We thus obtain from Theorem~\ref{convexity} the following classical result due to Sch\"ur.
 \begin{cor}  Let $a$ be a hermitian matrix in $\C^n$ with spectrum $(\lambda_1, \ldots, \lambda_n) \in \R^n$. Then the diagonal elements of the hermitian matrices in the conjugacy class of $a$ by elements of ${\bf U}(n)$ consisting of all points in the convex hull  in $\R^n$ of $\{(\lambda_{\sigma(1)}, \ldots, \lambda_{\sigma(n)}), \sigma \in S_n\},$ where $S_n$ denotes the symmetric group of permutation of $n$-elements.
 \end{cor}
\end{ex}}

\subsection{Toric symplectic manifolds and Delzant Theorem} In view of Lemma~\ref{l:toric}, we give the following
\begin{defn}\label{d:toric} A {\it symplectic toric manifold}  is a compact connected symplectic manifold $(M^{2m}, \omega)$ endowed with an effective hamiltonian action $\rho$ of a torus $\T$ with 
$${\rm dim}\  \T = m= \frac{1}{2}{\rm dim}\ M.$$
Two toric symplectic manifolds $(M_i, \omega_i, \T_i, \rho_i), i=1,2$ are  considered  {\it equivalent} if there exist an isomorphism of Lie groups $\phi : \T_1 \to \T_2$ and a diffeomorphism $\Phi: M_1 \to M_2$ with $\Phi^* \omega_2 = \omega_1$, satisfying
$$\Phi(\rho_1 (g)(p))= \rho_2(\phi(g)) (\Phi(p)), \ \forall g\in \T_1, \ \forall \ p\in M_1.$$
By Lemma~\ref{abelian-hamiltonian-action}, we can always assume that in this case the momentum maps $\mu_i$ of $(M_i, \omega_i, \T_i, \rho_i)$ are linked by $\mu_1 = \mu_2 \circ \Phi$.
\end{defn}
By virtue of Theorem~\ref{convexity}, the image of a toric symplectic manifold of dimension $2m$ is a compact convex polytope $\Delta$ in the $m$-dimensional vector space $\mathfrak{t}^*= {\rm Lie}(\T)$. The Delzant theorem provides a classification (up to the  equivalence of Definition~\ref{d:toric}) of symplectic toric manifolds  $(M, \omega, \T)$ in terms of the corresponding polytopes $\Delta \subset \mathfrak{t}^*$. To state it properly, we notice that being Lie algebra of a torus $\T$, the vector space $\mathfrak{t}$  comes equipped with a lattice $\Lambda\subset \mathfrak{t}$ such that $2\pi \Lambda= {\rm  exp}^{-1}(e)$ (where $e$ stands for the identity element of $\T$). In other words,
$$\exp: {\mathfrak{t}}/{2\pi \Lambda} \cong \T.$$
\begin{defn}\label{d:Delzant-polytope} Let $V$ be an $m$-dimensional real vector  space whose dual space is $V^*$, and let  $\Lambda \subset V$ be a lattice. We denote also by $V^*$ the affine space determined by $V^*$. Let $\Delta \subset V^*$ be a  convex polytope written as the intersection of a minimal  number of $d$ linear inequalities 
$$\Delta =\{x \in V^* : L_j(x) =\langle u_j, x\rangle + \lambda_j \ge 0, j=1, \ldots, d\}, $$
where $u_j \in V$  are called (labelled) {\it normals} of $\Delta$ and $\lambda_j \in \R$. We shall refer to the collection ${\bf L}=\{L^1, \ldots, L_d\}$ of affine-linear functions defining $\Delta$ as a {\it labelling} of $\Delta$, and to the couple $(\Delta, {\bf L})$ as a {\it labelled} polytope in $V$.  We say that the data $(\Delta, {\bf L}, \Lambda)$  define a {\it  Delzant polytope} if the following conditions are satisfied:
\begin{enumerate}
\item[\rm (i)] $\Delta$ is {\it compact};
\item[\rm (ii)] $\Delta$ is {\it simple}, meaning that each vertex $x_0$ of $\Delta$ annihilates precisely $m$ of the affine functions in ${\bf L}$ and the corresponding normals form a basis of $V$;
\item[\rm (iii)] $\Delta$ is {\it integral}, meaning that for each vertex $x_0$ of $\Delta$, $${\rm span}_{\Z} \{u_j \in V : L_j(x_0)= \langle u_j, x_0 \rangle + \lambda_j=0\}=\Lambda.$$
\end{enumerate}
We say that $(\Delta, {\bf L}, \Lambda)$ defines  a {\it rational Delzant polytope} if instead of (iii) we require the weaker assumption 
\begin{enumerate}
\item[\rm (iii)'] $\Delta$ is  {\it rational}, meaning that for each $L_j \in {\bf L}$, the normal $u_j\in \Lambda$.
\end{enumerate}
\end{defn}
{\tiny
\begin{exercise}~\label{ex:Delzant-simplified}Show that if the triple $(\Delta, {\bf L}, \Lambda)$ is a Delzant polytope, then $(\Delta, \Lambda)$ determine the labelling ${\bf L}$. Is this true for a rational Delzant polytope $(\Delta, {\bf L}, \Lambda)$?
\end{exercise}}
\begin{thm}[\bf Delzant~\cite{Delzant}]\label{thm:delzant} There exists a bijective correspondence between the equivalence classes of 2m-dimensional toric symplectic manifolds and equivalent Delzant polytopes $(\Delta, {\bf L}, \Lambda)$ in an $m$-dimensional vector space $V^*$, up to the natural action of the affine group  ${\rm Aff}(V^*)$ on the  triples $(\Delta, {\bf L}, \Lambda)$. 
\end{thm}
We shall not develop in these notes a  detailed proof of this key result, but in the next section we are going to sketch one direction of it, namely the Delzant construction which associates to a Delzant triple $(\Delta, {\bf L}, \Lambda)$ a toric symplectic manifold $(M, \omega, \T)$.

\section{The complex projective space}\label{s:CPm}
We start with the following basic example of a toric symplectic manifold.

Recall that the (complex) $m$-dimensional projective space
$$\C P^m= \C^{m+1}\setminus\{0\} /\C^*$$ is the quotient of $\C^{m+1}\setminus\{0\}$ by the diagonal  (holomorphic) action of $\C^{*}$ given by 
$$(\lambda, z) \to \lambda z, \  \lambda \in \C^*,  \ z\in \C^{m+1}\setminus\{0\}.$$
This description naturally realizes $M=\C P^m$ as a complex manifold, by passing from the  homogenous coordinates $[z_0, \ldots, z_m]$ (where $(z_0, \ldots, z_m)\in \C^{m+1}\setminus\{0\}$ and $[z_0, \ldots, z_m]$ stands for the equivalence under the diagonal action of $\C^{*}$) to the affine charts $U_i\cong \C^m$ on the open subsets of $\C P^m$ where $z_i\neq 0.$  In order to describe the Fubini--Study metric on $\C P^m$, or its symplectic form, it is more convenient to consider the identification
$$\C P^m = {\mathbb S}^{2m+1}/{\mathbb S}^1, $$
where ${\mathbb S}^{2m+1}\subset \C^{m+1}$ is the unit sphere with respect to the standard euclidean product of $\R^{2m+2}\cong \C^{m+1},$ and ${\mathbb S}^1$ denotes the diagonal circle action on ${\mathbb S}^{2m+1}\subset \C^{m+1}$ by
\begin{equation}\label{S1}
\rho(e^{\sqrt{-1}t})(z_0, \ldots, z_m) = e^{\sqrt{-1}t}(z_0, \ldots, z_m).
\end{equation}

\begin{defn}\label{almost-Kahler} Let $(M, \omega)$ be a symplectic manifold. An $\omega$-{\it compatible} riemannian metric $g$  on $M$ is a riemannian metric such that the field of endomorphisms $J$ defined by
$$g_p(J_p u, v)= \omega_p(u,v), \ \ \forall p\in M, \forall u,v \in T_pM, $$
is an {\it almost-complex structure} on $M$, i.e. satisfies 
$$J_p^2 = - {\rm Id}_{|_{T_pM}}.$$
We say that $(g, J)$ is an $\omega$-compatible  {\it almost-K\"ahler structure} on $M$.

If, furthermore, $J$ is an {\it integrable} almost-complex structure, i.e. satisfies
\begin{equation}\label{integrability}
N^J(X, Y) = [X,Y] -[JX, JY] +  J[JX, Y] + J[X, JY]=0,
\end{equation}
then $(g, J)$ defines a compatible {\it K\"ahler structure} on $(M, \omega)$.
\end{defn}

\begin{rem}\label{r:AK} (a) It is well-known (see e.g. \cite{MS}), and easy to show, that any symplectic manifold admits infinitely many compatible almost-K\"ahler structures. Actually, the space of such structures, denoted ${\mathcal AK}(\omega),$  is a contractible Fr\'echet manifold.  The group ${\rm Symp}(M,\omega)$ of symplectomorphisms of $M$ naturally acts on $\mathcal{AK}(\omega)$ by pull-backs of the riemannian metrics $g$. It thus follows that for any compact subgroup $G\subset {\rm Symp}(M,\omega)$ one can take an average over $G$, and thus produce a $G$-invariant $\omega$-compatible almost-K\"ahler structure $(g, J)$ on $(M,\omega)$.

(b)  By a result of Newlander--Nirenberg~\cite{NN},  the integrability condition \eqref{integrability} is equivalent to the existence of a holomorphic atlas on $M$,  compatible with the almost-complex structure $J$, i.e.  to $(M,J)$ being a complex manifold. There are many examples of symplectic manifolds which do not admit compatible K\"ahler structures.
\end{rem}

In our situation, the standard flat metric of $\C^{m+1} = \R^{2m+2}$
\begin{equation}\label{flat-riemannian}
g_0 = \sum_{i=0}^{m} (dx_i^2 + dy_i^2)
\end{equation}
is $\omega_0$-compatible and the corresponding almost complex structure $J_0$ is simply the standard complex structure on $\C^{m+1}$. As both $\omega_0$ and $g_0$ are preserved by the $\T^{m+1}$-action,  we  thus have a  $\T^{m+1}$-invariant flat K\"ahler structure $(g_0, J_0)$ on $\C^{m+1}$.
The restriction of the flat metric $g_0$ on $\C^{m+1}$ to ${\mathbb S}^{2m+1}$  induces the canonical {\it round} metric $g^{{\mathbb S}^{2m+1}}$ and  the projection map  (known as {\it the Hopf fibration})
$$\pi : {\mathbb S}^{2m+1} \to \C P^m,$$
provides an example of a {\it riemannian submersion} between $({\mathbb S}^{2m+1}, g^{{\mathbb S}^{2m+1}})$ and  the (uniquely determined) riemannian metric $g^{FS}$ on $\C P^m,$  introduced  by the property
$$\pi^*\Big(g^{FS}_{\pi(z)}\Big) = g_z^{{\mathbb S}^{2m+1}}{|_{H_z}},$$
where at each point $z\in {\mathbb S}^{2m+1}$, $H_z= \Big({\rm Ker}(\pi_*)_z\Big)^{\perp}$. 
We use a similar construction in order to endow $\C P^{m}$ with  a symplectic  form $\omega$, using  the pointwise isomorphism $(\pi_*)_z : H_z \cong T_{\pi(z)}\C P^m$ and the restriction of $\omega_0$ to the subspace $H_z$.   As all tensor fields used in this construction are $\T^{m+1}$-invariant, and the ${\mathbb S}^1$-action \eqref{S1} is obtained by restricting the action of $\T^{m+1}$ to its ${\mathbb S}^1$ subgroup $N=(e^{\sqrt{-1}t}, \ldots, e^{\sqrt{-1}t})\subset \T^{m+1}$, there is a natural induced $\T^m = \T^{m+1}/N$-action on $\C P^m = {\mathbb S}^{2m+1}/N$. With respect to this data,  $(\C P^m, \omega^{FS}, \T^m)$ is an example of a symplectic toric manifold. In fact, this is just a very particular case of a general fact in symplectic geometry, known as {\it symplectic reduction}, which we shall present below. 

To this end,  let us summarize the construction of $\C P^m$: by Example~\ref{flat-hamiltonian-action}, $(\C^{m+1}, \omega_0)$  is a symplectic manifold endowed with the hamiltonian action on $\T^{m+1}$ with momentum map
$$\mu_{\T^{m+1}}(z)=\frac{1}{2}(|z_0|^2, \ldots, |z_m|^2).$$
$N \subset \T^{m+1}$ is a circle subgroup and the momentum map for the action of $N$ is  therefore
$$\mu_N (z)= \frac{1}{2}\sum_{i=0}^m |z_i|^2.$$
Thus,  $${\mathbb S}^{2m+1} = \mu_N^{-1}\Big(\frac{1}{2}\Big)$$ is the level set of the momentum map $\mu_N$, and furthermore,  $N$ acts freely  on this level set. We are thus in the situation of the following fundamental result.

\begin{prop}[\bf Marsden--Weinstein~\cite{MW}, Meyers~\cite{M}]\label{kahler-reduction} Let $N\subset G$ be a closed normal  subgroup of a compact group $G$ acting in a hamiltonian way on a symplectic manifold $(\tilde M, \tilde \omega)$.  Let $ {\mu}_G : \tilde M \to \mathfrak{g}^*$ be  a momentum map for the $G$-action,  and 
$ {\mu}_N : \tilde M \to \mathfrak{n}^*$  the corresponding momentum map for  the $N$-action,  obtained from ${\mu}_G$ by composing with the natural projection $\imath^* : \mathfrak{g}^* \to \mathfrak{n}^*$ (which is  adjoint  to the inclusion of the Lie algebras $\imath: \mathfrak{n}\subset \mathfrak{g}$). 

Suppose further that $\tilde c \in \mathfrak{g}^*$ is a fixed point for the co-adjoint action of $G$ on $\mathfrak{g}^*,$ such that $\imath^*(\tilde c)=c$ is a regular point for $\mu_N$, and that the action of $N$ on $\mu_N^{-1}(c)$ is free.  

Then,  the $N$-invariant symplectic $2$-form $\tilde \omega,$ restricted to $\mu_N^{-1}(c)$ defines a symplectic form on the manifold $M:= \Big(\mu_N^{-1}(c)\Big)/N$, and the natural action of $G/N$ on $M$ is hamiltonian with moment map, viewed as an $N$-invariant function on $ \mu^{-1}_N(c)$, given by
$$\mu=  \mu_G - \tilde c.$$
\end{prop}
\begin{proof}  By construction, $S := {\mu}_{N}^{-1}(c) \subset \tilde M$ is a closed submanifold of $\tilde M$ on which $G$ acts smoothly. 
To see that $G$ acts on $S$, notice that, by construction,  $p\in S$ iff 
\begin{equation}\label{definition-S}
\langle \mu_G (p) - \tilde c, \xi \rangle =0, \ \forall \xi \in \mathfrak{n}.
\end{equation}
Using the equivariance of the momentum map, we have for any $g\in G$
\begin{equation*}
\begin{split}
\big\langle (\mu_G(g\cdot p) -\tilde c), \xi \big\rangle &= \big\langle {\rm Ad}^*_g(\mu_G(p))-\tilde c, \xi\big\rangle \\
                                                                            &=\big\langle {\rm Ad}^*_g(\mu_G(p)-\tilde c),  \xi \big\rangle \\
                                                                            &= \big\langle (\mu_G(p) -\tilde c), {\rm Ad}_{g^{-1}}(\xi) \big\rangle \\
                                                                            &= \big\langle (\mu_G(p) -\tilde c), \xi \big\rangle,
                                                                            \end{split}
                                                                            \end{equation*}
 where for passing from the second line to the third we have used that $\tilde c$ is a fixed point for the co-adjoint action, and for passing from the third line to the forth the fact that $N$ is a normal subgroup of $G$.

We denote by 
$$i: S \hookrightarrow \tilde M$$ the inclusion map  and by 
$$\pi : S \to M$$
the projection  map, which is a smooth submersion by the regularity assumption for $c$.  It follows from \eqref{definition-S}
that  ${\mu}_G - \tilde c$ is a smooth function on $S$ with values in the annihilator ${\rm Ann}(\mathfrak{n}) \subset \mathfrak{g}^*$. The latter subspace is canonically identified with $\big(\mathfrak{g}/\mathfrak{n}\big)^*$. We  summarize the construction in the following diagram:
\begin{equation} \label{diag} 
\begin{CD}
\tilde{M} @>{\mu}_G>> \mathfrak{g} ^*\\
@AAiA  @| \\
S @> ({\mu}_G)_{|_S}>>
 \mathfrak{g} ^*\\
@VV\pi V    @AA\ell := j +  \tilde{c} A\\
M @>\mu>>  (\mathfrak{g}/\mathfrak{n})^*
\end{CD}
\end{equation} 
where  $j$ denotes the natural inclusion of 
$(\mathfrak{g}/\mathfrak{n}) ^*$ in  $\mathfrak{g} ^*$,  
and  $\ell = j + \tilde{c}$ denotes the {\it affine} map from 
$(\mathfrak{g}/\mathfrak{n}) ^*$ to $\mathfrak{g} ^*$ obtained by 
composing  $j$ with the translation by $\tilde{c}$, i.e. $\ell (x) = j
(x) + \tilde{c}$, for any $x$ in  $(\mathfrak{g}/\mathfrak{n}) ^*$.

As $\pi : S \to M$ is a smooth submersion, at each $p\in S$ we let $V_p:= {\rm Ker}(\pi_*)_p$ be the vertical distribution. 
By Remark~\ref{r:AK}-(a), we can introduce an $\tilde \omega$-compatible, $G$-invariant riemannian metric $\tilde g$ on $\tilde M$, and use it  to define the horizontal space
$$H_p := V_p^{\perp_{\tilde g}} \subset T_p S. $$
As $\tilde g$ is $G$-invariant, $\pi : (S, \tilde g) \to M$  defines a riemannian submersion, thus giving rise to  a  riemannian quotient metric $g$ on $M$ with the property
$$(\pi^* g)_{|_H} = \tilde g_{|_H}.$$
Similarly, there exists  a $2$-form $\omega$ on $M$, determined by the property
$$(\pi^*\omega)_{|_H} = \tilde \omega_{|_H}.$$
We claim that $\omega$ is non-degenerate and closed.  

To see the non-degeneracy of $\omega$, we observe that $g$ is an $\omega$-compatible riemannian metric.  As $\tilde g$ is $\tilde \omega$-compatible on $\tilde M$, it is enough to show that the corresponding almost complex structure $\tilde J$ preserves the subspace $H_p \subset T_pS \subset T_p\tilde M$.  This will follow if we show that for each $p \in S$,
$$T_p \tilde M = V_p \oplus \tilde JV_p  \oplus H_p$$
is a $\tilde g_p$-orthogonal decomposition. For any $\xi \in \mathfrak{n}$,  the corresponding fundamental vector field $X_{\xi}(p)$  belongs  to $V_p$ and, furthermore, we have 
\begin{equation}\label{orthogonality}
(\imath_{X_{\xi}} \tilde \omega)|_{T_pS} = -d\big(\langle \mu_G- \tilde c, \xi \rangle_{|_S}\big) =0.
\end{equation}
Notice that, as the action of $N$ on $S$ is free, the fundamental vector fields $\{X_{\xi}(p), \xi \in \mathfrak{n}\}$ span $V_p$. Thus, \eqref{orthogonality} reads as $\tilde JV_p \perp_{\tilde g} (V_p \oplus H_p)$. As $V_p \perp_{\tilde g} H_p$ by construction, the claim follows.

We now discuss the closedness of $\omega$. For any vector field $X$ on $M$,  we denote by $\tilde X$ the corresponding horizontal lift to $S$, i.e. the unique section of $H\subset TS$ such that $\pi_*(\tilde X)=X$. For any vector fields $X, Y$ on $M$, we thus have the orthogonal decomposition
\begin{equation}\label{submersion}
\widetilde{[X,Y]}_p = [\tilde X, \tilde Y]_p + \Pi_p([\tilde X, \tilde Y]_p),
\end{equation}  where $\Pi_p( \cdot )$ denotes the orthogonal projection of  $T_pS$ to $V_p$. Using Cartan's formula and the fact that $V_p \subset {\rm Ker}(\tilde \omega_p)$, we compute
\begin{equation*}
\begin{split}
d\omega(X, Y, Z) &= \sigma_{X,Y, Z}  \Big(d(\omega(Y, Z))(X) - \omega([X, Y], Z)\Big)\\
                                &= \sigma_{\tilde X, \tilde Y, \tilde Z} \Big(d(\tilde\omega (\tilde Y, \tilde Z))(\tilde X) - {\tilde \omega}(\widetilde{[X, Y]}, \tilde Z)\Big)\\
                                &=  \sigma_{\tilde X, \tilde Y, \tilde Z} \Big(d(\tilde\omega(\tilde Y, \tilde Z))(\tilde X) - \tilde \omega([\tilde X, \tilde Y], \tilde Z)\Big)\\
                                &= d\tilde \omega(\tilde X, \tilde Y, \tilde Z) = 0.
                                \end{split}
                                \end{equation*}

We finally discuss the induced action of $G/N$ on $M$. The action of $[g]\in G/N$ on a point $\pi(p)\in S/N$ is defined by $\pi(g\cdot p)$. (This is  well-defined as $S$ is invariant under the action of $G$, as we have already shown.)  We have also observed that when restricted to $S$, the function  $\nu:=\mu_G - \tilde c$ takes values in the annihilator $\mathfrak{n}_0 \cong (\mathfrak{g}/\mathfrak{n})^*$ of $\mathfrak n$ in $\mathfrak{g}^*$. Using that $N$ is normal, and that $\mu_G$ is $G$-equivariant,  it follows that $\mu_G$ is $N$-invariant. It then descends to define a smooth,  $G/N$ equivariant
 function $\mu: M \to (\mathfrak{g}/\mathfrak{n})^*$ on $M=S/N$. By construction, for any $[\xi] \in \mathfrak{g}/\mathfrak{n}$, 
\begin{equation*}
\begin{split}
d \langle \mu, [\xi]\rangle_{\pi(p)} &=  d\langle \nu, \xi \rangle_p \\
                                                  &= d \langle \mu_G, \xi \rangle_p \\
                                                  &= - (\imath_{X_{\xi}(p)} \tilde \omega)_{|_{T_pS}}\\
                                                  &= -(\imath_{X_{\xi}(p)} \tilde \omega)_{|_{H_p}} \\
                                                  &= -(\imath_{X_{[\xi]}}\omega)_{\pi(p)},
                                                  \end{split}
                                                  \end{equation*}
where we have used that $V_p \subset {\rm Ker}(\tilde \omega_p)$ and $\pi_*(X_{\xi})= X_{[\xi]}$. \end{proof}

{\tiny \begin{exercise} (1) Show that the induced $2$-form $\omega$ on $M$  in the proof of Proposition~\ref{kahler-reduction} is independent of the choice of  a $G$-invariant $\tilde \omega$-compatible riemannian metric on $\tilde M$.

(2) Show that if $\tilde g$ is a $G$-invariant,  $\tilde \omega$-compatible  riemannian metric on $\tilde M$, which defines a K\"ahler structure on $\tilde M$, then the  almost-K\"ahler structure $(g, \omega, J)$ on $M$  defined in the proof of Proposition~\ref{kahler-reduction} is K\"ahler.

(3) Show that the complex structure  $J$ induced on $M=\mathbb{S}^{2m+1}/\mathbb{S}^1$ makes $(M, J)$ biholomorphic to $\C P^m = (\C \setminus\{0\})/\C^*$ endowed with its atlas of affine charts.
 \end{exercise}}

\bigskip
\noindent
Now we can apply Proposition~\ref{kahler-reduction} in order to conclude that $(\C P^m, \omega^{FS})$ is a toric symplectic manifold under the action of the torus $\T = \T^{m+1}/{\mathbb S}^1,$ and that the induced Fubini--Study metric $g^{FS}$ gives rise to a $\T$-invariant K\"ahler structure on $(\C P^m, \omega^{FS})$. Furthermore,  the image of the induced momentum map
$$\mu : \C P^m \to  (\mathfrak{t}^{m+1}/\mathfrak{t}^1)^* \subset (\mathfrak{t}^{m+1})^*$$
is identified with the intersection of ${\rm Im}(\mu_{\T^{m+1}}) = C_{m+1}= \{x_i \ge 0, i=0, \ldots,  m\}\subset \R^{m+1}$ with the hyperplane  $x_0 + \cdots + x_m =\frac{1}{2}$, which is a simplex in this hyperplane. Alternatively, we can consider the subtorus  $\T^{m} \subset \T^{m+1}$ defined by
$$\T^{m} = (1, e^{\sqrt{-1}t_1}, \cdots, e^{\sqrt{-1}t_m})$$ so that we have an exact sequence of Lie groups 
\begin{equation*} 
\begin{CD}
\{e\} @>>> {\mathbb T}^m @>>> \T^{m+1} @>>>  \T^{m+1}/{\mathbb S}^1 @>>>\{e\},
\end{CD}
\end{equation*}
giving rise to an  isomorphism between $\T=\T^{m+1}/{\mathbb S}^1$ and $\T^m$. Using the induced projection of the dual Lie algebras  $(\mathfrak{t}^{m+1})^* \to (\mathfrak{t}^m)^* \cong \R^m$,  the moment map $\mu$ sends $\C P^m$ onto the simplex 
$$\Delta_m=\Big\{(x_1, \ldots, x_m):  \Big(\frac{1}{2}-\sum_{i=1}^m x_i\Big) \ge 0, x_i\ge 0, i=1,\ldots, m,\Big\}.$$  Finally, the lattice $\Lambda$ of the dual space $\mathfrak{t}^m$ is just the standard lattice $\Z^m \subset \R^m \cong (\R^m)^*$ and  we thus conclude
\begin{lemma} $(\C P^m, \omega^{FS})$ is a toric symplectic manifold under the induced action of $\T^m$, classified by the standard simplex in $\Delta_m \subset \R^m,$ labelled by $${\bf L}=\Big\{L^1(x)=x_1, \ldots, L_m(x)= x_m, L_{m+1}(x)=\Big(\frac{1}{2} - \sum_{i=1}^m x_i\Big)\Big\},$$ and the standard lattice $\Z^m \subset  (\R^m)^*$.
\end{lemma}

\section{Toric symplectic manifolds from Delzant polytopes}\label{s:Delzant-construction} The discussion in the previous subsection is the main tool of the explicit construction,  proposed by Delzant in \cite{Delzant},  which associates a toric symplectic manifold $(M_{\Delta}, \omega, \T)$ to any Delzant triple $(\Delta, {\bf L}, \Lambda)$, as in Definition~\ref{d:Delzant-polytope}.  We first notice that by the condition (iii) of Definition~\ref{d:Delzant-polytope},  the lattice $\Lambda$ is determined by the labelling ${\bf L}=\{L^1(x), \ldots, L_d(x)\}$. Indeed, $\Lambda$  is the span over $\Z$ of the linearizations $dL_i := u_i \in V^*$ of  the affine-linear functions $L_j(x)= \langle u_j, x\rangle + \lambda_i  \in {\bf L}$.  We denote by $\T = V/2\pi \Lambda$ the corresponding torus,  and by $\mathfrak{t}= V$ its Lie algebra. Thus,  $\Delta \subset \mathfrak{t}^*$.  With this in mind, we consider the linear map 
$\tau : \R^d \to \mathfrak{t},$ defined by
$$\tau ( \xi_1, \ldots, \xi_d) : = \sum_{j=1}^d \xi_j u_j.$$
Using the Delzant condition (i)-(ii)-(iii), one checks that

\noindent
{\bf Claim 1.}  $\tau$ sends the  standard lattice $\Z^d$  of $\R^d$ onto $\Lambda$ and thus defines a homomorphism of tori
$$\tau : \R^d/2\pi \Z^d=\T^d \to \T.$$

\noindent
{\bf Claim 2.}  The kernel $N$ of $\tau: \T^d \to \T$ is a connected subgroup of $\T^d$, i.e. it is a $(d-m)$-dimensional torus.~\footnote{This is no longer true if we consider the weaker conditions (i)-(ii)-(iii)' : then $N$ can be the product of a torus with a finite abelian group.}

\smallskip
We denote by $\mathfrak{n}$ (resp. $\mathfrak{n}^*$) the Lie algebra (resp. its dual) of $N$. We thus have an exact sequence of Lie groups
\begin{equation*} 
\begin{CD}
\{e\} @>>> N @>\imath >> \T^d @>\tau>> \T @>>>\{e\}
\end{CD}
\end{equation*}
and the corresponding exact sequence of Lie algebras
\begin{equation*} 
\begin{CD}
\{0\} @>>> \mathfrak{n} @>\imath >> \R^d @>\tau>> \mathfrak{t} @>>>\{0\},
\end{CD}
\end{equation*} 
and its dual sequence
\begin{equation} \label{sequence*}
\begin{CD}
\{0\} @>>> \mathfrak{t}^* @>\tau^* >> (\R^d)^* @>\imath^*>> \mathfrak{n}^* @>>>\{0\}.
\end{CD}
\end{equation}

We now consider the hamiltonian action of $\T^d$ on $(\C^d, \omega_0)$ with momentum map $\mu_{\T^d}(z) =\frac{1}{2}(|z_1|^2, \ldots, |z_d|^2) \in (\R^d)^*.$ We denote by $\mu_N (z) = \imath^* \circ \mu_{\T^d}(z)$ the momentum map for the action of $N\subset \T^d$. 

By acting  with a translation on $\Delta$, we can assume without loss of generality that the origin of $\mathfrak{t}^*$ is in the interior of $\Delta$. Then, letting  $\lambda:=(\lambda_1, \ldots, \lambda_d) =(L^1(0), \ldots, L_d(0)) \in(\R^d)^*,$  we have

\noindent
{\bf Claim 3.}  $S:=\mu_{N}^{-1}(\imath^*({\lambda}))$  is  a compact submanifold of $\C^d$. 

 As $\lambda_i>0$, $\lambda$ is in the interior of the momentum image  $C_d=\{(\tilde x_1, \ldots, \tilde x_d) \in (\R^d)^*: \tilde x_i\ge 0\}$  of $\C^{d}$, showing that ${\lambda}$ is a regular point of $\mu_{\T^d}$. It  follows that $\imath^*({\lambda})$ is a regular point of $\mu_{N}$. Thus,  $S$ is a closed submanifold of $\C^d$. We still need to show that $S$ is compact. As the momentum map $\mu_{\T^d} : \C^d \to (\R^d)^*$ is manifestly proper, it is enough to show that  $\mu_{\T^d}(S)$ is bounded. 

By the very definition of $S$, $z\in S$  iff $\langle (\mu_{\T^d}(z) - \lambda), \xi \rangle =0, \  \forall \xi \in \mathfrak{n}$. Thus, 
$$\mu_{\T^d}(S)  = C_d \cap \{\tilde x \in (\R^d)^* : \langle (\tilde x-{\lambda}), \xi \rangle =0, \forall \xi \in \mathfrak{n}\}. $$  
Let $\Delta'= \tau^*(\Delta)+ \lambda$ be the (compact) image of $\Delta$ under the inclusion $\tau^*$ composed with the translation $\lambda$, see \eqref{sequence*}.  We claim that $\mu_{\T^d}(S)=\Delta'$. 
Indeed,  for $\tilde x= \tau^*(x)+ \lambda$ with $x\in \Delta$, we have 
\begin{equation}\label{transpose}
\begin{split}
\tilde x_i &= \langle \tau^*(x), e_i\rangle +\lambda_i \\
        &=\langle x, \tau(e_i)\rangle +\lambda_i=\langle x, u_i\rangle + \lambda_i\\
        &= L_i(x)  \ge 0.
        \end{split}
        \end{equation}
As $\tilde x-\lambda = \tau^*(x)$, we have by the exact sequence \eqref{sequence*}, 
\begin{equation}\label{2}
\langle (\tilde x-\lambda), \xi\rangle =0, \ \forall \xi \in \mathfrak{n}. 
\end{equation}
Thus, $\Delta' \subset \mu_{\T^d}(S)$. For the other inclusion, observe that, again using \eqref{sequence*}, the equality \eqref{2} tell us that 
$\tilde x-\lambda = \tau^*(x)$ for some $x\in \mathfrak{t}^*$.  Using $\tilde x_i\ge 0$ and the computation in \eqref{transpose}, we conclude that  $x\in \Delta$.

\noindent
{\bf Claim 4.} $N$ acts freely on $S$.

We first determine the stabilizer group of  a point  $z\in S,$ under the action of $\T^d$. It is a subtorus $\T_z\subset \T^d$ of dimension equal to  the number of vanishing  coordinates $z=(z_1,..., z_d)$, or equivalently, the number of vanishing coordinates $\tilde x=(\tilde x_1, \ldots, \tilde x_d)$ where $\tilde x=\mu_{\T^d}(z) \in (\R^d)^*$ is the momentum image of $z$. By the argument in the proof of Claim 3, $\mu_{\T^d}(S)= \Delta'= \tau^*(\Delta) + \lambda$ and the number of vanishing coordinates  of $\tilde x=\tau^*(x)+\lambda$ equals  the number of vanishing labels $L_j$ at $x$. Thus, the maximum number is $m$ and it is achieved at the images of the vertices  of $\Delta$. Suppose that  $\tilde x=\mu_{\T^d}(z) \in \Delta'$ is a vertex point. Up to a coordinate permutation, it can be written as $\tilde x=(0,...0, \tilde x_{m+1}, \ldots, \tilde x_{d})$. Then,  the stabilizer of $z$ is the torus $\T_z=\T^m=(e^{\sqrt{-1}t_1}, \ldots, e^{\sqrt{-1}t_m}, 1, \ldots, 1)$.  Notice that  $\tau : \T_z \to \T$ is an isomorphism because of the condition (iii) in Definition~\ref{d:Delzant-polytope}. Thus,  
\begin{equation*}
{\rm Stab}_N(z)={\rm Stab}_{\T^d}(z)\cap N = {\rm Ker}(\tau) \cap N = \{e\} \cap N =\{e\}, 
\end{equation*}
showing that $N$ acts freely as the stabilizer of any point of $S$ is contained in the stabilizer of some vertex point. 

\smallskip
As a final step, we let  $(M, \omega)$ be the K\"ahler quotient of $(\C^d, \omega_0, g_0)$ at the  momentum value $\lambda$, associated to $G=\T^d$ and $N$, see Proposition~\ref{kahler-reduction}. The verification that the momentum image of $M$ under the induced momentum map is $\Delta$ follows the above discussion.

 \begin{cor}[\bf Delzant~\cite{Delzant}]\label{kahler-natural} Any symplectic toric manifold $(M, \omega, \T)$ admits an $\omega$-compatible $\T$-invariant K\"ahler structure $(g_0, J_0)$.
 \end{cor}
{\tiny \begin{exercise} Show that any toric symplectic manifold $(M^{2m}, \omega, \T)$ equipped with a $\T$-invariant  $\omega$-compatible K\"ahler structure $J$ is a projective variety, i.e. $(M^{2m}, J)$ admits an holomorphic embedding into a complex projective space $\C P^N$.
\end{exercise}

\noindent
{\bf Hint.} Show that  the Dolbeault cohomology $H^{2,0}(M, J)=\{0\}$ and use Hodge decomposition theorem  to conclude that $M$ admits a K\"ahler structure $(\omega', J)$ with $[\omega'] \in H^2(M, \Q)$. The conclusion then follows from the Kodaira embedding theorem, see e.g. \cite{GH}.}

\section{Toric complex varieties from  Delzant polytopes. Fans}\label{s:fan} Given a symplectic toric manifold $(M, \omega, \T)$ classified by the Delzant triple $(\Delta, {\bf L}, \Lambda)$, we can associate a complex manifold $M_{\Delta}^{\C}$ of dimension $m$ as follows: to each vertex $v \in \Delta$, we take a copy of $\C_v^m=\{(z_1^v, \ldots, z_m^v)\}$ and consider the identification  $\T \cong \T^m= (\Sph^1)^m$  with respect to the lattice base $\{u_{v_1}, \ldots, u_{v_m}\}$ given by the normals of the $v$-adjacent facets of $\Delta$; we denote this identification by $\T_v^m$ and endow the chart $\C_v^m$ with the standard action of $\T_v^m$, as in Example~\ref{flat-action}.
If $w\in \Delta$ is another vertex,  we consider the respective bases $\{u_{v_1}, \ldots, u_{v_m}\}$ and $\{u_{w_1}, \cdots, u_{w_m}\}$ of $\Lambda$ (corresponding to the normals of the $v$-adjacent  and $w$-adjacent facets of $\Delta$), and let $A=(a_{ij})\in {\bf SL}(m, \Z)$ be the coordinate transition matrix.  We then identify the subset $(\C^*)^m_v \subset \C_v^m$ with the subset $(\C^*)^m_w \subset \C_w^m$ through the identification
\begin{equation}\label{affine-chart}
z_j^w = (z_1^v)^{a_{j1}} \cdots (z_{m}^v)^{a_{jm}}, j=1, \ldots, m.
\end{equation}
It is easily seen that  \eqref{affine-chart} is an equivariant map with respect to the action of $\T_v^m$ on $\C_v^m$ and $\T_w^m$  on $\C_w^n$. This way, one obtains a complex space $M_{\Delta}^{\C}$ endowed with an affective holomorphic  $\T$-action, covered by the affine charts $\C^m_v$ (parametrized by the vertices $v$ of $\Delta$) and  identified at the intersections $\C^m_v \cap \C^m_w\cong (\C^*)^m$  as explained above.  We shal denote by $\T^{\C}$ the induced $(\C^*)^m$-action.

To construct $M_{\Delta}^{\C}$ explicitly, one can use the Geometric Invariant Theory (GIT),  see \cite{kirwan-book}:  Accordingly, $M^{\C}_{\Delta} := (\C^{d})_{ss} / N^{\C}$  is the space of orbits  for the holomorphic action of the  {\it complexified}  $(d-n)$-dimensional torus $N^{\C} \cong (\C^*)^{d-n} \subset (\C^*)^d$ (corresponding to $N \subset \T^d$) on the subset 
$(\C^{d})_{ss} \subset \C^d$ of {\it semi-stable} points for the action of $N^{\C}$ on $\C^d$,  i.e. the points such that the closure of the $N^{\C}$-orbit {\it does not} contain $0\in \C^d$. Considerations similar to the one in the proof of Claim 4 of Section~\ref{s:Delzant-construction}  lead to  the introduction of $(\C^*)^d /N^{\C} \cong (\C^*)^{m}$ equivariant {\it affine charts}  on  $(\C^{d})_{ss} / N^{\C}$, parametrized by the vertices of $\Delta$ (which in turn determine sets of $(d-n)$ non-vanishing affine coordinates of $\C^d$ on which $N^{\C}$ acts transitively).

{\tiny
\begin{exercise} Show that  if we start with the standard Delzant simplex   $\Delta_m \subset \R^m$  and the standard lattice $\Z^m \subset (\R^m)^*$ as in Section~\ref{s:CPm}, the resulting complex manifold $M^{\C}_{\Delta_m}$ constructed  as above is $\C P^m$,  endowed with  its  atlas of affine  charts.
\end{exercise}}

We notice that in order to construct $M_{\Delta}^{\C}$ we did not use the whole data from $(\Delta, {\bf L}, \Lambda)$.  The relevant information is captured by the normals $u_j = d L_j\in \tor$ adjacent to  the set of vertices, and the combinatorial object which describes it is the so called  {\it Fan} associated to $\Delta$.
\begin{defn}[\bf Fan] Let $(\Delta, {\bf L}, \Lambda)$ be a Delzant triple and $\mathcal{P}=\{ F \subset \Delta \}$ the poset of closed facets of $\Delta$, partially ordered by the inclusion. 
The fan $\mathcal{F}(\Delta, {\bf L})$ of $(\Delta, {\bf L})$ is the union of polyhedral cones $\{\mathcal{C}_{F}, F \in {\mathcal P}\}$ in $V^*$, defined by
$$\mathcal{C}_F =\{ dL : L (x) \ge 0  \ \forall x \in \Delta,  L(x)=0 \ \forall x \in F\}.$$
\end{defn}

The central  fact is

\begin{thm}[\bf Lerman--Tolman~\cite{LT}]\label{complex-toric} Suppose $J$ is  an $\omega$-compatible, $\T$-invariant complex structure on the toric manifold $(M, \omega, \T)$. Then,  $(M, J)$  is  $\T$-equivariantly biholomorphic  to the complex manifold $M_{\Delta}^{\C}$ associated to the fan $\mathcal{F}(\Delta, {\bf L})$ of the corresponding Delzant triple $(\Delta, {\bf L}, \Lambda)$.
\end{thm}
\begin{proof}  The idea of the proof is the following. The effective action of $\T$ on $(M, J)$ complexifies to an effective holomorphic action of a complex algebraic torus $\T^{\C} = (\C^*)^m$. The pre-images of the vertices of $\Delta$ are precisely the fixed points for $\T^{\C}$. At each such fixed point $p\in M$, $\T^{\C}$ induces a linear action on the complex space $(T_pM, J_p)$. The holomorphic slice theorem (see e.g. \cite{Sj,HL}) tells us that there exist a $\T^{\C}$-invariant neighbourhood $0\subset  V \subset T_pM$, a $\T^{\C}$-invariant neighbourhood $p\subset U \subset M$ and a $\T^{\C}$-equivariant biholomorphism $f: V \to U$.  Since the action of $\T^{\C}$ is effective, $U$ must be the whole  $T_pM$, so we obtain an embedding $T_pM$ into $M$. This is a $\T^{\C}$-invariant affine chart $\C_p^m$ around $p$. The fan associated to this affine chart is a the simplicial cone $C_m^*$ dual to the Delzant image of $(\C^m, \omega_0)$. Moreover, if $q\in M$ is another fixed point, and $\C_q^m$ the corresponding affine chart, then both $\C_p^m$ and $\C_q^m$ must contain the dense principal orbit  $\T^{\C}(p_0) \cong (\C^*)^m$ of a point $p_0\in M$. It is not difficult to see that on $\T^{\C}(p_0)$, the transition map between $(\C^{*})^m_p$ and $(\C^*)^m_q$  is precisely as described in the construction of $M_{\Delta}^{\C}$ associated to $\mathcal{F}=(C^*(\Delta), \Lambda)$. \end{proof}

\section{Equivariant blow-up} We explain  now  how to blow-up  equivariantly a fixed point of the action of $\T$ on the complex manifold $M^{\C}_{\Delta}$ constructed in Section~\ref{s:fan}.  Recall that $M^{\C}_{\Delta}$ is endowed with a $\T$-equivariant atlas of affine charts $\C^m_v$, parametrized by the the vertices $v$ of $\Delta$, such that  the action $\T$ on each chart $\C^m_v$ is the standard action of $\T^m$ on $\C^m$, as described in Example~\ref{flat-action}. We notice that $\C^m_v= \C^m$ is the complex manifold associated to (the fan of)  the (unbounded) standard cone $C_m \subset \R^m$,  and the standard lattice $\Z^m \subset (\R^m)^*$, via the construction in Section~\ref{s:fan}: indeed, $C_m$ has a unique vertex at the origin, and the inward primitive normals of the adjacent facets form the standard basis of $(\R^m)^*$, which define a single chart $\C_0^m$. 

We now blow-up the origin of $\C^m$ and obtain as a resulting manifold  $\widehat{\C_0^m}=\cO_{\C P^{m-1}}(-1) \to \C P^{m-1},$ the total space of the tautological bundle over $\C P^{m-1}$. The blow-down map $b : \cO_{\C P^{m-1}}(-1) \to \C_0^m$ is explicitly given by
$$b\big([w_1, \ldots, w_m], \zeta (w_1, \ldots, w_m)\big) = ( \zeta w_1, \ldots, \zeta w_m), $$
where $[w_1, \ldots, w_m]$ are homogeneous coordinates on $\C P^{m-1}$,  and $\zeta$ is a fibre-wise coordinate of the tautological bundle (with respect to the generator $(w_1, \ldots, w_m)$).  The inverse map, defined on ${\C}_0^m \setminus \{0\}$, is 
$$ b^{-1} (z_1, \ldots, z_m) = ([z_0, \ldots, z_m], (z_1, \ldots, z_m)), $$
showing that  the action of $\T^m$ on $\C_0^m$  lifts to an action of $\T^m$ on  $\widehat{\C_0^m}$, given in the above  coordinates by
\begin{equation}\label{blow-up-action}
\begin{split}
&(e^{\sqrt{-1}t_1}, \ldots, e^{\sqrt{-1}t_m}) \cdot \big([w_1, \ldots, w_m]; \zeta(w_1, \ldots, w_m)\big) =\\
& \big([e^{\sqrt{-1}t_1}w_1, \ldots, e^{\sqrt{-1}t_m}w_m], \zeta(e^{\sqrt{-1}t_1}w_1, \ldots, e^{\sqrt{-1}t_m}w_m)\big), 
\end{split}
\end{equation}
thus making  $b$ manifestly $\T^m$-equivariant.
\begin{defn}\label{d:equivariant-blow-up}  The (non-compact) manifold $\widehat{\C_0^m}= \cO(-1)_{\C P^{m-1}}$ endowed the action of $\T^m$ defined in \eqref{blow-up-action} will be referred to as {\it the equivariant blow up} of $(\C^m, \T^m)$.
\end{defn}
Introducing affine charts on $\C P^{m-1}$ gives rise to $m$ affine charts $\C_{v_1}^m, \ldots, \C_{v_m}^m$ on $\widehat{\C_0^m}$, such that the $\T^m$ action \eqref{blow-up-action}  becomes the standard action of $\T^m$ on each such chart. Furthermore, by writing down  the transition between these charts  one sees that $\widehat{\C_0^m}= \cO(-1)_{\C P^{m-1}}$ becomes the complex manifold associated to  (the fan of) the unbounded polytope 
$$\hat{C}_m:= \{ L_0(x)=x_1+ \cdots x_m -1 \ge 0, L_j(x)=x_j \ge 0, j=1, \ldots, m\},$$
and standard lattice $\Z^m \subset (\R^m)^*$.  

{\tiny
\begin{figure}[h!]\label{f1}
\captionsetup[subfigure]{labelformat=parens}
\centering
\begin{subfigure}[b]{0.4\textwidth}
\includegraphics[scale=0.23,angle=0]{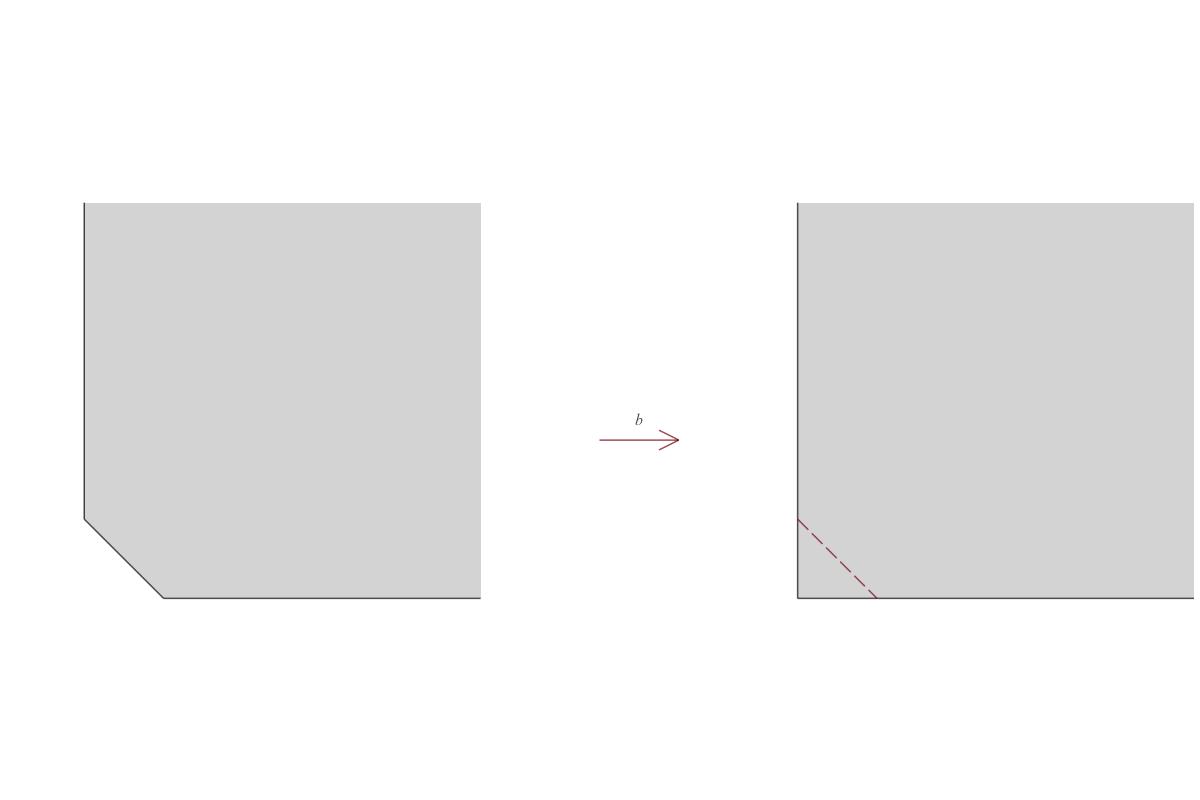}
\end{subfigure}
\caption{The polytopes of ${\widehat \C^2_0}$ and $\C^2$.}
\end{figure}}

We thus get the following
\begin{thm}\label{thm:blow-up}  Let $M_{\Delta}^{\C}$ be the compact complex manifold associated to a Delzant polytope $\Delta\subset \R^m$ with respect to the standard lattice $\Z^m \subset (\R^m)^*$. Let $w_1, w_2, \ldots, w_m$ be primitive inward-pointing edge vectors  at a vertex $v$ of $\Delta$ and $\hat \Delta_{\varepsilon}$ the polytope obtained from $\Delta$ by replacing $v$ with  the $m$-vertices $v+ \varepsilon w_i$, $i=1, \ldots, m$ for some $\varepsilon>0$. Then,  $\hat \Delta_{\varepsilon}$ is Delzant polytope too, and the corresponding complex manifold $M_{\hat \Delta_{\varepsilon}}^{\C}$ is obtained from $M_{\Delta}^{\C}$ by blowing up a fixed point for the torus action.
\end{thm}

{\tiny
\begin{exercise}\label{ex:test-configuration}
Identify the complex manifolds and the equivariant blowdown maps corresponding to the Delzant polytopes of Figure~2.
\begin{figure}[h!]\label{f2}
\captionsetup[subfigure]{labelformat=parens}
\centering
\begin{subfigure}[b]{0.4\textwidth}
\includegraphics[scale=0.23,angle=0]{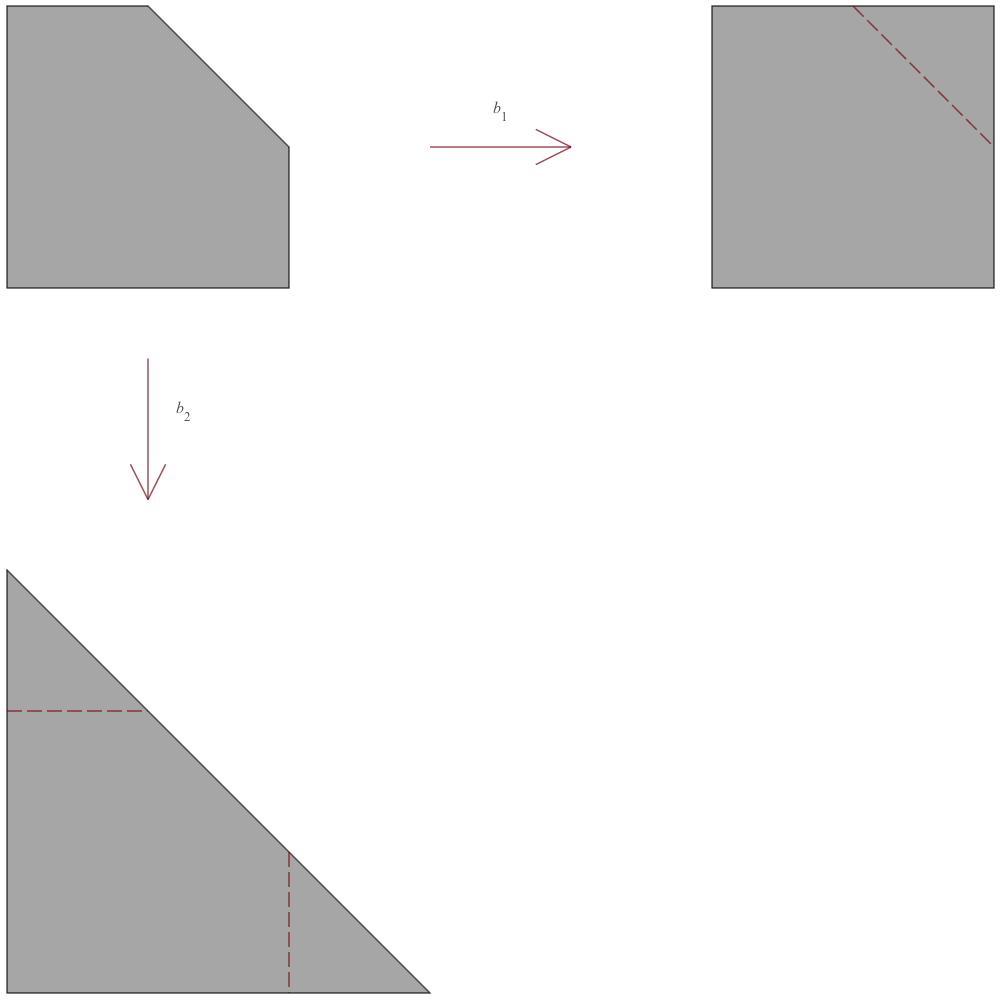}
\end{subfigure}
\caption{}.
\end{figure}
\end{exercise}}

\section{Polarized projective toric varieties} There is a third facet of the story, which relates the smooth  compact toric symplectic manifolds to the theory of {\it polarized projective varieties}.  Recall that a {smooth polarized projective complex variety} is a compact complex manifold $M$ of complex dimension $m$,  endowed with a  very ample holomorphic line bundle $L \to M$, i.e. a holomorphic line bundle such that the Kodaira map~\cite{GH}  $$\kappa : M \longrightarrow \Proj(W^*),$$  where $W=H^0(M, L)$ is the $(N+1)$-dimensional complex vector space of holomorphic sections of $L$,  is a holomorphic embedding and  $L = \kappa^* \cO(1)_{W^*}$. In this case,  one can  identify $(M, L)$ with its embedded image
$\tilde M =\kappa(M) \subset \Proj(W^*)\cong \C P^N,$ by considering  the polarization of $\tilde M$ defined by the  the restriction of the anti-tautological line bundle $\cO(1)$  of $\C P^N$.  In this setting, we have
\begin{defn}\label{toric-polarization} A  {\it toric} polarized projective variety $\tilde M \subset \C P^N$ is  an $m$-dimensional  complex submanifold  of $\C P^N$  which is the (Zariski) closure in $\C P^N$ of a principal orbit  for the action of a complex $m$-dimensional complex torus $\T^{\C} \subset {\bf SL}(N+1, \C)$  on $\C P^N$.
\end{defn}
Denote by  $\T$ the  real $m$-dimensional torus corresponding to $\T^{\C}$,   and let  $\T^N \subset {\bf SL}(N+1, \C)$  be a maximal real torus containing $\T$.  As discussed  in Section~\ref{s:CPm}, $\C P^N$ admits a 
$\T^N$-invariant Fubini--Study K\"ahler metric $\omega^{FS}$  and it is not difficult to show  that $2\omega^{FS}$  is the curvature of a $\T^N$-invariant hermitian metric on $\cO(1)$. Restricting to $\tilde M$, we obtain a $\T$-invariant symplectic form $\omega$ on $\tilde M$,   belonging to the deRham class $2\pi c_1(L)$. The $\T$-action is hamiltonian  with respect to this symplectic structure,  since $\T \subset \T^N \subset {\bf SL}(N+1, \C)$  and  the action of $\T^N$  is hamiltonian  with respect to $2\omega^{FS}$. We thus have
\begin{prop}\label{p:toric-polarization} Any smooth  polarized toric complex variety is a symplectic toric manifold.
\end{prop}
From the considerations in Section~\ref{s:fan}, we get a correspondence between  smooth toric projective varieties  and complex toric varieties defined  via  the Delzant construction.  It is however not entirely clear how this correspondence translates in terms of Delzant polytopes or,  in other words,  how to construct a polarized toric variety from a given Delzant polytope. To explain this, we  give the following
\begin{defn}[{\bf Lattice polytopes}] \label{d:lattice-polytope} Suppose $(\Delta, {\bf L}, \Lambda)$ is a Delzant polytope as in Definition~\ref{d:Delzant-polytope}. We say that $(\Delta, {\bf L}, \Lambda)$ is  a {\it lattice Delzant polytope} if, moreover, all vertices of $\Delta$ belong to the dual lattice $\Lambda^* \subset V$.
\end{defn}
For any lattice Delzant polytope, we  can take a basis $\{e_1, \ldots, e_m\}$ of $\Lambda$  and consider $\Delta \subset \R^m$ with the standard lattice $\Lambda^* =\Z^m \subset \R^m$;  the fact that $(\Delta, {\bf L}, \Lambda)$ is a lattice Delzant polytope translates to the fact that the vertices of $\Delta$ are in $\Z^m$ (see Exercise~\ref{ex:Delzant-simplified}). We denote by 
$$A(\Delta) := \{\lambda^{(j)}= (\lambda_1^{(j)}, \ldots, \lambda_m^{(j)}), j=0, \ldots, N\}$$ 
all the lattice points in $\Delta$. We then consider the $(\C^*)^m$-action on $\C P^{N}$,  given in homogeneous coordinates by
$$(z_1, \ldots, z_m) \cdot [w_0, \ldots, w_N] := [(z_1^{\lambda_1^{(0)}}\cdots z_m^{\lambda^{(0)}_m})w_0, \ldots, (z_1^{\lambda_1^{(N)}}\cdots z_m^{\lambda_m^{(N)}})w_N].$$
We associate to such a data the toric polarized variety $M_{A(\Delta)} \subset \C P^N$ which is the Zariski closure in $\C P^N$ of the $(\C^*)^m$-orbit of the point $[1, \ldots , 1]\in \C P^N$ for this action. The basic fact is the following
\begin{thm}[\bf Section 6.6 in \cite{daSilva-book}] \label{thm:toric-variety} For any lattice Delzant polytope $\Delta$, $M_{A(\Delta)} \subset \C P^N$ is a smooth polarized toric projective variety whose Delzant polytope with respect to a $\T^m$-invariant Fubini--Study symplectic form  in $2\pi c_1(\cO(1))$ is $\Delta \subset \R^m$ with lattice $\Lambda^*=\Z^m$. In particular, $M_{A(\Delta)}$ is biholomorphic to the complex manifold $M^{\C}_{\Delta}$ defined in Section~\ref{s:fan}.
\end{thm}
{\tiny 
\begin{ex}\label{ex:CPm-variety} Let $\Delta \subset \R^m$ be the  lattice Delzant simplex,  defined by  
$$\Delta= \{L_0=(1-\sum_{j=1^n} x_j) \ge 0, L_j(x)=x_j \ge 0,  j=1, \ldots m\}$$
and the standard  lattice $\Z^m \subset (\R^m)$. $\Delta$.  The set of  lattice points of $\Delta$ is 
$$A(\Delta)=\{\lambda^{(0)}=(0, \ldots, 0),  \lambda^{(j)}=e_j, j=1, \ldots, m\},$$
where $e_j$ is the standard basis of $\R^m$. The Delzant construction identifies the corresponding symplectic toric manifold with $(\C P^m, 2\omega^{FS})$,  whereas  Theorem~\ref{thm:toric-variety} realizes $M_{A(\Delta)}$ as the closure in $\C P^{m}$ of a principal $(\C^*)^m$-orbit  for the action
$$(z_1, \ldots, z_m) \cdot [w_0, \ldots, w_{m+1}] =[w_0, z_1w_1, \ldots, z_m w_m],$$
which clearly is $\C P^m$. The induced symplectic structure on this polarized variety is again $2\omega^{FS}$.
\end{ex} }
{\tiny
\begin{exercise}\label{e:weight-decomposition} Let $M_{A(\Delta)} \subset \C P^N$ be a toric polarized projective variety corresponding to a lattice Delzant polytope $\Delta \subset \R^m$ with respect to the lattice $\Z^m \subset \R^m$. Using that the sections of $\cO(1)$ are identified with linear functions in homogeneous coordinates $[w_0, \ldots, w_N]$ on $\C P^N$, show that the $(\C^*)^m$-action on $M_{A(\Delta)}$ defines a  $(\C^*)^m$-action  on the vector space $H^0(M_{A(\Delta)}, L)$ of holomorphic sections  of $L$. Furthermore, show that  here is a basis $\{ s_0, \ldots, s_N\}$ of $H^0(M_{A(\Delta)}, L)$,  parametrized by the lattice points $\{\lambda^{(0)}, \ldots, \lambda^{(N)}\}$ in $A(\Delta)$, such that $(\C^*)^m$ acts on $s_j$ by
$$(z_1, \ldots, z_m) \cdot s_j = (z_1^{\lambda_1^{(j)}} \cdots z_m^{\lambda_m^{(j)}}) s_j.$$
\end{exercise}}

\section{Toric orbifolds} We briefly discuss here what happens with the Delzant construction  reviewed in Section~\ref{s:Delzant-construction} when one starts with a rational Delzant triple $(\Delta, {\bf L}, \Lambda)$, i.e  satisfying  the weaker conditions (i)-(ii)-(iii)' of Definition~\ref{d:Delzant-polytope}. There are two points in the construction which deserve a special attention. 

The first point is {Claim 1}.  Consider for example  the labelled Delzant simplex 
$(\Delta_m, {\bf L})$ with ${\bf L}=\{L_j(x)=x_j, j=1, \ldots m, L_{m+1}(x)= -\sum_{j=1}^m x_j + \frac{1}{2}\}$, corresponding to $(\C P^m, \omega^{FS}),$  and change the standard lattice $\Lambda =\Z_m \subset (\R^m)^*$ with the lattice $\Lambda' = \frac{1}{2}\Lambda$. The condition (iii)'  then holds for the triple $(\Delta_m, {\bf L}, \Lambda')$. However, in this case,  
$$N' = {\rm Ker}(\tau) = \mathbb{S}^1 \times \Z_2^m, $$
where $\Z_2^m:=(\pm 1, \pm 1, \ldots, \pm 1) \subset \T^{m+1}$.
In particular, $N'$ is no longer connected. Following  the construction further, the quotient space will be 
$$M'={\mathbb S}^{2m+1}/N' = {\C P^m} /{\Z}^m_2.$$ Such a space is an example of an {\it orbifold}.

\begin{defn}\label{d:orbifold} An {\it orbifold chart} on a topological space $M$ is a triple $(\hat U, \Gamma, \varphi)$ where:
\begin{enumerate}
\item[$\bullet$] $\hat U \subset \R^n$ is an open subset;
\item[$\bullet$] $\Gamma \subset {\bf GL}(n)$ is a finite subgroup acting on $\hat U$
\item[$\bullet$] $\varphi : \hat U/\Gamma \to M$  is  a homeomorphism between $\hat U/\Gamma$ and an open subset $V\subset M$.
\end{enumerate}
We denote by $\hat \varphi : \hat U \to M$ the induced $\Gamma$-invariant map.

An {\it embedding} between two orbifold charts $(\hat U_1, \Gamma_1, \hat \varphi_1)$ and $(\hat U_2, \Gamma_2, \hat \varphi_2)$ is a smooth embedding $\lambda : \hat U_1 \to \hat U_2$ such that $\hat \varphi_1 = \hat \varphi_2 \circ \lambda$.

Two orbifold charts $(\hat U_1, \Gamma_1, \hat \varphi_1)$ and $(\hat U_2, \Gamma_2, \hat \varphi_2)$  are {\it compatible} if there exists an open subset $V\subset V_1 \cap V_2$ (where we have set $V_i = \hat \varphi_i(\hat U_i)$)  and an orbifold chart $(\hat U, \Gamma, \hat \varphi)$  with $\hat \varphi (\hat U) = V$, and two embeddings $\lambda_i : \hat U \to \hat U_i, i=1,2$.

An {\it orbifold} $M$ of (real) dimension $n$ is a Hausdorff paracompact topological space,  endowed with an open covering of compatible orbifold charts $\{(\hat U_i, \Gamma_i, \hat \varphi_i)\}_{i\in I}$.
A {\it smooth function} $f: M \to \R$ on an orbifold $M$ is defined by the property that on each orbifold chart $(\hat U, \Gamma, \hat \varphi)$, it pulls-back  by $\hat \varphi$ to a $\Gamma$-invariant smooth  function on $\hat U$.  One can define smooth tensors on $M$ is a similar way.
\end{defn}

The above phenomenon can be remedy by taking $\Lambda_{\rm min} = {\rm span}_{\Z}\{u_1, \ldots u_d\},$ which is the minimal (under the inclusion) lattice for which the condition (iii)' is satisfied for the given labelling ${\bf L}$. The subgroup $N_{\rm min}$ will be then a torus and for any other choice of lattice $\Lambda$,  the produced quotient space $S/N$ will be a quotient of $S/N_{\rm min}$ by the finite abelian group  $\Lambda/\Lambda_{\rm min}$. Geometrically, this  is translated to taking orbifold coverings, $S/N_{\rm min}$ being the simply connected orbifold covering all other quotients, see Thurston~\cite{Th}.

The second point is {Claim 4}. When we assume (iii)' instead of (iii), the stabilizers for the $\T^d$-action on $S$ will be finite abelian groups, and in general so be the stabilizers of the action of $N$.  However, once again the quotient space $S/N$ will be an orbifold in the sense of the above definition.

To summarize, we observed that 
\begin{prop} Suppose $(\Delta, {\bf  L}, \Lambda)$ is a rational Delzant triple, i.e.  satisfies the conditions {\rm (i)-(ii)-(iii)'} of Definition~\ref{d:Delzant-polytope}. Then the Delzant construction associates to $(\Delta, {\bf L}, \Lambda)$ a compact symplectic orbifold $(M, \omega)$ endowed with a Hamiltonian action of the torus $\T= V/2\pi\Lambda$ and momentum image $\Delta$.
\end{prop}
The converse is also true, due to the following extension of Delzant's correspondence to toric orbifolds
\begin{thm}[\bf Lerman--Tolman~\cite{LT}] Compact symplectic orbifolds modulo equivalence are in bijective correspondence with rational Delzant triples modulo the action of the affine group.
\end{thm}
We give below one specific example 
{\tiny \begin{ex}[\bf The weighted projective spaces]\label{wps} Similarly to the way we introduced the complex projective space, we can consider for any $(m+1)$-tuple of positive integers ${\bf a}=(a_0, \ldots, a_m)$  with ${g.c.d}(a_0, \ldots, a_m)=1$ the quotient space
$$\C P^m_{\bf a} = \C^{m+1}\setminus \{0\}/\C^{*}_{\bf a}, $$
where $\C^{*}_{\bf a}$ denotes the $\C^*$ action $\rho_{\bf a}$ on $\C^{m+1}=(z_0, \ldots, z_m)$ by
$$\rho_{\bf a}(\lambda)(z):= (\lambda^{a_0}z_0, \ldots, \lambda^{a_m}z_m).$$
\begin{lemma} $\C P^m_{\bf a}$ is an orbifold, which is diffeomorphic (as an orbifold) to the quotient space
$${\mathbb S}^{2m+1}/{\mathbb S}^1_{\bf a}$$
where ${\mathbb S}^1_{\bf a}$ stands for the circle action $\rho_{\bf a}$ on ${\mathbb S}^{2m+1},$ given by
$$ \rho_{\bf a} (e^{\sqrt{-1}t})(z)= (e^{\sqrt{-1}a_0t}, \ldots, e^{\sqrt{-1}a_mt}).$$
\end{lemma} 
The proof of this result is left as an exercise.
We can now easily modify the construction of the toric structure on $\C P^m$ and get the following
\begin{lemma} The orbifold $\C P_{\bf a}^m$ admits a toric K\"ahler structure $(g, \omega, J)$ obtained by the Delzant construction starting with the rational Delzant triple $(\Delta_m, {\bf L}_{\bf a}, \Lambda_{\bf a})$ where $\Delta_m\subset \R^m$ is the standard simplex labelled as
\begin{equation*}
\begin{split}
\Delta_m :=& \Big\{(x_1, \ldots, x_m) :  L_j(x)= \Big(\frac{a_0\cdots a_m}{a_j}\Big) x_j \ge 0, j=1, \ldots, m, \\
              & \ \ \ \ \ L_{m+1}(x) =\Big(\frac{a_0 \cdots a_m}{a_0}\Big)\Big(\frac{1}{2}-\sum_{j=1}^m x_j\Big) \ge 0\Big\},
              \end{split}
              \end{equation*}
and  the lattice $\Lambda_{\bf a}$ is defined by 
$$\Lambda_{\bf a} := {\rm span}_{\Z}\Big\{\Big(\frac{a_0\cdots a_m}{a_j}\Big)e_j^*, \  j=1, \ldots, m, \ \Big(\frac{a_0\cdots a_m}{a_0}\Big)(e_1^* + \cdots + e_m^*)\Big\}$$
where $\{e_j^*\}_{j=1}^m$ stands for the standard basis of $(\R^m)^*$.
\end{lemma}
\end{ex}}

\chapter{Abreu--Guillemin Theory}
\section{The orbit space of a toric manifold} Let $(M, \omega, \T)$ be a toric symplectic manifold and $\mu : M \to \Delta$ the corresponding Delzant polytope. We denote by
$$M_{\rm red} := M/\T$$ the space of orbits for the $\T$-action on $M$, which is a compact Hausdorff space  with respect to the quotient topology, see e.g. \cite{bredon}. The momentum map $\mu$ being $\T$-invariant descends to a continuous map $\check{\mu} : M_{\rm red} \to \Delta$. Some important ingredients in Delzant's theorem are the following

\smallskip
\noindent
{\bf Fact 1.} $\check{\mu}$ is a bijection. 

\smallskip
\noindent
{\bf Fact 2.} Delzant's proof~\cite{Delzant} shows that the pre-image $p=\mu^{-1}(x)$ of a point $x\in \Delta$ situated on an open face $F^0\subset \Delta$ of co-dimension $\ell$  has a stabilizer $\T_p$ which is a torus of dimension $\ell$.  In particular, the pre-image of a point $x\in \Delta^0$ is a principal orbit isomorphic to $\T$. We shall denote $M^0:=\mu^{-1}(\Delta^0)$ the dense open subset of $M$ consisting of points having principal orbits.  As $\T$  acts freely on $M^0$,  $\mu: M^0 \to  \Delta^0$ is a principal $\T$-bundle over $\Delta^0$.

\smallskip
Another fact coming from \cite{Delzant} is 

\smallskip
\noindent
{\bf Fact 3.} $M_{\rm red}$ and $\Delta$ both admit ``differentiable structure'' of manifolds with corners induced from the smooth structure of $M$, and $\check{\mu}$ is a diffeomorphism in this category, (\cite{KL}, Proposition C7). 

{\tiny
\begin{exercise}  Check Facts 1--3 for the symplectic manifold $(M, \omega)$ associated to a Delzant polytope $(\Delta, {\bf L}, \Lambda)$ via the construction in Section~\ref{s:Delzant-construction}.
\end{exercise}}
\smallskip
The differentiable structure on $\Delta \subset \mathfrak{t}^*$ is the one naturally  induced by restricting to $\Delta$ the smooth functions on $\mathfrak{t}^*$. To see how it is related to the smooth structure on $M$, we notice  the following
\begin{lemma}\label{l:Schwarz}{\bf (Schwarz~\cite{schwartz})} A $\T$-invariant function $f(p)$ on $M$ is smooth iff 
$$f(p)= \varphi(\mu(p))$$
for some smooth function $\varphi(x)$ on $\mathfrak{t}^*$.  
\end{lemma}
In the case when $(M, \omega, \T)=(\R^{2}, \omega_0, {\mathbb S}^1)$,  this is  a well-know result by Whitney; in general,  Lemma~\ref{l:Schwarz} can be easily derived from (\cite{Delzant}, Lemme 5). Because of this basic  fact, we shall often write $f(x)$  for the composition $f(p) =\varphi (\mu(p))$.

\section{Toric K\"ahler metrics:  local theory}
To simplify the presentation, we fix a basis $\{e_1, \ldots, e_m\}$ of $\tor$  and  denote by $K_j=X_{e_j}$ the induced fundamental vector fields; we shall also identify $\tor^* \cong (\R^m)^*$ by using the dual basis $\{e_1^*, \ldots, e_m^*\}$ and write $x=(x_1, \ldots, x_m)$ for the elements of $\tor^*$. As explained in the previous section, on $M^0$, $K_1, \ldots, K_{m}$ are functionally independent, i.e.  for each point $p\in M^0$, $(K_1\wedge \cdots \wedge K_m)(p)\neq 0$. We shall also identify the coordinate function $x_i=\langle x, e_i\rangle$ with the momentum function $\langle \mu, e_i\rangle$ on $M$, i.e. we write
\begin{equation}\label{x-momenta}
\imath_{K_i} \omega = - dx_i.
\end{equation}

Let $(g, J)$ be a $\T$-invariant $\omega$-compatible K\"ahler structure on $M$. Letting  $$H_{ij} = g(K_i, K_j)$$ we get a $\T$-invariant smooth function on $M$, which will tacitly identify with a smooth function $H_{ij}(x)$ on $\Delta$, see {Fact 2} above. We denote by ${\bf H}_x=(H_{ij}(x))$ the corresponding symmetric matrix with smooth entries over $\Delta$.  In a more intrinsic language, we regard ${\bf H}$ as a smooth function over $\Delta$ with values  in $S^2\tor^*$ by letting
$${\bf H}_x(\xi_1, \xi_2) : = g_p(X_{\xi_1}, X_{\xi_2})$$  for any $\xi_1, \xi_2 \in \tor$ and any $p\in \mu^{-1}(x)$.

On $\Delta^0$, ${\bf H}$ is positive definite and we denote by ${\bf G} := {\bf H}^{-1}$ the inverse matrix; equivalently, ${\bf G}$ is a smooth $S^2\tor$-valued function on $\Delta^0$.

Notice that, by \eqref{x-momenta}, 
\begin{equation}\label{dx}
dx_i(JK_j)= -\omega(K_i, JK_j)=--g(JK_i, JK_j)=-g(K_i,K_j)= -H_{ij}(x).
\end{equation}

We now consider the vector fields $\{K_1, \ldots, K_{m}, JK_{1}, \ldots, JK_m\}$. They form a basis of $T_pM$ at each $p\in M^0$  (because $\{K_1, \ldots, K_m\}$ span an $m$-dimensional space and $\{JK_1, \ldots, JK_m\}$ span its $g$-orthogonal complement) and satisfy
\begin{equation}\label{commuting}
[K_i, K_j] =[K_i, JK_j]=[JK_i, JK_j]=0,
\end{equation}
where for the second identity we used that $\T$ preserves $J$,  whereas for the third identity we used the integrability of $J$, see \eqref{integrability}.
We now denote by $$\{\theta_1, \ldots, \theta_m, J\theta_1, \ldots, J\theta_m\}$$ the {\it dual basis} of $T^*M^0,$  corresponding to $$\{K_1, \ldots, K_m, JK_1, \ldots, JK_m\},$$ where for a $1$-form $\theta$ we set $J\theta(X) = - \theta(JX)$, for any  vector field $X$. The commuting relation \eqref{commuting} is equivalent to 
\begin{equation}\label{commuting1}
d\theta_i=0 = d(J\theta_i), \ \ i=1, \ldots m.
\end{equation}
As the $1$-forms $J\theta_i$ satisfy
$$\imath_{K_j} J\theta_i =0, \ \ \mathcal{L}_{K_j} J\theta_i =0, $$
in terms of the $\T$-bundle structure $\mu : M^0 \to \Delta0$,  each $J\theta_i$ is {\it basic}, i.e. $J\theta_i=\mu^*(\alpha_i)$ of a closed $1$-form $\alpha_i$ on $\Delta^0$. As $\Delta^0$ is contractible, we can write 
\begin{equation}\label{y}
J\theta_i = -dy_i
\end{equation}
for some smooth function $y_i(x),$ defined on $\Delta^0$ up to an additive constant (as usual, we omit the pull-back by $\mu$ in the notation). Furthermore, by \eqref{dx}, we find
\begin{equation}\label{y-x}
\begin{split}
-J\theta_i &=dy_i = \sum_{j=1}^m G_{ij}(x)dx_j, \\
Jdx_i &= \sum_{j=1}^m H_{ij}(x) \theta_j.
\end{split}
\end{equation}

The $1$-forms $\{\theta_1, \ldots, \theta_m\}$ define a $1$-form ${\boldsymbol \theta}$ with values $\tor = {\rm Lie}(\T)$,  by letting
\begin{equation}\label{connection}
{\boldsymbol \theta} = \sum_{i=1}^m \theta_i \otimes e_i, 
\end{equation}
whose curvature is identically zero.  The symplectic $2$-form $\omega$ then becomes 
\begin{equation} \label{omega}
\omega:= \langle d\mu \wedge {\boldsymbol \theta}\rangle=\sum_{i=1}^m dx_i \wedge \theta_i, \end{equation}
whereas the K\"ahler metric  is 
\begin{equation}\label{g}
\begin{split}
g &= \sum_{i,j=1}^m g(K_i, K_j)\Big(\theta_i \otimes \theta_j + J\theta_i \otimes J\theta_j\Big) \\
   &= \sum_{i,j=1}^m H_{ij} \Big( \theta_i \otimes \theta_j + dy_i \otimes dy_j \Big) \\
    &=\sum_{i,j=1}^m \Big( G_{ij} dx_i\otimes dx_j + H_{ij} \theta_i \otimes \theta_j \Big) \\
    &= \langle d\mu, {\bf G}, d\mu \rangle + \langle {\boldsymbol \theta}, {\bf H},  {\boldsymbol \theta}\rangle.
    \end{split}
    \end{equation}

\begin{lemma}[\bf Guillemin~\cite{guillemin}]\label{guillemin} Let $(M, \omega, \T)$ be a symplectic toric manifold with Delzant polytope $\Delta$ and $(g, J)$ an $\omega$-compatible $\T$-invariant  K\"ahler structure. Then, on $M^0$, $(g, \omega)$ can be written in the form \eqref{omega}--\eqref{g}, where $G_{ij} = \frac{\partial^2 u}{\partial x_i \partial x_j}$ for a smooth, strictly convex function $u(x)$ on $\Delta^0$, called {\bf symplectic potential} of $g$. Conversely, for any strictly-convex smooth function $u$ on $\Delta^0$, the riemannian metric on $M^0$ defined by \eqref{g} with ${\bf G} = {\rm Hess} (u),  \  {\bf H}= {\bf G}^{-1}$ and the flat connection $1$-form ${\boldsymbol \theta}$  defines an $\omega$-compatible, $\T$-invariant K\"ahler structure on $M^0$. 
\end{lemma}
\begin{proof} We shall denote  with subscript $ f_{,k}$ the partial derivative $\frac{\partial f}{\partial x_k}$  of a smooth  function $f$ on $\Delta$. Letting
$$\beta:= \sum_{i=1}^m y_i dx_i, $$ we have by \eqref{y-x} 
$$d\beta= \sum_{i=1}^m dy_i \wedge dx_i = \sum_{i,j=1}^m G_{ij}dx_i \wedge dx_j =0.$$
By the Poincar\'e Lemma on $\Delta^0$, $\beta = du$ for some smooth function $u$ on $\Delta^0$, i.e. $y_i = u_{,i}$. It follows that
\begin{equation}\label{potential}
G_{ij} = y_{i,j}= u_{, ij} = ({\rm Hess}(u))_{ij}
\end{equation}

Conversely, for the metric \eqref{g}, we let $G_{ij}=g(K_i, K_j)$ and suppose that \eqref{potential}  holds true. It follows that the $1$-forms
$$-J\theta_i= \sum_{j=1}^m G_{ij} dx_j= \sum_{j=1}^{m} u_{,ij}dx_j$$
are closed. As $\theta_i$ are closed too (by the assumption that ${\boldsymbol \theta}$ is flat), we get a basis of $\Lambda^{1,0}(M^0, J)$ of closed $(1,0)$-forms $-J\theta_i  + \sqrt{-1}\theta_i$, so writing locally  $-J\theta_i  + \sqrt{-1}\theta_i= dy_i + \sqrt{-1} dt_i$ we get holomorphic coordinates  $y_i + \sqrt{-1} t_i$ for  $J$,  i.e. $J$ is integrable.  \end{proof}

We have seen in Chapter~\ref{ch:Delzant} that the action of $\T \cong (\mathbb{S}^1)^m$ on $(M, J)$ extends to an (effective) holomorphic action of the complex torus $\T^{\C} \cong (\C^*)^m$. Fixing a point $p_0 \in M^0$, we can identify $M^0$ with  the orbit $\T^{\C}(p_0) \cong (\C^*)^m$. Using the polar coordinates $(r_i, t^0_i)$ on each $\C^*$, this identification gives rise the so-called {\it angular coordinates} 
$$\boldsymbol{t}=(t^0_1, \ldots, t^0_m) : M^0 \to \mathfrak{t}/2\pi \Lambda.$$
If another reference point is chosen, $\boldsymbol{t}$ varies by an additive constant in $\mathfrak{t}$. Writing $\boldsymbol{\theta} = d\boldsymbol{t}$, we have
\begin{defn} For a fixed $\omega$-compatible, $\T$-invariant complex structure $J$ on $(M,\omega, \T)$ (corresponding to a symplectic potential $u(x)$ on $\Delta^0$),  and a base point $p_0 \in M^0$ (giving rise to  angular coordinates ${\bf t} = \sum_{i=1}^m t_i e_i$ with respect to a basis $e=\{e_1, \ldots, e_m\}$ of $\tor$),  the functions $\{x_1, \ldots, x_m; t_1, \ldots, t_m\}$ on $\Delta^0\times \T$ are called {\it momentum-angle coordinates} associated to $(g, J)$.
\end{defn}

\section{The scalar curvature}

Recall that if $(M^{2m}, J, g, \omega)$ is a K\"ahler manifold, the riemannian metric $g$ induces a hermitian metric (still denoted by $g$) on the {\it anti-canonical} complex line bundle $K^{-1}(M) = \wedge^m(T^{1,0} M)$, where $TM \otimes \C= T^{1,0}M \oplus T^{0,1}M$ is the {\it type decomposition} of the complexified tangent bundle of $M$ into $\sqrt{-1}$ and  $-\sqrt{-1}$ eigen-spaces of  $J$. Furthermore, $K^{-1}(M)$ has a canonical holomorphic structure, induced by the holomorphic structure of $T^{1,0}M$ (for which the holomorphic vector fields are the holomorphic sections).  The {\it Ricci-form}  $\rho^g$ of $g$ is defined in terms of the curvature of the Chern connection $\nabla^g$ of $(K^{-1}(M), g)$, by $$R_{X,Y}^{\nabla^g} = \sqrt{-1} \rho_g(X,Y).$$ In the K\"ahler case, $\nabla^g$ coincides with  the induced connection on $K^{-1}(M)$ by the Levi--Civita connection of $g$,  and the $2$-form $\rho_g$ is essentially the Ricci tensor ${\rm Ric}_g$ of $g$, i.e. we have $$\rho_g (X, Y) = {\rm Ric}_g(JX, Y).$$ It thus follows that the scalar curvature $s_g:= {\rm tr}_g ({\rm Ric}_g)$ of $g$ is equivalently given by \begin{equation}\label{scal} s_g = {\rm tr}_{\omega} (\rho_g) := 2m \Big((\rho_g \wedge \omega^{m-1})/\omega^m\Big). \end{equation}

We now consider a $\T$-invariant, $\omega$-compatible K\"aher metric on a symplectic toric manifold $(M, \omega, \T)$, written on $M^0$  in the form \eqref{g} where ${\bf G} = {\rm Hess}(u)$  for a symplectic potential $u \in S(\Delta, {\bf L})$.

\begin{lemma}[\bf Abreu~\cite{abreu0}]\label{l:abreu} The Ricci form of $(g, J)$ is given by  $$\rho_g = -\frac{1}{2} \sum_{i,j,k=1} ^m H_{ij,ik} dx_k \wedge \theta_j $$ whereas the scalar curvature is  $$s_g = -\sum_{i,j=1}^m H_{ij,ij}, $$ where we recall that  for $u\in S(\Delta, {\bf L})$, $H_{ij}= ({\rm Hess}(u))^{-1}_{ij}= u^{,ij}$. \end{lemma}

\begin{proof} As  the fundamental vector fields $K_i$ preserve $J$, 
$$\sigma:= (K_1-\sqrt{-1} JK_1) \wedge \cdots \wedge (K_m- \sqrt{-1}JK_m)$$
 is a holomorphic section on $K^{-1}(M)$ which does not vanish on $M^0$. It is a well-known fact (see e.g. \cite{Wells}) that the (Chern) Ricci form is then given on $M^0$ by
$$\rho_g= - \frac{1}{2} dd^c \log |\sigma|_g^2.$$
In our case,  $|\sigma|_g^2 = 2^m {\rm det} (g(K_i, K_j)) = {\rm det} \ {\bf H}$, so we obtain
\begin{equation*}
\begin{split}
d^c \log \det {\bf H} &= {\rm tr} \Big({\bf H}^{-1} d^c {\bf H}\Big) = \sum_{i,j,k}  G_{ij} H_{ij,k} J dx_k \\
                                   &=\sum_{i,j,k,\ell} G_{ij} H_{ij,k} H_{k\ell} \theta_{\ell} \\
                                   &=-\sum_{i,j,k,\ell} G_{ij,k} H_{ij} H_{k\ell} \theta_{\ell}  \\
                                   &=-\sum_{i,j,k,\ell} G_{ik,j} H_{ij} H_{k\ell} \theta_{\ell}  \\
                                  &=  \sum_{i,j,k,\ell} G_{ik} H_{ij} H_{k\ell,j} \theta_{\ell}  \\
                                  &= \sum_{j, \ell} H_{j\ell, j} \theta_{\ell}.
                                  \end{split}
                                  \end{equation*}
We then compute                                 
\begin{equation}\label{ricci}
\begin{split}
\rho_g &= - \frac{1}{2} dd^c \log \det {\bf H} \\
             &=  -\frac{1}{2} \sum_{i,j,k} H_{ij,ik} dx_k \wedge \theta_j.
             \end{split}
             \end{equation}
The formula for $s_g$ follows  from \eqref{scal}, by using that $\omega = \sum_{i} dx_i \wedge \theta_i$ and \eqref{ricci}. \end{proof}

\section{Symplectic versus complex: the Legendre transform} \label{s:complex} We now turn our attention to the function $y_j = u_{,j}$ on $M^0$ introduced in \eqref{y-x}. For a choice of angular coordinates  $\theta_j = dt_j$  we have $-J\theta_j + \sqrt{-1} \theta_j=  dy_j + \sqrt{-1} dt_j$ and $y_j + \sqrt{-1}t_j$ define  $J$-holomorphic coordinates on $M^0$; exponentiating, we obtain another set of holomorphic coordinates $z_j := e^{y_j} e^{\sqrt{-1}t_j}$ on $M^0$. 

Now, choosing $\{e_1, \ldots, e_m\}$ to be a basis of $\Lambda$, and considering the action of the flows of $\{K_1, \ldots, K_m; JK_1, \ldots, JK_m\}$ around the reference point $p_u \in M^0$  corresponding to  the unique minima $x_u \in \Delta^0$ of $u$ under the momentum map and $t_j(p_0)=0$,  we see that $(z_1, \ldots, z_{m})$ provide an $\T^{\C}$ equivariant  identification
\begin{equation}\label{complex-identification}  \Phi_{p_u, u,e} : (\C^*)^m\cdot p_u  \cong (M^0, J),
\end{equation}
A subtle point in the construction  is that $\Phi_{p_u,u,e}$  depends upon the choice of $u$ {\it and} the basis $e=\{e_1, \ldots, e_m\}$ of $\mathfrak{t}$.

\begin{defn} [\bf Legendre transform] Let $u$ be a strictly convex smooth function on $\Delta^0$.  Letting $y_j(x) := u_{,j}(x)$ and 
$${y}(x)  : = \sum_{i=1}^m y_i(x) e_i = (du)_x \in (T_x \Delta^0)^* \cong (\tor^*)^*=\tor$$  be the map from $\Delta^0$ to $\tor$,  we define 
the {\it Legendre transform} of $u(x)$ to be the function  $\phi(y)=\phi(y_1, \ldots, y_m)$ satisfying
\begin{equation}\label{legendre-transform}
\phi(y(x)) + u(x)= \langle y(x), x\rangle, 
\end{equation}
where ${x}= \sum_{i=1}^{m} x_i e_i^*$ is seen as a smooth function from $\Delta^0$ to $\tor^*$.
\end{defn}

\begin{lemma}[\bf Guillemin~\cite{guillemin}] \label{symplectic-kahler potential} Let $(g,J)$ be an $\omega$-compatible, $\T$-invariant K\"ahler structure on $(M, \omega, \T)$ with symplectic potential $u(x)$ defined on $\Delta^0$. Then, the Legendre transform $\phi(y)$ of $u(x)$ is a K\"ahler potential of the symplectic form $\omega$ on $M^0$, i.e. satisfies
$$\omega = dd^c \phi,$$
where $d^c\phi= J d\phi$.
\end{lemma}
\begin{proof} By its very definition,
\begin{equation*}
\phi(y(x)) = \sum_{i=1}^m y_i x_i - u(x) = \sum_{i=1}^m x_i u_{,i} - u(x),
                 \end{equation*}
                 so that
 \begin{equation*}
d\phi=  \sum_{i=1}^m \Big(d(x_i u_{,i}) - u_{,i} dx_i \Big)= \sum_{i,j=1}^m x_i u_{, ij} dx_j=  \sum_{i,j=1}^m x_i G_{ij} dx_j,
                 \end{equation*}
\begin{equation*}
\begin{split}
dd^c \phi = d J d\phi  &=   d\sum_{i,j=1}^m x_i G_{ij} J dx_j \\
                                       &=  d\sum_{i,j,k=1}^m x_i G_{ij} H_{jk} \theta_k \\
                                       &= \sum_{i=1}^m dx_i \wedge \theta_i = \omega
                                       \end{split}
                 \end{equation*}
                 \end{proof}

\section{The canonical toric K\"ahler metric} We now compute the symplectic potential of the K\"ahler metric $(g_0, J_0)$ induced  by the flat K\"ahler structure $(\tilde g_0, \tilde J_0)$ on $\C^d$ via the Delzant construction, see Section~\ref{s:Delzant-construction}.  We use the general form \eqref{g} of a $\T$-invariant, $\omega$-compatible K\"ahler structure and adopt the following
\begin{defn} The reduced metric  $g_{\rm red}$ associated to $g$ written as \eqref{g} on the fibration $\mu : M^0 \to \Delta^0$ is 
$$g_{\rm red}= \sum_{i,j=1} ^mG_{ij} dx_i \otimes dx_j.$$
Geometrically, $g_{\rm red}$  is the unique metric on $\Delta^0$ such that $$\mu: (M^0, g) \to \Delta^0$$ is a riemannian submersion.
\end{defn}
 \begin{lemma}[\bf Calderbank--David--Gauduchon~\cite{CDG}]\label{l:elementary} In the setting of Proposition~\ref{kahler-reduction} suppose,  moreover,  that $G$ acts freely on $\tilde M$ and $G/N$ acts freely on $M$. Denote respectively  by ${\tilde g}_{\rm red}$ the reduced metric on $\tilde M \cong {\rm Im} \mu_g  \subset \mathfrak{g}^*$ and by $g_{\rm red}$ the reduced metric on $M/(G/N) \cong {\rm Im} \mu \subset (\mathfrak{g}/\mathfrak{n})^*$. Then, 
 $$g_{\rm red} = \ell^* {\tilde g}_{\rm red},$$
 see \eqref{diag}.
 \end{lemma}
\begin{proof} Let $q \in M$ and $p\in S$  with $q=\pi(p)$. We decompose,  orthogonally with respect to $\tilde g_p$, 
$T_pS= T_pG(p) \oplus  (T_pG(p))^{\perp}$ and observe that, by the definition of reduced metric, for any $\tilde X \in (T_pG(p))^{\perp}$,  $|\tilde X|_{\tilde g} = |(\mu_G)_* \tilde X|_{{\tilde g}_{\rm red}}$. Let $X:=\pi_* (\tilde X)$. We have, by the definition of $g$,  $|X|_g= |\tilde X|_{\tilde g}$ and as $X$ is orthogonal  to the tangent space at $q$ of $G/N(q)$, we deduce $|\mu_*(X)|_{g_{\rm red}} = |(\mu_G)_* \tilde X|_{{\tilde g}_{\rm red}}$. The claim follows. \end{proof}
We turn again to Example~\ref{flat-hamiltonian-action}.
\begin{ex}[\bf The symplectic potential of the flat K\"ahler structure]\label{flat-example} Let $(\C^d, \tilde g_0, \tilde J_0, \tilde \omega_0)$ be the flat K\"ahler structure of $\C^d$. Then, the reduced metric is written on the interior of the cone $C_d^0=\{\tilde x_i >0\}$  as 
$$(\tilde g_0)_{\rm red}= \frac{1}{2}\sum_{i=1}^d \frac{d\tilde x_i\otimes d\tilde x_i}{\tilde x_i}.$$
To see this, recall that the flat metric $\tilde g_0$ on $\R^{2d}$ can be written in polar coordinates $(r_i, t_i), i=1, \ldots, d$ as
$$\tilde g_0=\sum_{i=1}^d dr_i^2 + r_i^2 dt_i^2.$$
We have already observed in Example~\ref{flat-hamiltonian-action} that the momentum coordinates are $\tilde x_i= r_i^2/2$. 
\end{ex}
\begin{thm}[\bf Guillemin~\cite{guillemin}]\label{thm:guillemin} The symplectic potential of the induced K\"ahler structure $(g_0, J_0)$ via the Delzant construction is 
$$u_0 (x)= \frac{1}{2} \sum_{j=1}^d L_j \log L_j.$$
\end{thm}
\begin{proof} We have already observed in Section~\ref{s:Delzant-construction}  that the map $\ell=j + \tilde c$ in \eqref{diag} is given by $y_j = L_j(x)$. It follows from Lemma~\ref{l:elementary} and Example~\ref{flat-example} above  that
$$g_{\rm red} = \frac{1}{2} \sum_{j=1}^d \frac{dL_j \otimes dL_j}{L_j}.$$
The claimed expression for the symplectic  potential follows easily. \end{proof}              

\section{Toric K\"ahler metrics: compactification}
We now turn to global questions. To this end we suppose that $(g_0, J_0)$ is a globally defined,  $\T$-invariant, $\omega$-compatible K\"ahler metric on a {\it compact} toric symplectic manifold $(M,\omega)$. For instance,  we can take the metric provided by Corollary~\ref{kahler-natural}. We denote by ${\bf G}_0={\rm Hess}(u_0)$ and $\boldsymbol{\theta}_0$ the hessian of the symplectic potential and the angular coordinates of this fixed reference metric. We can take for instance $u_0$  and $\boldsymbol{\theta}_0$ to be symplectic potential  and connection $1$-form of the Guillemin K\"ahler metric, see Theorem~\ref{thm:guillemin}.

\begin{lemma} \label{abreu1}Suppose $g$ is  an invariant K\"ahler structure,  given on $M^0$ by \eqref{g},  where $\boldsymbol{\theta}=\boldsymbol{\theta}_0$ are the angular coordinates  of $g_0$. Then $g$ extends to a K\"ahler structure on $M$ provided that 
\begin{equation}\label{a1}
\ \ \ \  \ \ \ {\bf G} - {\bf G}_0 \ \ is \  \ smooth \ \ on  \ \ \Delta, 
\end{equation}
\begin{equation}\label{a2}
{\bf G}_0 {\bf G}^{-1} {\bf G}_0 - {\bf G}_0 \ \ is \  \ smooth \ \ on  \ \ \Delta
\end{equation}
\end{lemma}
\begin{proof} The key point is that it suffices to show that $g$ extends as smooth tensor on $M$: Indeed, $g$ will be then non-degenerate as the endomorphism it defines with respect to $\omega$ will be, by continuity,  an almost-complex structure $J$ everywhere; the latter will be integrable everywhere by continuity too. For the smoothness of $g$, we compute (on $M^0$):
\begin{equation*}
\begin{split}
g- g_0 &= \langle d\mu, {\bf G} - {\bf G}_0, d\mu\rangle + \langle {\boldsymbol \theta}_0, {\bf H} - {\bf H}_0, {\boldsymbol \theta}_0\rangle \\
             &= \langle d\mu, {\bf G} - {\bf G}_0, d\mu\rangle + \langle J_0 d\mu, {\bf G_0}( {\bf H} - {\bf H}_0) {\bf G}_0, J_0 d\mu \rangle  \\
             &= \langle d\mu, {\bf G} - {\bf G}_0, d\mu\rangle + \langle J_0 d\mu, {\bf G_0}{\bf H} {\bf G}_0 -{\bf G}_0, J_0 d\mu \rangle.  
\end{split}
\end{equation*}
The claim follows. \end{proof}
{\tiny
\begin{exercise}[\bf Abreu's boundary conditions~\cite{abreu}]\label{abreu2} Show that \eqref{a1} and \eqref{a2} are equivalent to
\begin{equation}\label{a3}
\ \ \ \  \ \ \ {\bf G} - {\bf G}_0 \ \ is \  \ smooth \ \ on  \ \ \Delta, 
\end{equation}
\begin{equation}\label{a4}
{\bf G}_0^{-1}{\bf G} \ \ is \  \ smooth \ \  and \ \ nondegenerate \ \ on  \ \ \Delta.
\end{equation}
\end{exercise}}

{\tiny  \begin{exercise}\label{e4}\cite{abreu} Show that the conditions \eqref{a3}-\eqref{a4} are equivalent to
 \begin{equation}\label{a5}
 u- \frac{1}{2}\sum_{j=1}^d L_j \log L_j   \ is   \ smooth  \ on \ \Delta.
 \end{equation}
 \begin{equation}
 {\rm det} ({\rm Hess}(u)) \times \Big(\prod_{j=1}^d L_j(x)\Big) \ is \ positive \ and \ smooth  \ on \ \Delta.
 \end{equation}
 
 \end{exercise}}
 
\bigskip
There is, however,  a subtle point in the above theory (which I believe  is often neglected in the literature): in order to apply  the sufficient conditions \eqref{a1}-\eqref{a2} or \eqref{a3}-\eqref{a4}, we need to use the angular coordinates defined by the initial metric $g_0$. This is the main difficulty to show that the conditions are also necessary. The following criterion is established in \cite{ACGT2}.
\begin{prop}\label{ACGT-boundary} A smooth,  positive definite $S^2(\tor^*)$-valued function ${\bf H}$ on $\Delta^0$ corresponds via \eqref{g}  to a $\T$-invariant $\omega$-compatible almost-K\"ahler structure on $M$ if and only if ${\bf H}$ satisfies the following conditions.
\begin{enumerate}
\item[$\bullet$] {\rm [smoothness] }${\bf H}$ has a smooth extension as a $S^2(\tor^*)$-valued function on $\Delta$;
\item[$\bullet$] {\rm [boundary conditions]} If $x$ belongs to a co-dimension one face $F_j \subset \Delta$ with normal $u_j \in \tor$, then
$${\bf H}_x(u_j, \cdot) =0, \ \ \ d{\bf H}_x(u_j,u_j) = 2u_j.$$
\item[$\bullet$] {\rm [positivity]}  Let $F^0\subset \Delta$ be the interior of a face of $\Delta$ and denote by $\tor_F\subset \tor$ the subspace spanned by the normals of all labels vanishing on $F^0$. Then,  the restriction of ${\bf H}$ to $F^0$, viewed  as a $S^2(\tor_F^0)$-valued function for  $\tor_F^0 ={\rm  ann}(\tor_F) \subset \tor^*$ is positive definite.
\end{enumerate}
\end{prop}
{\tiny
\begin{exercise} State and prove Proposition~\ref{ACGT-boundary} in the case $m=1$.
\end{exercise}}

\begin{lemma}\label{l:angular-coordinates} Suppose $g, g'$ are two $\T$-invariant, $\omega$-compatible K\"ahler structures on  a compact toric symplectic manifold $(M, \omega,\T)$ which induce the same $S^2(\tor)$-valued function ${\bf G}$ on $\Delta^0$. Then $g$ and $g'$ are isometric under a $\T$-equivariant  symplectomorphism.
\end{lemma}
\begin{proof} Fix $p_0\in M^0$ and let ${\bf t}$ (resp. ${\bf t}'$) be the angular coordinates determined by $g$ (resp. $g'$). The map $\Psi^0$ which sends $\bf t$ to ${\bf t}'$ and leaves $x$ unchanged defines a $\T$-equivariant symplectomorphism on $M^0,$ which sends $g$ to $g'$.  As $(M, g)$ is complete, we can extend $\Psi^0$ to an isometry of  the metric spaces $(M, g)$ and $(M, g')$, so to a riemannian isometry by a well-known result (see e.g. \cite{KN}). \end{proof}

\begin{defn}\label{d:symplectic-potential-space} For any compact convex labelled polytope $(\Delta, {\bf L})$, we denote by $\cS(\Delta, {\bf L})$ the space of smooth strictly convex functions  $u$ defined on the interior of $\Delta^0$,  such that  ${\bf H}= {\rm Hess}(u)^{-1}$ satisfies the conditions of Proposition~\ref{ACGT-boundary}, or equivalently, the conditions \eqref{a3}-\eqref{a4}.
\end{defn}
Combining Lemma~\ref{abreu1}, Exercise~\ref{abreu2}, Proposition~\ref{ACGT-boundary} and Lemma~\ref{l:angular-coordinates}, we obtain the following key result

\begin{thm}[\bf Abreu~\cite{abreu0, abreu}] \label{thm:symplectic-potential-space} There exists a bijection between the space of $\T$-invariant,  $\omega$-compatible K\"ahler structures $(g, J)$ on $(M, \omega, \T)$, modulo the action of the group  of $\T$-equivariant  symplectomorphisms,  and the space $\cS(\Delta, {\bf L})$ modulo the additive action of the space of affine-linear functions.
\end{thm}

One can   further refine Theorem~\ref{thm:symplectic-potential-space}.
\begin{prop}[\bf Donaldson~\cite{Donaldson2}]\label{p:boundary-donaldson} The functional space $\cS(\Delta, {\bf L})$ of symplectic potentials of globally defined $\T$-invariant, $\omega$-compatible K\"ahler metrics on $(M, \omega, \T$) can be equivalently defined as the sub-space of the space $\cC(\Delta)$ of convex continious functions on $\Delta$, such that 
\begin{enumerate}
\item[$\bullet$] {\rm [convexity]} The restriction of $u$ to the interior of any face of $\Delta$ (including $\Delta^0$) is a smooth strictly convex function.
\item[$\bullet$] {\rm [asymptotic behaviour]} $u- \frac{1}{2}\sum_{j=1}^d L_j \log L_j$ is  smooth on  $\Delta$. 
\end{enumerate}
 \end{prop}

\chapter{The Calabi Problem and Donaldson's theory}

\section{The Calabi problem on a toric manifold} In \cite{cal},  Calabi introduced the notion of an extremal K\"ahler metric $(g, \omega)$ on a complex manifold $(M,J)$ 
\begin{defn}[\bf Calabi~\cite{cal}] A K\"ahler metric $(g, \omega)$ on a complex manifold $(M, J)$  is called {\it extremal} if the  $\omega$ hamiltonian vector field $X_g:= \omega^{-1} (ds_g)$ preserves the complex structure $J$, i.e. ${\mathcal L}_{X_g} J =0$. 
\end{defn}
Constant scalar curvature (CSC) K\"ahler metrics are examples, and, motivated by the Uniformization Theorem for complex curves, Calabi  proposed the problem of finding an extremal K\"ahler metric $(g, \omega)$ in a given deRham  class $[\omega]\in H^2(M, \R)$. This is known as the {\it Calabi Problem}, and is of greatest interest in current research in K\"ahler geometry.   

We now turn to the toric case, and observe the following ramification  of the Calabi problem.
\begin{lemma}\label{extremal-toric} Suppose $(g, J)$ is a $\T$-invariant, $\omega$-compatible K\"ahler metric on $(M,\omega, \T),$ corresponding to a symplectic potential $u \in \cS(\Delta, {\bf L})$. Then $g$ is extremal iff $s_g$ is the pull back by the momentum map of an affine-linear function $s(x)=\langle \xi, x\rangle + \lambda$  on $\Delta$.
\end{lemma}
\begin{proof} As $g$ is $\T$ invariant, its scalar curvature is a $\T$-invariant function on $M$, whence is the pull back by the momentum map of a smooth function $s(x)$ on $\Delta$.
It thus follows from \eqref{omega} that on $M$, $X_g= \omega^{-1} (ds)=  \sum_{i} s_{,i} K_i$. As each $K_i$ preserves $J$, the condition $\mathcal{L}_{X_g} J=0$ reads as
\begin{equation*}
\begin{split}
0=JK_j \cdot s_{,i} = (d s_{,i})(JK_j) &=-\sum_{k} s_{,ik} Jdx_k(K_j) =-\sum_{k,\ell} s_{,ik}H_{k\ell} \theta_{\ell}(K_j) \\
                                                               &=-\sum_{k} s_{,ik} H_{kj} 
                                                               \end{split}
                                                               \end{equation*}
                                                               As ${\bf H}$ is non-degenerate on $\Delta^0$, it follows that $s_{,ik}=0$, i.e. $s(x)$ must be an affine-linear function on $\Delta^0$, and hence on $\Delta$.
                                                               
                                                               Conversely, for any affine-linear function $s(x)=\langle \xi, x\rangle + \lambda$, $\omega^{-1} ds=X_{\xi}$, which preserves $J$. \end{proof} 
 
A key observation of Donaldson~\cite{Do2} is that the affine-linear functions $s(x)$ in Lemma~\ref{extremal-toric} is in fact predetermined from the labelled polytope 
$(\Delta, {\bf L})$.  To state the precise result, we need to introduce measures on $\Delta$ and $\partial \Delta$. The standard Lebesgue measure $dv = dx_1\wedge \ldots \wedge dx_m$ on $\tor^* \cong \R^m$ and the labels ${\bf L}=(L^1, \ldots, L_d)$ of $\Delta$ induce a measure $d\sigma$ on each facet $F_i \subset \partial \Delta$ by letting
\begin{equation}\label{boundary-measure}
dL_i \wedge d\sigma = u_i\wedge d\sigma = - dv.
\end{equation}
\begin{prop}[\bf Donaldson~\cite{Do2}] \label{p:Abreu-equation}Suppose $(\Delta, {\bf L})$ is a compact convex simple labelled polytope in $\tor^*$. Then, there exists a unique affine-linear function $s_{(\Delta, {\bf L})}$ on $\tor^*$, called {\rm the extremal affine-linear function} of $(\Delta, {\bf L})$ such that for any affine-linear function $f$
\begin{equation}\label{affine-extremal}
 2 \int_{\partial \Delta} f d\sigma - \int_{\Delta} s_{(\Delta, {\bf L})} f dv  =0.
\end{equation}
Furthermore, if for $u \in \cS(\Delta, {\bf L})$ the metric \eqref{g} is extremal, i.e.  satisfies 
$$-\sum_{i,j=1}^m \big(({\rm Hess}(u))^{-1}\big)_{ij,ij}= s(x)= \langle \xi, x\rangle + \lambda,$$
then the affine-linear function $s(x)$ must be equal to $s_{(\Delta, {\bf L})}$.
\end{prop}
Before we give the proof of this important result, we start with  a technical observation.
\begin{lemma}\label{by-parts}  Let ${\bf H}$ be any smooth $S^2\tor^*$-valued function on $\Delta$ which satisfies the boundary conditions  of Proposition~\ref{ACGT-boundary}, but not necessarily the positivity condition. Then, for any smooth function $\varphi$ on $\Delta$
\begin{equation}\label{Futaki-Hessian}
\int_{\Delta} \Big(\sum_{i,j=1}^m H_{ij,ij}\Big) \varphi \ dv =\int_{\Delta} \Big(\sum_{i,j=1}^m {H}_{ij}\varphi_{,ij} \Big)dv\\
-2\int_{\partial \Delta} \varphi d\sigma .\end{equation}
\end{lemma}
\begin{proof} The proof is elementary and uses integration by parts: recall that for any smooth $\tor^*$-valued function $V=(V_1, \ldots, V_m)$ on $\tor^*$, Stokes theorem gives
\begin{equation}\label{integration-by-parts}
\int_{\Delta} \sum_{j=1}^m V_{j,j} dv = -\sum_{k=1}^d \int_{F_k} \langle V, dL_k \rangle d\sigma,
\end{equation}
where we have used the convention \eqref{boundary-measure} for $d\sigma$.
We shall use the identity
\begin{equation}\label{basic-formula}
\sum_{i,j=1}^m \varphi_{,ij} {H}_{ij} = \sum_{i,j=1}^m \varphi H_{ij,ij} - \sum_{j=1}^m V_{j,j} ,
\end{equation}
where
\begin{equation}
V_j:=\varphi \sum_{i=1}^m H_{ij,i} - \sum_{i=1}^m \varphi_{,i} H_{ij}.
\end{equation}
It follows by \eqref{basic-formula} and \eqref{integration-by-parts} that
\begin{equation}\label{stokes}
\int_{\Delta} \sum_{i,j=1}^m \varphi_{,ij} H_{ij} dv = \int_{\Delta} \sum_{i,j=1}^m \varphi H_{ij,ij} + \sum_{k=1}^{d}\int_{F_k}\langle V, dL_k\rangle d\sigma.
\end{equation}
On each facet $F_k$ we have, using a basis $\{e_1,\ldots e_m\}$ of $\tor$,
\begin{equation}\label{detailed}
\begin{split}
\langle dL_k, V\rangle  &= \sum_{j=1}^m \langle dL_k, e_j^*\rangle V_j \\
                                          &= \sum_{j=1}^m \langle dL_k, e_j^*\rangle (\varphi \sum_{i=1}^m H_{ij,i} - \sum_{i=1}^m \varphi_{,i}H_{ij})\\
                                          &= \varphi \sum_{i=1}^m \langle d{\bf H}(dL_k, e_i), e_i^*\rangle - \sum_{i=1}^m {\bf H}(dL_k, d\varphi).
\end{split}
\end{equation}
Using the boundary conditions of Proposition~\ref{ACGT-boundary}, we have  ${\bf H}(dL_k, \cdot) =0$ on $F_k$ and (by choosing a basis $e_1 = dL_k, e_2, \ldots, e_m$ with $e_j^*$ tangent to $F_k$ for $j>1$) 
\begin{equation*}
\sum_{i=1}^m  \langle d {\bf H}(dL_k, e_i), e_i^*\rangle  =2 .
\end{equation*}
Substituting back in \eqref{detailed} and \eqref{stokes} completes the proof. \end{proof}  
\begin{rem} Taking $\varphi$ to be an affine-linear function, Lemma~\ref{by-parts} shows that the $L^2$ projection to the space of affine-linear functions of  the expression $-\sum_{i,j=1}^m H_{ij,ij}$ is independent of $H$. When $H_{ij}= H^{u}_{ij}$ corresponds to a toric K\"ahler metric, this observation yields that the $L^2$ projection of the scalar curvature to the space of affine-linear functions is independent of the symplectic potential $u$. \end{rem}                                                           
\begin{proof}[Proof of Proposition~\ref{p:Abreu-equation}]
Writing  $$s_{(\Delta, {\bf L})} = a_0 + \sum_{j=1}^m a_j x_j,$$  the condition \eqref{affine-extremal} gives rise to a linear system  with  positive-definite symmetric matrix
\begin{equation}\label{extremal-affine-system}
\begin{split}
a_0 \int_{\Delta} x_i dv + \sum_{j=1}^m a_j \int_{\Delta} x_j x_i dv &= 2\int_{\partial \Delta} x_i d\sigma \\
a_0  \int_{\Delta} dv + \sum_{j=1}^m a_j \int_{\Delta} x_j dv  \ \ \ \  \ &=  2\int_{\partial \Delta}  d\sigma,
\end{split}
\end{equation}
which therefore determines $(a_0, \ldots, a_m)$ uniquely.

We now suppose ${\bf H}=(H_{ij})$ is a smooth $S^2(\tor^*)$-valued function on $\Delta$ which satisfies the boundary conditions of Proposition~\ref{ACGT-boundary} and 
\begin{equation}\label{matrix-extremal}
-\sum_{i,j=1}^m H_{ij,ij} = s(x)= \langle \xi, x \rangle + \lambda.
\end{equation}
We are going to show that $s(x)= s_{(\Delta, {\bf L})}$. Notice that  to this end we do not assume that ${\bf H}$ satisfies  the positivity conditions of Proposition~\ref{ACGT-boundary} nor  that ${\bf H}= {\rm Hess}(u)^{-1}$ for some $u\in \cS(\Delta, {\bf L})$. Indeed, by Lemma~\ref{by-parts} applied to an affine-linear function $\varphi=f$, we get that $s(x)$ satisfies the defining property \eqref{affine-extremal}. \end{proof}

\begin{defn} Let $(\Delta, {\bf L})$ be a labelled compact convex simple polytope in $\R^m$, $\cS(\Delta, {\bf L})$ the space of strictly convex smooth function on $\Delta^0$ satisfying the conditions of Proposition~\ref{ACGT-boundary} and $s_{(\Delta, {\bf L})}$ the extremal affine function of $(\Delta, {\bf L})$. Then the non-linear PDE
\begin{equation}\label{Abreu-equation}
s(u):=-\sum_{i,j=1}^m (u^{,ij})_{,ij} = s_{(\Delta, {\bf L})}
\end{equation}
is called the {\it Abreu equation}.  If $(\Delta, {\bf L}, \Lambda)$ is a Delzant triple, solutions of   \eqref{Abreu-equation} correspond to extremal $\T$-invariant, $\omega$-compatible K\"ahler metrics on the toric symplectic manifold $(M, \omega, \T)$.
\end{defn}

\section{Donaldson--Futaki invariant}
We now introduce  an  obstruction  to the existence of solutions of \eqref{Abreu-equation},  due to Donaldson~\cite{Do2}.
\begin{defn}[\bf Donaldson-Futaki invariant] Given a labelled compact convex simple polytope $(\Delta, {\bf L})$ in $\R^m$, we define the functional
\begin{equation}\label{DF}
{\mathcal F}_{(\Delta, {\bf L})} (\varphi) := 2\int_{\partial \Delta} \varphi d\sigma -  \int_{\Delta} s_{(\Delta, {\bf L})}  \varphi dv
\end{equation}
acting on the space of continuous functions on $\Delta$. $\mathcal{F}_{(\Delta, {\bf L})}$ is called the Donaldson--Futaki invariant associated to $(\Delta, {\bf L})$
\end{defn}

\begin{prop}[\bf Donaldson~\cite{Do2}]\label{p:semistable} If $(\Delta, {\bf L})$ admits a solution of \eqref{Abreu-equation}, then
$${\mathcal F}_{(\Delta, {\bf L})} (\varphi) >0$$
for any smooth convex function $\varphi$ on $\Delta$ which is not affine-linear.
\end{prop}
\begin{proof}
Using Lemma~\ref{by-parts} we find
$${\mathcal F}_{(\Delta, {\bf L})} (\varphi) = \int_{\Delta} \sum_{i,j=1}^m H_{ij} \varphi_{,ij} dv \ge 0$$
where we have used the convexity of $\varphi$ for the inequality. Furthermore, as ${\bf H}>0$ on $\Delta^0$, the inequality is strict unless $\varphi_{,ij}=0$, i.e. $\varphi$ is affine-linear. \end{proof}

{\tiny
\begin{exercise}[\bf Donaldson~\cite{Do2}] Show that the statement of Proposition~\ref{p:semistable} holds true for continuous convex functions $\varphi$ which are smooth on the interior $\Delta^0$ and are not affine-linear.
\end{exercise}}

\section{Uniqueness} We show now that the solution of \eqref{Abreu-equation} is unique up to  the addition of affine-linear functions.
\begin{thm}[\bf D. Guan~\cite{Guan}]\label{uniqueness} Any two solutions $u_1, u_2\in \cS(\Delta, {\bf L})$ of \eqref{Abreu-equation} differ by an affine-linear function. In particular, on a compact toric K\"ahler manifold or orbifold $(M, \omega, \T)$, there exists at most one, up to a $\T$-equivariant isometry, $\omega$-compatible $\T$-invariant extremal K\"ahler metric  $(g, J)$.
\end{thm}
\begin{proof}  Consider the functional
\begin{equation}\label{K-energy}
\mathcal{E}_{(\Delta, {\bf L})}(u) := \mathcal{F}_{(\Delta, {\bf L})}(u)  - \int_{\Delta}\Big(\log \det \Hess (u)- \log \det \Hess (u_0)\Big) \, dv,
\end{equation}
 referred to as the {\it relative $K$-energy} of $(\Delta, {\bf L})$ with respect to $\T$. It is well-defined for elements  $u \in \cS(\Delta, {\bf L})$ by virtue of the equivalent boundary conditions \eqref{a3}-\eqref{a4}. Using the formula $d \log \det {\bf A} = \trace {\bf A}^{-1} d {\bf A}$ for any non-degenerate matrix ${\bf A}$ and  Lemma~\ref{by-parts}, one computes the first variation of  $\mathcal{E}_{(\Delta, {\bf L})}$ at $u$ in the direction of $\dot u$
\begin{equation*}
\begin{split}
\Big(d \mathcal{E}_{(\Delta, {\bf L})}\Big)_{u} (\dot{u})=&  \mathcal{F}_{(\Delta, {\bf L})}(\dot{u}) -\int_{\Delta}\sum_{i,j=1}^m {H}^{u}_{ij} \dot{u}_{,ij} dv \\
                                                                               =& \int_{\Delta}\Big[ \Big( -\sum_{i,j=1}^m {H}^u_{ij}\Big)_{,ij}\Big) - s_{(\Delta, {\bf L})}\Big]  \dot{u} \ dv,
\end{split}
\end{equation*}
showing that the critical points  of $\mathcal{E}_{(\Delta, {\bf L})}$ are precisely the solutions of \eqref{Abreu-equation}. Furthermore, using $d{\bf A}^{-1} = -{\bf A}^{-1} d{\bf A} {\bf A}^{-1}$, the second variation of $\mathcal{E}_{(\Delta, {\bf L})}$ at $u$ in the directions of $\dot{u}$ and $\dot{v}$ is computed to be
\begin{equation*}
\Big(d^2 \mathcal{E}_{(\Delta, {\bf L})}\Big)_{u}(\dot{u}, \dot{v}) =\int_{\Delta} \trace \Big(\big(\Hess(u)\big)^{-1} \Hess(\dot{u}) \big(\Hess(u)\big)^{-1} \Hess(\dot{v})\Big) dv,
\end{equation*}
showing  that $\mathcal{E}_{(\Delta, {\bf L})}$ is convex. In fact, as $\Hess(u)$ is positive definite and $\Hess(\dot{u})$ is symmetric, the vanishing  $\Big(d^2 \mathcal{E}_{\Delta, {\bf L}, f}\Big)_{u}(\dot{u}, \dot{u}) = 0$ is equivalent to $\dot{u}$
being affine-linear.

It follows from \eqref{a3}-\eqref{a4} that for any two elements $u_1, u_2 \in {\cS}(\Delta, {\bf L})$, $u(t) = tu_1 + (1-t)u_2, \ t \in [0,1]$ is a curve in ${\cS}(\Delta, {\bf L})$ with tangent vector $\dot{u}= u_1-u_2$. Using the convexity of $\mathcal{E}_{(\Delta, {\bf L})}$, it follows that that if $u_1$ and $u_2$ are two solutions of \eqref{Abreu-equation} (equivalently, $u_1$ and $u_2$ are critical points of $\mathcal{E}_{(\Delta, {\bf L})}$), then $u_1-u_2$ must be affine-linear. \end{proof}
\begin{cor}\label{bounded-mabuchi} If \eqref{Abreu-equation} admits a solution $u^* \in \cS(\Delta, {\bf L})$,  then the relative K-energy  $\mathcal{E}_{(\Delta, {\bf L})}$  atteints its minimum at $u^*$.
\end{cor}
\begin{proof} The arguments in the proof of Theorem~\ref{uniqueness} show that $\mathcal{E}_{(\Delta, {\bf L})}$ is convex on $\cS(\Delta, {\bf L})$.  The solution $u^*$ being a critical point of $\mathcal{E}_{(\Delta, {\bf L})}$, it is  therefore a global minima. \end{proof}

\section{K-stability} 
\begin{defn}[\bf Toric K-stability] \label{K-stable} We say that a labelled compact convex simple polytope $(\Delta, {\bf L})$ in $\R^m$ is $K$-{\it semistable} if 
\begin{equation}\label{K-semistable}
{\mathcal F}_{(\Delta, {\bf L})} (\varphi) \ge 0
\end{equation}
for any convex piecewise affine-linear (PL) function $\varphi={\max}(f_1, \ldots, f_k)$ on $\Delta$.  The labelled polytope $(\Delta, {\bf L})$ is $K$-{\it stable}, if moreover, equality in \eqref{K-semistable} is achieved only for the affine-linear functions $\varphi$.
If $(\Delta, {\bf L}, \Lambda)$ is a Delzant triple, we say that the corresponding toric symplectic manifold$(M, \omega, \T)$ is $K$-stable iff $(\Delta, {\bf L})$ is.
\end{defn}
An immediate corollary of Proposition~\ref{p:semistable} is the following 
\begin{cor}[\bf Donaldson~\cite{Do2}]\label{cor:K-stability} If \eqref{Abreu-equation} admits solution, then $(\Delta, {\bf L})$ is K-semistable.
\end{cor}

This  can be  improved to

\begin{thm}[\bf Zhou--Zhu~\cite{ZZ}] \label{c:K-stability} If \eqref{Abreu-equation}  admits a solution $u \in {\cS}(\Delta, {\bf L})$ then $(\Delta, {\bf L})$ is K-stable.
\end{thm}
\begin{proof} The proof is elementary and uses integration by parts (similar to the proof of Lemma~\ref{by-parts}) in order to obtain a distribution analogue of \eqref{Futaki-Hessian} in the case when $\varphi$ is a PL convex function.  For simplicity we  will check that if \eqref{Abreu-equation} admits a solution $u \in {S}(\Delta, {\bf L})$,  then ${\mathcal F}_{(\Delta, {\bf L})}$ is strictly positive for a {\it simple} convex PL function $\varphi$, i.e. $\varphi = {\rm max}(L, \tilde{L})$ where $L$ and $\tilde{L}$  are affine-linear functions on $\tor^*$  with $L- \tilde{L}$ vanishing in the interior of $\Delta$; as $\mathcal{F}_{(\Delta, {\bf L})}(\tilde{L})=0$ by \eqref{affine-extremal} we can assume without loss of generality that $\tilde L \equiv 0$.

Denote by $F= \Delta \cap \{L=0\}$ the `crease'  of $\varphi$. This introduces a partition $\Delta= \Delta' \cup \Delta''$ of $\Delta$ into $2$ sub-polytopes  with $F$ being a common facet of the two. Notice that $\varphi$ is affine-linear over each $\Delta'$ and $\Delta''$, being zero over $\Delta''$, say. Furthermore, $dL$ defines an inward normal for $\Delta'$ (by convexity) and we use \eqref{boundary-measure} to define a measure $d\sigma$ on $\partial \Delta'$.  Let us write $\partial \Delta' = F \cup \partial \Delta_0$, where $\partial \Delta_0$ is the union of facets of $\Delta'$ which belong to $\partial \Delta$. We then have (using that $\varphi \equiv 0$ on $\Delta''$)
\begin{equation*}
\begin{split}
\mathcal{F}_{(\Delta, {\bf L})} (\varphi) = &\, 2\int_{\partial \Delta} \varphi d\sigma -\int_{\Delta} s_{(\Delta, {\bf L})} \varphi dv \\
= &\, 2\int_{\partial \Delta_0}\varphi d\sigma - \int_{\Delta'} s_{(\Delta, {\bf L})} \varphi dv.
  \end{split}
 \end{equation*}
We now use \eqref{stokes}  and \eqref{detailed} over $\Delta'$ with ${H}_{ij}$ satisfying  the conditions of Proposition~\ref{ACGT-boundary} on $\Delta$  and $-\sum_{ij} H_{ij,ij}= s_{(\Delta, {\bf L})}$.  Noting that $\varphi$ is affine-linear over $\Delta'$, i.e. $\varphi_{,ij}=0$ we have
\begin{equation*}
\begin{split}
- \int_{\Delta'} s_{(\Delta, {\bf L})}\varphi  dv = & \int_{\Delta'} \varphi \Big(\sum_{i,j=1}^m {H}_{ij,ij}\Big) dv \\
                                                                                                                     = &-2\int_{\partial \Delta_0}\varphi d\sigma\\
                                                                                                                        &- \int_{F}\Big(\varphi \sum_{i=1}^m \langle d{\bf H}(dL, e_i), e_i^*\rangle\Big) d\sigma \\
                                                                                                                          & + \int_{F} {\bf H}(dL, dL) d\sigma.
\end{split}
\end{equation*}
As $\varphi$ vanishes on $F$, the term at the third line is zero, so that
\begin{equation}\label{Futaki-segment}
\mathcal{F}_{(\Delta, {\bf L})} (\varphi) =  \int_{F}  {\bf H}(dL, dL) d\sigma >0,
\end{equation}
as ${\bf H}$ is positive definite over $\Delta^0$. \end{proof}

Thus motivated, the central conjecture is

\begin{conj}[\bf Donaldson~\cite{Do2}] \eqref{Abreu-equation} admits a  solution in $\cS(\Delta, {\bf L})$ if and only if $(\Delta, {\bf L})$ is $K$-stable.
\end{conj}

\section{Toric test configurations} In this section we shall explain the notion of {\it toric test configuration}  which gives a geometrical meaning of the convex PL functions appearing in the definition of K-stability. 

We suppose that  $(\Delta, {\bf L}) \subset \R^m$ is a Delzant polytope with respect to (the dual of)  the standard lattice $\Z^m \subset \R^m$, corresponding to a symplectic toric manifold $(M, \omega, \T^m)$,   and $f={\rm max}(f_1, \ldots, f_p)$ is a convex PL function over $\Delta$, with $f_1, \ldots, f_p$ being a minimal set of affine-linear functions with {\it rational coefficients} defining $f$. Multiplying  $f$ by a suitable denominator, we can assume without loss that all $f_i$'s  have integer coefficients. We choose $R>0$ such that $(R- f)$ is strictly positive on $\Delta$,  and consider the labelled polytope in ${\rm P} \subset \R^{m+1}= \R^m \times \R$,  defined by 
\begin{equation}\label{configuration-polytope}
{\rm P} = \Big\{ (x, x_{m+1}) \in \R^{m+1} : x\in \Delta, 0 \le x_{m+1} \le (R- f(x))  \Big\}.
\end{equation}
We further assume that the polytope ${\rm P}$   is  rational Delzant~\footnote{This places a condition on $f$ but the set of convex  rational PL functions defining rational Delzant polytopes is dense in the $C^0(\Delta)$ topology.}  with respect to the labels
$$\Big\{ L_j(x) \ge 0,  \ x_{m+1} \ge 0, \   (R - x_{m+1} -f_i(x)) \ge 0, \    j=1, \ldots, d, \ i=1, \ldots, p\Big\},$$ 
where ${\bf L}=(L_1(x), \ldots, L_{d}(x))$ are the labels of $\Delta$,  and  (the dual of)  the lattice $\Z^{m+1} \subset \R^{m+1}$. For simplicity, we shall assume that ${\rm P}$ is Delzant, i.e. that it gives rise to a {\it smooth} toric symplectic $2(m+1)$-dimensional manifold $(\mathcal{M}, \Omega)$,  but the discussion below  holds  in the general orbifold case.  

We denote by $\T^{m+1} = \T^m \times \Sph^1_{(m+1)}$ the corresponding torus, with $\T^m$ identified with the torus acting on $M$. Notice that $\Delta$ is a facet of ${\rm P}$, whose pre-image  is  a  smooth submanifold $\tilde M\subset \mathcal{M}$.   The Delzant construction identifies the stabilizer of points in $\tilde M^0$ with the circle subgroup $\Sph^1_{(m+1)} \subset \T^{m+1}$,  so that   with  respect to the induced action of $\T^{m+1}/\Sph^1_{(m+1)} \cong \T^m$,   $(\tilde M,  \Omega_{|_{\tilde M}})$, is equivariantly isomorphic to $(M, \omega)$ by Delzant's theorem. We shall thus assume without loss of generality that $M=\tilde M$ and  $\omega=\Omega_{|_{\tilde M}}$.  

Let us now choose an $\Omega$-compatible $\T^{m+1}$-invariant complex structure ${\mathcal J}$ on $\mathcal{M}$, which  induces a $\T^m$-invariant   $\omega$-compatible  complex structure  $J$ on $(M, \omega)$.   Donaldson~\cite[Proposition~4.2.1]{Do2} shows that with respect to the $\C^*$-action  $\rho : \C^* \to {\rm Aut}(\mathcal M)$ induced by $\Sph^1_{(m+1)}$,  the complex $(m+1)$-dimensional manifold  $(\mathcal M, \Omega, \rho)$ is an example of a {\it K\"ahler test configuration}  associated to the K\"ahler manifold $(M, \omega)$, meaning  that the following are satisfied:
\begin{defn}[\bf K\"ahler test configurations]\label{d:kahler-test-configuration} A  smooth K\.ahler test configuration associated to a compact complex $m$-dimensional K\"ahler manifold $(M, \omega)$ is a compact complex $(m+1)$-dimensional K\.ahler manifold $(\mathcal M, \Omega)$ endowed with a $\C^*$-action $\rho$, such that 
\begin{enumerate}
\item[$\bullet$] there is  a surjective holomorphic map $\pi : \mathcal M \to \C P^1$ such that each fibre $M_{\tau} :=\pi^{-1}(\tau) \cong M$ for $\tau \in \C P^1 \setminus \{0\}$;
\item[$\bullet$] the $\C^*$-action $\rho$ on $\mathcal M$ is equivariant with respect to the standard $\C^*$-action on $\C P^1$ fixing $0$ and $\infty$;
\item[$\bullet$] there is a $\C^*$-equivariant biholomorphism $\alpha : \mathcal M \setminus M_0 \to M \times (\C P^1 \setminus \{ 0\})$ with respect to the $\C^*$-action $\rho$ and the standard action of $\C^*$ on the second factor  of the product.
\item[$\bullet$]  $\mathcal M$ is endowed with a K\"ahler metric $\Omega$, which restricts to the  K\"ahler metric $\omega$  on $M \cong M_1$.
\end{enumerate}
\end{defn}
To see  the existence of such equivariant $\pi$ in our toric case,  it is instructive to  think of the blow down map $b_1$  plotted  in Figure~2.  The Delzant polytope at the lhs  is associated to a PL convex function over the  interval whereas  the Delzant polytope at the rhs represents the product $\C P^1 \times \C P^1$. It this case,  $M=\C P^1$ is one of the factors of the product,  and the projection to the other factor $\C P^1$  defines the map $\pi$. Notice that for this example,  $\pi^{-1}(0) = \C P^1 \cup E$ is the union of  a copy of $M$ with the exceptional divisor of the blow-up, which  is also the pre-image of the facets of the polytope satisfying $(R-f)=0$.
\begin{defn}[\bf Toric test configuration]  Let $(M, \omega)$ be a toric K\"ahler manifold with labelled Delzant polytope $(\Delta, {\bf L})$ in $\R^m$ with respect to the lattice $\Z^m$. A  K\"ahler test configuration $(\mathcal M, \Omega, \rho$) for $(M,\omega)$ obtained from a rational PL convex  function $f$ as above is called {\it toric test configuration}. \end{defn}

We shall now express the Donaldson--Futaki invariant \eqref{DF}  of  the PL function $f$ defining  a toric test configuration  in terms of  a differential-geometric quantity on $\mathcal M$.  To simplify notation, we shall denote by ${\rm Scal}(\Omega)$ the scalar curvature of the corresponding K\"ahler metric  on $\mathcal{M}$ (defined by $(\Omega, J)$).  We  also notice that the momentum map of the sub-torus $\T^m \subset \T^{m+1}$ sends $\mathcal M$ onto $\Delta$ (which is the projection  of ${\rm P}$ to $\R^m$) and,   with a slight abuse of notation,  we shall denote by $s_{(\Delta, {\bf L})}$  the smooth function on $\mathcal M$ obtained by pulling back  the extremal affine-linear function  $s_{(\Delta, {\bf L})}$  associated to $(M, \omega)$. We then have 
\begin{lemma}\label{l:Kahler-Futaki}\label{l:KF} In the above setting, the Donaldson--Futaki invariant~\eqref{DF} associated to the convex PL function $f$ is given by
\begin{equation}\label{futaki-toric}
\begin{split}
(2\pi)^{m+1}\cF_{(\Delta, {\bf L})} (f)  = &  -\int_{\mathcal M} \Big({\rm Scal}(\Omega) - s_{(\Delta, {\bf L})}\Big) \frac{\Omega^{m+1}}{(m+1)!} \\
                                                        & + (8\pi) \int_M\frac{\omega^{m}}{m!}.
                                                      \end{split}
                                                      \end{equation}  
\end{lemma}
\begin{proof}  We write $(x, x') \in \R^m \times \R$ for the linear  coordinates on $\R^{m+1}$ and   let $dx$ and $dx\wedge dx'$ denote the standard Lebesgue measures on $\R^m$ and $\R^{m+1}$, respectively.  In what follows, ${\mathrm P}$ (resp. $\partial {\rm P})$ denotes the Delzant polytope (resp. its boundary) in $\R^{m} \times \R$ corresponding tp $(\mathcal M, \Omega, \T^{m+1})$, $\Delta \subset {\rm P}$ being the facet corresponding to  $x'=0$. Using the description \eqref{omega} of the K\"ahler forms  $\omega$ and $\Omega$ in momentum/angular coordinates, we get
\begin{equation}\label{a}
 (8\pi) \int_M \frac{\omega^m}{m!} = 4(2\pi)^{m+1} \int_{\Delta} dx
 \end{equation}
 \begin{equation}\label{b}
 \begin{split}
 \int_{\mathcal M}s_{(\Delta, {\bf L})} \frac{\Omega^{m+1}}{(m+1)!}  &= (2\pi)^{m+1} \int_{\rm P} s_{(\Delta, {\bf L})}(x) dx \wedge dx' \\
         &=(2\pi)^{m+1}  \int_{\Delta} s_{(\Delta, {\bf L})}(x) (R - f(x) dx.
         \end{split}
 \end{equation}
Using  Lemma~\ref{l:abreu} for ${\rm Scal}(\Omega)$  and \eqref{Futaki-Hessian} with $\varphi=1$ on ${\rm P}$, we compute
\begin{equation}\label{c}
\begin{split}
& \frac{1}{(2\pi)^{m+1}}\int_{\mathcal M} {\rm Scal}(\Omega) \frac{\Omega^{m+1}}{(m+1)!} = 2 \int_{\partial P} d\sigma_{\partial \rm P} \\
 &= 2\Big( \int_{\Delta} dx + \int_{(R-f)(\Delta)} d\sigma_{(R-f)(\Delta)} + \int_{\partial \Delta}(R-f) d \sigma_{\partial \Delta}\Big) \\
 &= 2 \Big( 2 \int_{\Delta} dx  + \int_{\partial \Delta} (R-f) d \sigma_{\partial\Delta}\Big),
 \end{split}
 \end{equation}
 where for passing from the second line to the third  we have used that $d\sigma_{(R-f)(\Delta)}$ is defined by the equality
 $$df \wedge d\sigma_{(R-f)(\Delta)} = dx\wedge dx'.$$
 Lemma~\ref{l:KF} now follows easily  by combining \eqref{a}, \eqref{b} and \eqref{c}.  \end{proof}
 
 \begin{rem} (a)  In the special case when $s_{(\Delta, {\bf L})} = \lambda$ is a constant, which means that  $(M, \omega, J)$ has vanishing Futaki invariant, one can further specify the expression  in the rhs of \eqref{futaki-toric} as follows. By \eqref{affine-extremal} with $f=1$ we  obtain  $\lambda = 2\frac{\int_{\partial \Delta} d\sigma}{\int_{\Delta} dv}$ whereas \eqref{Futaki-Hessian} yields  $\int_{\Delta} s(u) = 2\int_{\partial \Delta} d\sigma$. It thus follows
that 
\begin{equation*}
\begin{split}
\lambda & = \frac{\int_{\Delta} s(u) dv}{\int_{\Delta} dv}  \\
            &  = \frac{\int_M {\rm Scal}(\omega) \omega^{m}}{\int_M \omega^m}  \\
             &= 4\pi  m  \frac{c_1(M) \cdot [\omega]^{m-1}[M]}{[\omega]^m[M]},
             \end{split}
             \end{equation*}
where ${\rm Scal}(\omega)$ stands for the induced toric K\"ahler metric on $M$ and $u$ for its symplectic potential.  Substituting in \eqref{futaki-toric}, we obtain the following co-homological expression for the Donaldson--Futaki invariant
\begin{equation*}
\begin{split}
  \cF(\mathcal M, \Omega)  &:=(2\pi)^{m}\cF_{(\Delta, {\bf L})} (f)   \\
      &= - 2\Big[\Big(\frac{c_1(\mathcal M) \cdot [\Omega]^m[\mathcal M]}{m!}\Big) - \Big(\frac{c_1(M) \cdot [\omega]^{m-1}(M)}{(m-1)!}\Big) \Big(\frac{{\rm Vol}(\mathcal M, \Omega)}{{\rm Vol}(M, \omega)}\Big)\Big] \\
                                                         &+ 4{\rm Vol}(M, \omega).
                                                         \end{split}
                                                         \end{equation*}
This formula  makes sense for any K\"ahler test configuration (see Definition~\ref{d:kahler-test-configuration})  and  only depends upon the deRham classes $[\Omega]$ on $\mathcal M$ and $[\omega]$ on $M$. By Corollary~\ref{cor:K-stability},   $\cF(\mathcal M, \Omega) \ge 0$ on any toric K\"ahler test configuration $({\mathcal M}, \Omega)$ associated to $(M, \omega)$, with equality if and only if $\mathcal M$ corresponds to a single  rational  affine-linear function $f$. This is the original notion of {\it K-stability} going back to Tian~\cite{Tian1}. In this integral form,  the Futaki invariant of a test configuration was first used by Odaka~\cite{odaka1,odaka2}  and  Wang~\cite{wang}   to study   (possibly singular) polarized projective  test configurations. 
 
  (b) It turns out that the expression  at  the rhs of \eqref{futaki-toric}  makes sense for any $\T$-invariant K\"ahler test configuration associated to a K\"ahler manifold $(M, \omega)$ endowed with a maximal compact torus $\T$   in its reduced group of complex automorphisms, and it merely depends upon the deRham classes $[\omega]$ and $[\Omega]$ and the momentum image $\Delta$ of $M$ for that action of $\T$.  This leads to the notion of a $\T$-relative Donaldson--Futaki invariant $\cF^{\T}(\mathcal M, [\Omega])$ of a compatible test configuration.  This point of view has been taken and developed in \cite{DR,D, Z1,Z2,lahdili} for  a general  K\"ahler manifold, where an extension of Corollary~\ref{cor:K-stability} is  also obtained. 
 \end{rem}

\section{Uniform K-stability}
Let $\cC(\Delta)$ denote the set of continuous convex functions on $\Delta$ (continuity
follows from convexity on the interior of $\Delta$), $\cC_\infty(\Delta) \subset \cC(\Delta)$ the subset
of those functions which are smooth on the interior  $\Delta^0$, 
and $\cS(\Delta, {\bf L})\subset\cC_\infty(\Delta)$ the set of
symplectic potentials. Note that, by virtue of Proposition~\ref{p:boundary-donaldson},  if $u\in\cS(\Delta, {\bf L})$ and $f\in\cC_\infty(\Delta)$, with $f$
smooth on all of $\Delta$, then $u+ f\in\cS(\Delta, {\bf L})$. Conversely,  the difference of any two functions in  $\cS(\Delta, {\bf L})$ is a function in  $\cC_{\infty}(\Delta)$ which is smooth on $\Delta$.

\smallskip
The affine-linear functions act on $\cC(\Delta)$ and $\cC_\infty(\Delta)$ by translation.  Let $\cC^*(\Delta)$ be a slice for the action on $\cC(\Delta)$ which is  closed under positive linear combinations, and 
$\cC_\infty^*(\Delta)$  the induced slice   in $\cC_\infty(\Delta)$. Then any $f$ in $\cC(\Delta)$ can be written uniquely as $f=\pi(f)+g$, where $g$ is affine-linear and
$\pi(f)\in\cC^*(\Delta)$ for a linear projection $\pi$. Functions in $\cC^*(\Delta)$ are sometimes said to be {\it normalized}.
{\tiny \begin{ex}\label{e:normalization-donaldson}\cite{Do2} If $x_0 \in \Delta^0$ is  a fixed interior point,  a slice as above  is given by
\begin{equation*}
\cC^*(\Delta):=\Big\{f \in \cC(\Delta) : f(x) \ge f(x_0)=0\Big\}.
\end{equation*}
\end{ex}}
{\tiny \begin{ex}\label{e:normalization-gabor}\cite{Sz}  Another natural choice for the slice is given by 
\begin{equation*}
\cC^*(\Delta):=\Big\{f \in \cC(\Delta) : \int_{\Delta}f(x) g(x) \, dv =0\ \ \forall  \ g \ \text{affine-linear} \Big\}.
\end{equation*}
\end{ex}}

\begin{defn}\label{d:norm}
Let $||\cdot||$ be  a semi-norm on $\cC(\Delta)$ which indices a norm (in the
obvious sense) on $\cC^*(\Delta)$. We say that $||\cdot ||$ is {\it tamed} if there exists $C>0$ such that
$$ \frac{1}{C} || \cdot ||_1 \le || \cdot || \le  C || \cdot ||_{\infty},$$
where $|| \cdot ||_1:=\int_\Delta | \cdot | \,dv$  is the  $L^1$-norm  and $|| \cdot ||_{\infty}$ is the $C^0$-norm  on $\cC(\Delta)$. 
\end{defn}

{\tiny \begin{ex}\label{ex:L^p}  The $L^p$ norm  
$||f||_p : = \Big(\int_{\Delta} |f|^p \, dv\Big)^{1/p}$
defines a  tamed norm on the slices defined in Examples~\ref{e:normalization-donaldson} and \ref{e:normalization-gabor} for any $p\ge 1$.
\end{ex}}

{\tiny \begin{exercise}\label{ex:norm-donaldson}  Donaldson  considers  in \cite{Do2} the slice $\cC^*(\Delta)$ of Example~\ref{e:normalization-donaldson}  with the norm $||f||_* := \int_{\partial \Delta} f d\sigma$. Show that $||\cdot ||_*$ is tamed.
\noindent{\bf Hint.} Use that $f$ is a positive convex function.
\end{exercise}}
\begin{rem}\label{r:density}
(a) For any  tamed norm $||\cdot ||$,   both the  spaces of PL convex functions  and  smooth convex functions on the whole of $\Delta$ are dense in $\cC^*(\Delta)$.

(b) With respect to a tamed norm $|| \cdot ||$,   $\mathcal{F}_{(\Delta, {\bf L})} $ is well defined and continuous on $\cC^*(\Delta)$.   
\end{rem}


In this general situation arguments of Donaldson~\cite{Do2}
together with an enhancement by Zhou--Zhu~\cite{ZZ1} can be used to prove  the following key result.

\begin{prop}[\bf Donaldson~\cite{Do2}, Zhou--Zhu~\cite{ZZ1}]\label{p:uniform-K-stability} For any $\lambda>0$ the following are equivalent\textup:
\begin{numlist}
\item $\mathcal{F}_{(\Delta, {\bf L})} (f)\geq \lambda ||\pi(f)||$ for all $f\in\cC(\Delta)$, where $\mathcal{F}_{(\Delta, {\bf L})}$ is the Donaldson--Futaki linear functional \eqref{DF}\textup;
\item for all $0\leq \delta<\lambda$ there exists $C_\delta$ such that
${\mathcal E}_{(\Delta, {\bf L})}(u)\geq \delta||\pi(u)||+C_\delta$ for all $u\in \cS(\Delta, {\bf L})$, where $\mathcal E_{(\Delta, {\bf L})}$ is the relative K-energy introduced in \eqref{K-energy}.
\end{numlist}
\end{prop}
\begin{proof}

As $\cF_{(\Delta, {\bf L})}(f)$ and and ${\mathcal E}_{(\Delta, {\bf L})}(u)$ are unchanged by
the addition of an affine-linear function, it suffices to prove the
equivalence for normalized $f$ and $u$.

\smallbreak (i)$\Rightarrow$(ii)
For any bounded function $a$ on $\Delta$, one can define a modified Futaki
invariant $\cF_a$ by replacing the second integral (over $\Delta$) in the
formula  \eqref{DF}  by $-\int_\Delta a(x)f(x) \, dv$. Similarly,  one can
define a modified K-energy ${\mathcal E}_a$ using $\cF_a$ instead of $\cF_{(\Delta, {\bf L})}$ in the formula \eqref{K-energy}.  Donaldson~\cite{Do2}  shows that  $\cE_a$  (which is introduced on the space $\cS(\Delta, {\bf L})$) can be extended to $\cC_\infty(\Delta)$ (in fact on a slightly larger space) taking values in $(-\infty,+\infty]$. 

For any bounded functions $a,b$, there is a constant $C=C_{a,b}>0$ with
$|\cF_a(f)-\cF_b(f)|\leq C||f||$ for all $f\in\cC^*(\Delta)$, because
$||\cdot||$ bounds the $L^1$ norm on $\cC^*(\Delta)$ by assumption. Let us write
$C=(1+k)C-kC$ for an arbitrary $k\geq 0$ and take $b$ to be the extremal affine-linear
function $b= s_{(\Delta,{\bf L})}$, so that $\cF_b=\cF_{(\Delta, {\bf L})}$, whereas  we can take $a= s(u_0)$ be the scalar curvature of the canonical potential $u_0\in \cS(\Delta, {\bf L})$. Thus, $u_0$ trivially solves the  equation $s(u_0)=a$.

By assumption, $|\cF_a(f)-\cF_{(\Delta, {\bf L})}(f)|\leq
C\lambda^{-1}(1+k)\cF_{(\Delta, {\bf L})}(f)-k C||f||$ for all $f\in\cC_\infty^*(\Delta)$ and so
$\cF_a(f)\leq (1+C\lambda^{-1}(1+k))\cF_{(\Delta, {\bf L})}(f)-kC||f||$. Turning this
around,
\begin{equation*}
\cF_{(\Delta, {\bf L})}(f)\geq \eps\cF_a(f)+\delta||f||,
\end{equation*}
where $0<\eps:=(1+C\lambda^{-1}(1+k))^{-1}<1$ and
$\delta:=kC\lambda(\lambda+C(1+k))^{-1}$. Notice that $\delta$ is an injective
function of $k\in[0,\infty)$ with range $[0,\lambda)$. For any normalized $u\in \cS(\Delta, {\bf L})$ now we estimate
\begin{align*}
\cE_{(\Delta, {\bf L})}(u) &=-\int_\Delta\Big(\log\det(\mathrm{Hess}\, u)- \log\det(\mathrm{Hess}\, u_0)\Big) \,dv
+\cF_{(\Delta, {\bf L})}(u)\\
&\geq - \int_\Delta \Big(\log\det(\mathrm{Hess}\, u)- \log\det(\mathrm{Hess}\, u_0)\Big)\,dv
+\eps\cF_a(u)+\delta||u||\\
&=\cE_a(\eps u) + \delta||u|| + m {\rm Vol}(\Delta)\log \eps = \cE_a(\eps u) + \delta||\pi(u)|| + m {\rm Vol}(\Delta)\log \eps .
\end{align*}
It is shown in~\cite{Do2}, Proposition~3.3.4 that $\cE_a$ is bounded below on the space 
$\cC_\infty(\Delta)$.~\footnote{It follows along the lines of Corollary~\ref{bounded-mabuchi} that $\cE_a$ is  convex on $\cS(\Delta, {\bf L})$,  and  that $u_0$  being a critical point is  a global minima. Donaldson's argument extends this property to the larger domain $\cC_{\infty}(\Delta)$} Letting  $C_{\delta}$ be a lower bound of $\cE_a$ plus $m {\rm Vol}(\Delta)\log \eps $, the claim follows.

\smallbreak (ii)$\Rightarrow$(i) Suppose $\cE_{(\Delta, {\bf L})}(u)\geq
\delta||u||+C_\delta$ for all normalized $u\in\cS(\Delta, {\bf L})$. We shall fix one such $u$. By density and continuity (see Remark~\ref{r:density}),
it suffices to prove (i)  for $f\in\cC_\infty^*(\Delta)$ which are smooth on
$\Delta$.  Then for all $k>0$, $u+kf\in\cS(\Delta, {\bf L})$ and so $\cE_{(\Delta, {\bf L})}(u+kf)\geq
\delta||u+kf||+C_\delta$. We thus find 
\begin{align*}
k\cF_{(\Delta, {\bf L})} (f) &=  \cE_{(\Delta, {\bf L})}(u+ kf)- \cE_{(\Delta, {\bf L})}(u)  + \int_\Delta \log\Big(\frac{\det\mathrm{Hess}\,(u+kf)}
{\det\mathrm{Hess}\, u}\Big)\,dv \\
&\geq \delta||u+kf|| + C_\delta+
\int_\Delta \log\Big(\frac{\det\mathrm{Hess}\,(u+kf)}
{\det\mathrm{Hess}\, u}\Big)\,dv
-\cE_{(\Delta, {\bf L})}(u)\\
&\geq \delta||u+kf|| + \tilde C_\delta
\end{align*}
with $\tilde C_\delta=C_\delta-\cE_{(\Delta, {\bf L})}(u)$ (for the  fixed $u\in\cS(\Delta, {\bf L})$), since the ratio of the
determinants is at least one for $k$ sufficiently large. Dividing by $k$ and letting $k\to\infty$ we obtain $\cF_{(\Delta, {\bf L})}(f)\geq \delta ||f||$. Since this is
true for all $0\leq \delta<\lambda$ and  all smooth functions in $\cC_{\infty}(\Delta)$, we have (see Remark~\ref{r:density})  that $\cF_{(\Delta, {\bf L})}(f)\geq\lambda||f||$
for all $f\in\cC(\Delta)$.
\end{proof}
\begin{rem} In the case $|| \cdot || = || \cdot ||_*$ of  Example~\ref{e:normalization-donaldson} and Exercise~\ref{ex:norm-donaldson}, the implication (i)$\Rightarrow$(ii) of Proposition~\ref{p:uniform-K-stability}
is due to Donaldson~\cite{Do2} when $\delta=0$, and to Zhou--Zhu~\cite{ZZ1} for
some $\delta>0$. The proof presented here is due to David Calderbank~\cite{Cal}. 
\end{rem}

\begin{defn}[\bf Uniform K-stability~\cite{Sz}]\label{d:uniform K-stability} A compact convex simple labelled polytope $(\Delta, {\bf L})$ for which the condition (i) of Proposition~\ref{p:uniform-K-stability} is satisfied for some constant $\lambda>0$ is called {\it uniformly K-stable}  with respect to the chosen slice $\cC_{\infty}^*(\Delta)$ and norm $||\cdot ||$.  We say that $(\Delta, {\bf L})$  is {\it $L^p$ uniformly K-stable} if it is uniformly K-stable with respect to the slice introduced in Example~\ref{e:normalization-gabor} and the $L^p$ norm with $p\ge 1$, see Example~\ref{ex:L^p}. We say that $(\Delta, {\bf L})$ is  {\it $*$-uniformly K-stable} if it is uniformly K-stable with respect to the slice introduced in Example~\ref{e:normalization-donaldson} and norm in Exercise~\ref{ex:norm-donaldson}.
\end{defn}
Using that convex PL functions are dense in $\cC(\Delta)$ (see Remark~\ref{r:density}), uniform K-stability can be equivalently introduced  by requiring that  the condition (i) of Proposition~\ref{p:uniform-K-stability} is satisfied on convex PL functions. Thus, uniform K-stability  (with respect to any tamed norm) is a stronger condition on $(\Delta, {\bf L})$ than the K-stability.~\footnote{For a polarized smooth toric variety $(M, L)$, T. Hisamoto~\cite{Hisamoto} showed that the uniform K-stability of $(\Delta, {\bf L})$ with respect to $||\cdot ||_*$ agrees with a notion of equivariant uniform K-stability of $(M, L)$  with respect to the complex torus $\T^{\C}$.}

\bigskip
A key result in \cite{Do2} is the following
\begin{thm}[\bf Donaldson~\cite{Do2}]\label{thm:stable-to-iniform} If $(\Delta, {\bf L})$ is a compact convex labelled polygone in $\R^2$ such that the corresponding extremal affine-linear function $s_{(\Delta, {\bf L})}$ is strictly positive on $\Delta$, then $(\Delta, {\bf L})$ is K-stable if and only if it is $*$-uniformly K-stable.
\end{thm}
Notice that the above result applies in particular  to labelled polygons $(\Delta, {\bf L})$ for which $s_{(\Delta, {\bf L})}$ is constant.  In this case, we also have
\begin{thm}[\bf Sz\'ekelyhidi~\cite{Sz}]  If $(\Delta, {\bf L})$ is a compact convex labelled polygone in $\R^2$ such that the corresponding extremal affine-linear function $s_{(\Delta, {\bf L})}$ is constant, then $(\Delta, {\bf L})$ is K-stable if and only if it is $L^2$-uniformly K-stable.
\end{thm}
We end this section by mentioning the following generalization of Theorem~\ref{c:K-stability}.
\begin{thm}[\bf Chen--Li--Sheng~\cite{CLS0}]\label{thm:existence-to-uniform} If $(\Delta, {\bf L})$ is a compact convex simple labelled polytope in $\R^m$ such that the Abreu equation \eqref{Abreu-equation} admits a solution in $\cS(\Delta, {\bf L})$,  then $(\Delta, {\bf L})$ is $*$-uniformly K-stable.
\end{thm}

\section{Existence: an overview} 

\subsection{The Chen--Cheng and He results} We review here a recent breakthrough in the  general existence  theory of extremal K\"ahler metrics,   due to Chen--Cheng~\cite{CC1,CC3} in the constant scalar curvature case,  with an enhancement  by He~\cite{He} to cover the extremal K\"ahler metric case.  The key notion  for these results to hold is the {\it properness} of the relative K-energy with respect to a  reductive complex Lie group, first introduced by Tian~\cite{Tian} in the Fano case,  and adapted by Zhou--Zhu~\cite{ZZ1} to the general K\"ahler case.   We start by explaining the general setting.

\bigskip 

Let $(M, J)$ be a compact $m$-dimensional complex manifold admitting a K\"ahler metric $g_0$ with K\"ahler form, $\omega_0$.  It is well known (see e.g. \cite{gauduchon-book}) that the group of complex automorphisms ${\rm Aut}(M)$ of $(M, J)$  is a complex Lie group with  Lie algebra identified with the real smooth vector fields on $M$ whose flow preserves $J$. We denote by ${\rm Aut}_0(M)$ its connected component to the identity.

\smallskip
We let
\begin{equation*}
{\mathcal H}(M) =\{ \phi \in \cC^{\infty}(M)  : \omega_{\phi} = \omega_0 + dd^c \phi >0\}
\end{equation*}
be the space of smooth K\"ahler potentials  relative to $\omega_0$.   As $\R$ acts on  ${\mathcal H}(M)$ by translation, preserving  the K\"ahler metric  $\omega_{\phi}$, it is convenient to choose a slice (or normalization) $\mathcal H^0(M),$  establishing a bijection between K\"ahler metrics with K\"ahler forms in the deRham class $[\omega_0]$ and elements of $\cH^0(M)$. One popular normalization is
$\cH^0(M) = \big\{ \phi \in \cH(M)  :  \, \int_M \phi  \, \omega_{\phi}^m =0 \big\}$
but  we shall use below another choice for $\cH^0(M)$.  For  any chosen normalization $\cH^0(M)$, we  will write 
$$\phi = \phi^0 + const, $$
where $\phi\in \cH(M)$ and $\phi^0 \in \cH^0(M)$. Furthermore, for any $\sigma \in {\rm Aut}_0(M)$, we denote by $\sigma[\phi]  \in \cH^0(M)$ the  unique $\omega_0$-relative K\"ahler potential  in $\cH^0(M)$ associated to the K\"ahler form $\sigma^*(\omega_{\phi})$.

\smallskip
We now fix a (real) connected compact subgroup $K \subset {\rm Aut}_0(M)$ and denote by $G=K^{\C}\subset {\rm Aut}_0(M)$  its complexification, i.e. the smallest closed complex subgroup in  ${\rm Aut}(M)$ containing $K$. By a standard averaging argument, we can assume that the initial K\"ahler structure $(g_0, J, \omega_0)$ is invariant under the action of $K$,  and consider the subspace ${\mathcal H}(M)^K  \subset \cH(M)$ of  $K$-invariant K\"ahler potentials  in ${\mathcal H}(M)$.  Thus, with a chosen normalisation $\cH^0(M)$, the space ${\mathcal H}^0(M)^K:= \cH^0(M) \cap \cH(M)^K$ parametrizes the $K$-invariant K\"ahler metrics on $(M,J)$ whose K\"ahler forms belong to $[\omega_0]$.

\smallskip
Let $d$ be a distance on ${\mathcal H}(M)$.  For $\phi_1, \phi_2 \in {\mathcal H}(M)$, we let
\begin{equation}\label{G-distance}
d_{G}(\phi_1, \phi_2):=  \inf_{\sigma\in G} d(\phi_1, \sigma[\phi_2]),
\end{equation}
where, we recall,  $\phi^0 \in \cH^0(M)$ denotes the normalized K\"ahler potential of $\omega_{\phi}$ for the chosen normalization $\cH^0(M)$.

\begin{defn}[\bf Tian~\cite{Tian}] Let $\cE^K$ be a functional defined on the space $\mathcal H(M)^K$ and $d$ a distance on ${\mathcal H}^K(M)$. We say that $\cE^K$ is $G$-proper with respect to $d$ if 
\begin{enumerate}
\item[$\bullet$] $\cE^K$ is bounded  on $\cH(M)^K$;
\item[$\bullet$]  for any sequence $\phi_i \in \cH^0(M)^K$ with $d_G(\phi_0, \phi_i) \to \infty$,   $\cE^K(\phi_i) \to \infty$.
\end{enumerate}
\end{defn}
It is useful to notice that  if $K_0 \subset K$ is a compact subgroup with complexification $G_0 \subset G$,  and $\cE^{K}$ is the restriction to $\cH(M)^K$ of a functional $\cE^{K_0}$ on $\cH(M)^{K_0}$ which is $G_0$-proper, then $\cE^K$ is  also $G$-proper. We also notice that the above notion of properness depends upon the choice of normalization $\cH^0(M)$.

\bigskip
The normalization we shall  use  is given by
\begin{equation}\label{normalization}
\cH^0(M):=\{\phi \in \cH(M) : {\mathcal I}(\phi)=0\},
\end{equation}
where the functional ${\mathcal I} : \cH(M) \to \R$ is introduced by the formula
\begin{equation}\label{I}
{\mathcal I} (\phi) = \int_M \phi \Big(\sum_{j=0}^m \omega_0^{m-j}\wedge \omega_{\phi}^j \Big)=0.
\end{equation}
It is  known that (see \cite{gauduchon-book}, Chapter 4)  a smooth segment $\phi(t)$  starting at $\phi_0=0$  belongs to $\cH^0(M)$ iff 
\begin{equation}\label{differential-normalization}
\int_M \dot\phi(t)  \,  \omega_{\phi(t)}^m=0, \ \forall t.
\end{equation}
Furthermore, the distance relevant to us will be the one considered by Darvas in \cite{darvas}:  For any smooth segment $\phi(t) \in \cH(M), t\in [0,1]$, its length 
can be defined by $$\int_{0}^{1} \Big(\int_M  \big|\dot{\phi}(t)\big| \frac{ \omega_{\phi(t)}^m}{m!}\Big) dt.$$  We then  define $d_1(\phi_1, \phi_2)$ to be the infimum  of the lengths of  all segments with endpoints  $\phi_1$ and $\phi_2$.

\bigskip
It is known (see e.g. \cite{gauduchon-book}) that the group of complex automorphisms ${\rm Aut}_0(M)$  admits a closed  connected subgroup ${\rm Aut}_{r}(M)$, called {\it reduced} group of automorphisms, whose Lie algebra is the space of real vector fields whose flow preserves $J$,  and which vanish somewhere on $M$.  Furthermore, we let $\T \subset {\rm Aut}_{r}(M)$ be a maximal (real)  subtorus  of  ${\rm Aut}_{r}(M)$,  and denote by $G=T^{\C} \subset {\rm Aut}_{r}(M)$  its complexification (which is  a maximal  complex subtorus of ${\rm Aut}_{r}(M)$). By a result of Calabi (see \cite{gauduchon-book}, Chapter 3), if there exists an extremal K\"ahler metric $\omega_{\phi}$  for some $\phi \in \cH(M)$, then there is also an isometric extremal 
K\"ahler metric  $\omega_{\tilde \phi}$ with $\tilde \phi \in \cH(M)^{\T}$. Thus, without loss, one can reduce the problem of finding extremal K\"ahler metric in the de Rham class of $\omega_0$ to the related problem on $\cH^0(M)^{\T}$. Furthermore, there is  a natural functional $\cE^{\T} : \cH^0(M)^{\T} \to \R$, called the $\T$-relative K-energy  and introduced by Mabuchi~\cite{mabuchi} and Guan~\cite{Guan},  whose critical points are precisely the K\"ahler potentials in $\cH^*(M)^{\T}$,  corresponding to $\T$-invariant  extremal K\"ahler metrics in $[\omega_0]$ (see \cite{gauduchon-book}, Chapter 4).  Classically, the relative K-energy  is introduced  with respect to a fixed maximal subgroup $K\subset {\rm Aut}_r(M)$,  and acts on the space of K\"ahler potentials in $\cH^*(M)^K$, but  it is not difficult to see (using that the extremal vector field is central in the Lie algebra of $K$) that its definition actually extends to $\cH(M)^{\T}$ for {\it any} maximal torus $\T \subset K$.

The main  result of \cite{CC3,He} can be then stated as follows.
\begin{thm}[\bf Chen--Cheng~\cite{CC3}, He~\cite{He}]\label{strong} Suppose $\T\subset {\rm Aut}_{r}(M)$ is a maximal real torus and  $G=\T^{\C}$ its complexification inside ${\rm Aut}(M)$. If the relative  $K$-energy $\cE^{\T}$   acting on $\cH(M)^{\T}$  is $G$-proper with respect to the distance $d_1$ and the normalization~\eqref{normalization}, then there exists $\phi\in \cH(M)^{\T}$ such that $\omega_{\phi}$ is an extremal K\"ahler metric.
\end{thm}

\bigskip

We now turn to the toric case. Let $(M, \omega)$ be a smooth toric symplectic manifold,  classified by its Delzant polytope $(\Delta, {\bf L}, \Lambda)$.  In this case $\T$ is a maximal torus in ${\rm Aut}_{r}(M)$ for any $\T$-invariant  compatible  complex structure $J$ on $M$.

Using the identification \eqref{complex-identification}, for any $u\in \cS(\Delta, {\bf L})$ we associate a  function $F_u(z)$ on $(\C^m)^*$, defined by 
$$F_u(z) = \phi(y)= \phi(\frac{1}{2}\log|z_1|^2, \ldots, \frac{1}{2}\log|z_m|^2),$$ 
where $\phi(y)$ is given by  \eqref{legendre-transform}. Furthermore,  Lemma~\ref{symplectic-kahler potential} tels us that the $\T^m$-invariant  function $F_u(z)$  introduces a K\"ahler form $\omega_u := dd^c F_u(z)$ on $M^0=(\C^*)^m\cdot p_u$, which extends to a smooth K\"ahler metric $(g_u, \omega_u, J_u)$ on $M$  by identifying  $(M^0, J_u)$ to $(\C^*)^m$ via \eqref{complex-identification}.  
Let  $u_0$  be the canonical symplectic potential  given in Theorem~\ref{thm:guillemin}. Then, on $(\C^*)^m$,  the corresponding  K\"ahler forms $\omega_{0}:=dd^cF_{u}$  and $\omega_{u}=dd^c F_u$ are related by
$$\omega_u = \omega_{0} + dd^c(F_u(z) - F_{u_0}(z))= \omega_0 + dd^c \phi_u$$
for the $\T^m$-invariant smooth function $\phi_u(z):= F_u(z) - F_{u_0}(z)$ on $(\C^*)^m\cdot p_0$,    satisfying
\begin{equation}\label{normalization1}
\phi_u  \,  \textrm{extends smoothly on} \ M.
 \end{equation}
The point is that now $\omega_{u}$ and $\omega_0$ define two different K\"ahler metrics on the same complex manifold $M_{\Delta}^{\C}$, and $\phi_u$ extends to a $\T$-invariant K\"ahler potential with respect to $\omega_0$, i.e. $\phi_u \in \cH(M)^{\T}$.
Conversely, using the dual Legendre transform, one can show that any  $\T$-invariant  K\"ahler potential $\phi \in \cH(M)^{\T}$ with respect to $\omega_0$ gives rise to a  symplectic potential $u\in \cS(\Delta, {\bf L})$ with $\phi=\phi_u$.  A key feature of this  correspondance is  that if we take a path $u(t)\in \cS(\Delta, {\bf L})$ and consider the corresponding path $\phi(t):=\phi_{u_t} \in \cH(M)^{\T}$, formula \eqref{legendre-transform} shows that 
\begin{equation}\label{legendre-differential}
\dot{u}(t) = - \dot\phi(t).
\end{equation}

\bigskip
We shall now fix a point $p_0$ on $M^0$, corresponding to $x_0 \in \Delta^0$, a basis $e$ of $\Lambda$, and consider the slice 
$$\cC_{\infty}^*(\Delta):= \Big\{  f \in \cC^*_{\infty}(\Delta) : f(x_0)=0 \ \textrm{and} \ d_{x_0} f=0 \Big\} $$
introduced in Example~\ref{e:normalization-donaldson}.  We shall denote $\cS^*(\Delta, {\bf L}):=\cS(\Delta, {\bf L})\cap \cC_{\infty}^*(\Delta)$ the corresponding slice in the space of symplectic potentials. Under this normalization, we have $p_u =p_0$ in \eqref{complex-identification}.

However, the above normalization for the symplectic potentials  is not consistent with  the normalization \eqref{normalization} for the corresponding relative K\"ahler potentials, so we further consider the actions of $\R$   and $\T^{\C}$ on $\cH^0(M)^{\T}$. The point is that if $\tilde u = u + c$ for $u\in \cS^*(\Delta, {\bf L})$, then   we have  $\tilde y_i = y_i$ and $\tilde \phi = \phi - c$ in  \eqref{legendre-transform}. Similarly, if we act on $\phi\in \cH^0(M)^{\T}$ with an element $\gamma \in \T^{\C}$, this will result in  translation of $y_j$, and in modifying the corresponding symplectic potential by adding an affine-linear function.  We now introduce a different slice on $\cC_{\infty}(\Delta)$:
\begin{equation}\label{normalization3}
{\cC}_{\infty}^{0}(\Delta) := \Big\{ f \in \cC_{\infty}(\Delta) :  df_{x_0} =0 \ \textrm{and} \  \int_{\Delta} f \,dv = \int_{\Delta} u^*_0 \, dv\Big\},
\end{equation}
where, we recall,  $u_0^* = \pi(u_0)$ is the $\cC_{\infty}^*(\Delta)$ normalization of the canonical symplectic potential.  Letting $\cS^0(\Delta, {\bf L})= \cS(\Delta, {\bf L})\cap \cC^0_{\infty}(\Delta)$ be the corresponding slice in the space of symplectic potentials, we have
\begin{lemma}\label{l:normalization}  For any path $\tilde u(t)$ in $\cS^0(\Delta, {\bf L})$,  the corresponding $\omega_0$-relative K\"ahler potentials $\phi(t)=\phi_{\tilde u(t)}$  obtained by \eqref{legendre-transform} belong to $\cH^0(M)^{\T}$  and  satisfy
$$\frac{d}{dt} \tilde u(t)= - \frac{d}{dt} \phi(t).$$
Conversely, any path in $\cH^0(M)^{\T}$ comes from a path $\tilde u(t)$ in $\cS^0(\Delta, {\bf L})$, up to the action of $\T^{\C}$.
\end{lemma}
\begin{proof}
We have already observed that the formula for the variations holds   for any path.  Using this and \eqref{differential-normalization}, it follows that $\mathcal I(\phi(t))= const$. As  $\cS^0(\Delta, {\bf L})$ is convex and $u^*_0 \in \cS^0(\Delta, {\bf L})$, it follows that  $\mathcal I (\phi(t))=0$. In the other direction,  the relative K\"ahler potential $\phi(t)$ can be pulled back by an  element $\gamma_t \in \T^{\C}$ so that the symplectic potential associated to $\gamma_t^*(\phi_t )$ satisfies the normalization condition at $x_0$. \end{proof}
Through the above correspondence,  one can relate as in \cite{Do2,ZZ}  the relative  $K$-energy $\cE^{\T}$   acting on $\cH^0(M)^{\T}$  to the functional $\cE_{(\Delta, {\bf L})}$ defined in \eqref{K-energy}, acting on $\cS^0(\Delta, {\bf L})$. We are now ready to state and prove our main observation in this section.
\begin{prop}\label{p:toric-proper}  Suppose $(\Delta, {\bf L})$ is a $*$-uniformly K-stable  Delzant labelled polytope, corresponding to toric K\"ahler manifold $(M, \omega_0, J)$. Then the relative K-energy  $\cE^{\T}$ is $\T^{\C}$-proper on $\cH(M)^{\T}$ with respect to the distance $d_1$ and the normalization~\eqref{normalization}.
\end{prop}
\begin{proof} We use the bijection between the $\omega_0$-relative K\"ahler potentials in $\cH^0(M)^{\T}$ and symplectic potentials $\tilde u \in \cS^0(\Delta, {\bf L})$  established in Lemma~\ref{l:normalization} in order  to deduce  the $\T^{\C}$-properness of  $\cE^{\T}$  from the  property (ii) in Proposition~\ref{p:uniform-K-stability}  of the functional $\cE_{(\Delta, {\bf L})}$.

Let $\phi_j, j=1, \ldots, \infty$ be a sequence in $\cH^0(M)^{\T}$ with $d_{1, G} (0, \phi_j) \to \infty$.  Denote by $\tilde u_j \in \cS^0(\Delta, {\bf L})$ the corresponding sequence of normalized symplectic potentials. We write  
 $$\tilde u_j = u^*_j  +  \langle a_j, x \rangle   + c_j\ \ j=0, 1, \ldots, \infty,$$
 where $u^*_j:= \pi(\tilde u_j)$ is the projection of $\tilde u_j$ to the slice $\cS^*(\Delta, {\bf L})$ and $\langle a_j, x\rangle + c_j$ are affine-linear functions on $\tor^*$. By acting with an element $\gamma_j\in (\C^*)^m$ on $\phi_j$ we can assume that $a_j=0$ and thus
 $$\tilde u_j = u^*_j + \frac{1}{{\rm Vol}(\Delta)}\int_{\Delta} (u^*_0 - \tilde u^*_j) \, dv  \ \ j=0, 1, \ldots, \infty.$$
Let us consider the path $\tilde u_j(t)= (1-t)u_0^* + t \tilde u_j \in \cS^0(\Delta, {\bf L})$.  Using Lemma~\ref{l:normalization}, the $d_1$-length of this path is  $C_m \int_{\Delta}|\tilde u_j - u^*_0| \,dv$   where $C_m=(2\pi)^m$.  Thus, we have
\begin{equation}\label{distance-inequality}
\begin{split}
d_{1, G}(0, \phi_j) & \le d_1(0, \gamma_j[\phi_j]) \\
                             & \le C_m \int_{\Delta} |\tilde u_j - u^*_0 | \,dv \\
                            &= C_m \int_{\Delta} \Big| (u^*_j - u^*_0) +  \frac{1}{{\rm Vol}(\Delta)}\int_{\Delta} (u^*_0 - \tilde u^*_j) \, dv\Big| \, dv \\
                            & \le (C_m +1) \int_{\Delta}|u_j^* - u^*_0| \,dv.
                            \end{split}
                            \end{equation}
Thus,  $d_{1,G}(0, \phi_j) \to \infty$  yields  $\int_{\Delta} |u^*_j| \,dv \to \infty$.

Now, by assumption, $(\Delta, {\bf L})$ is $*$-uniform K-stable and  by Proposition~\ref{p:uniform-K-stability} (ii)  we have
\begin{equation*}
\begin{split}
\cE^{\T}(\phi_j)= \cE_{(\Delta, {\bf L})} (\tilde u_j)  &\ge \delta ||u_j^*||_* + C_{\delta} \\
                                                    & \ge \delta' \int_{\Delta} |u^*_j| \, dv  + C_{\delta} \to \infty, 
                                                    \end{split}
                                                    \end{equation*}
                                                    where for passing from the first line to the second we have used that $||\cdot||_*$ bounds the $L^1$-norm, see Exercise~\ref{ex:norm-donaldson}. \end{proof}

\begin{cor} Suppose $(\Delta, {\bf L})$ is a $*$-uniformly K-stable  Delzant labelled polytope, corresponding to a  toric K\"ahler manifold $(M, \omega_0, J)$. Then $(M, J)$ admits a $\T$-invariant extremal K\"ahler metric  whose K\"ahler form is in the deRham class $[\omega_0]$. Equivalently, the Abreu equation \eqref{Abreu-equation} admits a solution in $\cS(\Delta, {\bf L})$.
\end{cor}
\begin{proof} This follows from  Theorem~\ref{strong}  and Proposition~\ref{p:toric-proper},   using the bijection between the $\omega_0$ relative K\"ahler potentials in $\cH^0(M)$ and symplectic potentials $\tilde u \in \cS^0(\Delta, {\bf L})$ established in Lemma~\ref{l:normalization}. \end{proof}

\subsection{Donaldson's result}
We now discuss the Donaldson existence result in dimension $m=2$.
\begin{thm}[\bf Donaldson~\cite{Do-09}] Suppose $(\Delta, {\bf L})$ is a labelled compact  convex simple polytope in $\R^2$ for which the affine-linear extremal function
 $s_{(\Delta, {\bf L})}$  is constant and $(\Delta, {\bf L})$ is $K$-stable. Then \eqref{Abreu-equation} admits a solution. In particular, a compact symplectic toric $4$-manifold admits  a CSC K\"ahler metric iff it is K-stable and its extremal affine-linear function is constant.
 \end{thm}
 
An extension to the extremal case, under the $*$-uniform stability assumption and positivity of $s_{(\Delta, {\bf L})}$,  appears in \cite{CLS}.

\subsection{The continuity method}  The main idea of the approach in \cite{Do-09} is to use the continuity method over paths $(\Delta, {\bf L}(t)),  \  t\in [0,1]$ of labelled, compact, convex, simple Delzant polytopes in $\R^m$, where  the labels $L_j^t(x) =\langle u_j^t, x \rangle + \lambda_j^t$ vary in a continuous way. We thus introduce
\begin{defn}\label{normals-space} Let $\Delta \subset (\R^m)^*$ be a given  compact, convex, simple polytope which has $d$ codimension-one faces and
 $${\bf L}=\{L^1(x), \cdots, L_{d}(x)\}$$ is a chosen labelling. The cone  $N(r) \subset \R^d$ of admissible  labels of $\Delta$ is
$$N(r):=\Big\{ {\bf L}(r)=\{\frac{1}{r_1} L^1(x), \ldots, \frac{1}{r_d} L_d(x)\}:  \ r_i >0\Big\}.$$
Thus, ${\bf L}(1,\ldots, 1)={\bf L}$. We further introduce the subspaces
\begin{enumerate}
\item[$\bullet$] $SN(r) \subset N(r)$ the subset of normals  for which $(\Delta, {\bf L}(r))$ is $K$-stable;
\item[$\bullet$] $CN(r) \subset N(r)$ the subset of normals for which the extremal affine-linear function of $(\Delta, {\bf L}(r))$ is constant;
\item[$\bullet$] $SCN(r) = CN(r) \cap SN(r)$.
\end{enumerate}
\end{defn}
The basic observation is
\begin{lemma} $SN(r)$ is a convex subset of $N(r) \cong \{ r_i >0 \} \subset \R^d$ whereas $CN(r)$ is the (convex) cone obtained by intersecting $N(r)$ with a co-dimension $m$ hyperplane of $\R^d$. In particular, $SCN(r)$ is convex too. \end{lemma}
\begin{proof} By the definition \eqref{boundary-measure} of the measure on $\partial \Delta$, $d\sigma$ depends linearly on $r=(r_1, \ldots, r_{d})$.  Thus, the RHS of \eqref{extremal-affine-system} depends linearly on $r$, and so does the extremal affine-linear function $s_{(\Delta, {\bf L}(r))}$. This shows that $CN(r)$ is a the intersection of $N(r)$ with a co-dimension $m$  hyperplane  in $\R^d$. Finally, by the above arguments, the Donaldson--Futaki invariant $\mathcal{F}_{(\Delta, {\bf L}(r))}(\varphi)$ (see \eqref{DF})  depends linearly on $r$, showing the first claim too. \end{proof}
Thus, given a $K$-stable labelled polytope polytope  $(\Delta, {\bf L})$, the continuity method of \cite{Do-09} consists of establishing the following

\begin{enumerate}
\item[\bf Step 1.] There exists a canonical choice  $r_0 \in SCN(r)$ such that  on $(\Delta, {\bf L}(r_0))$ \eqref{Abreu-equation} admits a solution.  Connecting $r_0$ and $r_1:=(1,\ldots, 1)$ by a linear segment, we obtain a path of labelled polytopes $(\Delta, {\bf L}(t)), \ t\in [0,1]$ such that:  (a) for any $t$, $(\Delta, {\bf L}(t))$ is K-stable and its extremal affine-linear function is a constant; (b) for $t=0$, \eqref{Abreu-equation} admits solution.
\item[\bf Step 2.] If  \eqref{Abreu-equation} admits a solution in  $\cS(\Delta, {\bf L}(t_0))$ for $t_0 \in [0, 1]$, then there exists $\varepsilon_0>0$ such that  \eqref{Abreu-equation} admits a solution in  $\cS(\Delta, {\bf L}(t))$ for any $|t-t_0|<\varepsilon_0$.
\item[\bf Step 3.] The subset $\{ t \in [0, 1] :  \eqref{Abreu-equation} \ {\rm admits \  a  \ solution \  in}  \  \cS(\Delta, {\bf L}(t))\}$ is closed.
\end{enumerate}
The hardest part is Step 3.  We discuss below how to obtain Steps 1 and 2.
\subsection{The canonical solution} This concerns Step 1. The original method in \cite{Do-08} uses a result of Arezzo--Pacard. However, in \cite{Do-08}, an alternative,  simpler approach is suggested,  modulo a suitable generalization of a result of \cite{WZ} to arbitrary labelled polytopes, which is now available due to E. Legendre~\cite{eveline}. We follow this second approach.

\smallskip
Suppose $(\Delta, {\bf L}, \Lambda)$ is a Delzant triple.  It is well-known that the corresponding  K\"ahler  manifold $(M, J_0, \omega)$ is Fano, i.e. satisfies 
$$c_1(M, J_0) = \frac{1}{2\pi} [\omega]$$
if and only if $(\Delta, {\bf L}, \Lambda)$ is  a {\it reflexive lattice polytope}  in $(V^*, \Lambda^*)$, where $\Lambda^* \subset V^*$ denotes the dual lattice of $\Lambda \subset V$. Recall that a {\it lattice polytope} in $(V^*, \Lambda^*)$ is a polytope whose vertices belong to $\Lambda^*$. The {\it dual polytope}  $\Delta^*  \subset V$ is defined by 
$$\Delta^* =\{ v \in V : (v, x) + 1 \ge 0, \forall x \in \Delta\},$$
and $\Delta$ is {\it reflexive} if $\Delta^*$ is also a lattice Delzant polytope in $(V,  \Lambda)$.  It is easy to see that for such reflexive Delzant  polytopes $(\Delta, {\bf L}, \lambda)$, $L_j(0)=\lambda$ is independent of $j$, i.e., in a more affine-invariant way,  there is an interior point $x_0\in \Delta$ and a positive constant $\lambda>0$ such that
\begin{equation}\label{monotone}
L_j(x_0)=\lambda,   \ \forall L_j \in {\bf L}
\end{equation}

Of course, this condition makes sense for an arbitrary labelled convex polytope $(\Delta, {\bf L})$
\begin{defn} We say that a convex, simple labelled polytope $(\Delta, {\bf L})$ is {\it monotone} if there exists an interior point $x_0\in \Delta$   for which \eqref{monotone} is satisfied.
\end{defn}
The reason for the name is the following
\begin{lemma}\label{l:monotone} Suppose $(\Delta, {\bf L})$ is a monotone labelled simple convex polytope and and $u\in \cS(\Delta, {\bf L})$ a symplectic potential defining a K\"ahler metric \eqref{g} on $\Delta^0\times \T^m$. Then,  
$$\rho_g - \frac{1}{\lambda} \omega = \frac{1}{2}dd^c f,$$
for a smooth function on $\Delta,$ called a Ricci potential of $u$.
\end{lemma}
\begin{proof} We can translate $(\Delta, {\bf L})$ and assume without loss of generality that $x_0$ is the origin in $V$.  We now use \eqref{ricci} and Lemma~\ref{symplectic-kahler potential} to write
\begin{equation*}
\rho_g - \frac{1}{\lambda }\omega  = \frac{1}{2} dd^c \Big(\log \det \big({\rm Hess}(u)\big)  -  \frac{2}{\lambda}(-u(x) + \sum_{i=1}^m x_i u_{,i})\Big).
\end{equation*}
Using the boundary conditions restated in Exercise~\ref{e4},  it is easily seen that 
\begin{equation}\label{ricci-potential}
\begin{split}
f(u):= & \log \det \big({\rm Hess}(u)\big) - \frac{2}{\lambda}(-u(x) + \sum_{i=1}^m x_i u_{,i})\\
  = & \sum_{j=1}^d \Big(-\log L_j - \frac{1}{\lambda}(- L_j \log L_j  + \sum_{i=1}^{m} x_i(L_j \log L_j)_{, i})\Big) + \ {\rm smooth}\\
  = & \sum_{j=1}^d \Big(-\log L_j - \frac{1}{\lambda}(- L_j \log L_j  + (L_j-\lambda)\log L_j)\Big) + \ {\rm smooth}\\
  = & \ \  {\rm smooth}.
  \end{split}
  \end{equation}
\end{proof}
If  $(\Delta, {\bf L}, \Lambda)$ is a Delzant triple satisfying \eqref{monotone}, using that $\frac{1}{2\pi}\rho_g$ represents $c_1(M, J)$,   Lemma~\ref{l:monotone} explicitly shows the Fano property of $(M, J, \omega)$. In this case, Wang--Zhu~\cite{WZ} proved the following important result
\begin{thm}[\bf Wang--Zhu~\cite{WZ}]\label{KE}  Suppose $(\Delta, {\bf L})$ is a monotone Delzant polytope for which the extremal affine-linear  function $s_{(\Delta, {\bf L})}$ is constant. Then there exists $u\in \cS(\Delta, {\bf L})$ such that $\rho_g = \frac{1}{\lambda} \omega$, i.e. the corresponding $\T$-invariant, $\omega$-compatible  metric on $(M, \omega, \T)$ is  K\"ahler--Einstein. In particular, $u$ is a solution to \eqref{Abreu-equation}.
\end{thm}
The proof of this result itself uses the continuity method.  By Lemma~\ref{l:monotone}, we want to solve for each $t\in [0,1]$
\begin{equation}\label{MA}
\frac{\det \big({\rm Hess}(u_0 + \varphi)\big)}{\det \big( {\rm Hess} (u_0)\big)} = \exp{\Big(\frac{2t}{\lambda}(-\varphi(x) + \sum_{i=1}^m x_i \varphi_{,i}) - f_0(x)\Big)},
\end{equation}
where $(u_0, f_0)$  are the symplectic and Ricci potentials of the canonical metric $(g_0, J_0)$.  For $t=0$, the existence of a solution $u_0 + \varphi \in \cS(\Delta, {\bf L})$ follows from Yau's theorem~\cite{Yau} of prescribing the Ricci form in a given K\"ahler class on $M_{\Delta}^{\C}=(M, J_0)$. The openness follows too from a standard argument  on $M_{\Delta}^{\C}$. The closedness reduces, via the method of \cite{Yau},  to a uniform $C^0(M)$ bound for the solutions $\varphi_t$ of \eqref{MA}, see e.g. \cite{TZ, Z}. The latter is derived in \cite{WZ} by using only the convexity  of $\Delta \subset (\R^m)^*$.

\bigskip
An extension of Theorem~\ref{KE} to arbitrary compact convex simple labelled polytopes was obtain by E. Legendre~\cite{eveline}, building on the general  strategy suggested in \cite{Do-08} to apply the machinery of global analysis to compact simple  labelled polytopes which are not necessarily Delzant. Recall that by Theorem~\ref{complex-toric}, if $(\Delta, {\bf L})$ is Delzant, there are natural holomorphic chart $\C^m_v$ associated to a vertex $v$ of $\Delta$, and two different charts are identified using the relation \eqref{affine-chart}. We can match this with the discussion in Section~\ref{s:complex}. Indeed, fixing a reference point $p_0\in M^0$ with $x_0=\mu(p_0) \in \Delta^0$ and, by adding an affine-linear function to $u$,  we assume $y_i(x_0)= u_{, i}(x_0)=0$. Furthermore, at each vertex $v\in \Delta$ we use the identification \eqref{complex-identification} with $e_v=\{u^v_1, \ldots, u^v_m\}$ being the basis obtained from the normals of facets containing $v$. We can also translate the vertex $v$ at the origin so that, by the boundary conditions,  $u= \frac{1}{2} \sum_{j=1}^m (x_j \log x_j - x_j)+ \varphi $ for some smooth function $\varphi(x)$ on $\Delta$. Then, the $J$ holomorphic coordinates  
$$z_j^u =e^{u_{,j}} e^{\sqrt{-1}t_j} =  z_j^0 e^{\varphi_{,j}}$$
allow to extend the identification $\Phi_{u, p_0, e_v} :  (\C^*)_v^m \to M^0$ to an equivariant embedding  $\Phi_{u, p_0, e_v} : \C^m_v \to (M, J)$. Choosing another vertex $w$, we get that $\Phi_{u, p_0,e_v}$  and $\Phi_{u,p_0,e_w}$ are related via \eqref{affine-chart}.  Now, for any $u\in \cS(\Delta, {\bf L})$,  Lemma~\ref{symplectic-kahler potential}  shows that  
the function on $\C^m_v$ defined by 
$$F_v(z) = \phi(y)= \phi(\frac{1}{2}\log|z_1|^2, \ldots, \frac{1}{2}\log|z_m|^2)$$ satisfies
\begin{enumerate}
\item[$\bullet$]  $2i \partial \bar \partial  F_v >0$,
\item[$\bullet$] $F_v(z)$ is invariant under the linear action of $\T_v^m$ on $\C^m_v$;
\item [$\bullet$] $F_w = F_v + \sum_{i=1}^m (v-w)_i \log |z_i|,$
\end{enumerate}
where the third relation reflects the fact that each time we translate the polytope to ensure the corresponding vertex is at the origine. Conversely, any such collection of functions gives rise to $u \in \cS(\Delta, {\bf L})$.

\smallskip
When $(\Delta, {\bf L})$ is not Delzant, we cannot use \eqref{affine-chart} in order to construct a complex manifold $M^{\C}_{\Delta}$. But we still have  ``charts''
for  the ``quotient space'' $|M| = M^{\C}_{\Delta}/\T \cong \Delta$. More precisely, for each  vertex  $v$ we have  maps $\Phi_{u, v} : \C_v^m  \to \Delta$  where $r_i=|z_i|$ are  the radial coordinates on $\C_v^m$.  We can thus define an embedding of $\Phi_{u, v} : [0, \infty)_v^m \to \Delta$ by letting $r_i = y_i=u_{,i}$. The identification \eqref{affine-chart} has now sense if we replace $z_i$ with $r_i$, and thus we associate to $\Delta$ a Hausdorff topological space $|M|$ which admits ``uniformizing'' charts  $\Phi_{v} : \C^m_v /\T^m_v \to |M|$, precisely as in the definition \eqref{d:orbifold} for orbifolds, but now the groups $\Gamma_i \cong \T_v^m$ are tori. Furthermore, we can define the sheaf of smooth functions of $|M|$ by pull back to each $\C_v^m$, etc. Using the relevant analysis in such charts, Legendre has proved
\begin{thm}[\bf Legendre~\cite{eveline}] Theorem~\ref{KE} holds true for any compact convex simple labelled polytope $(\Delta, {\bf L})$.
\end{thm}
The main advantage of this is the following
\begin{thm}[\bf Donaldson~\cite{Do-08}, Legendre~\cite{eveline}] Let $(\Delta, {\bf L})$ be a labelled compact convex simple polytope in $(\R^m)^*$.  Then there exits a unique up to an overall positive scale labelling ${\bf L}^*$ such that $(\Delta, {\bf L}^*)$ is monotone  and $s_{(\Delta, {\bf L}^*)}$ is constant. \end{thm}
\begin{proof}
For any interior point $x^* \in \Delta^0$ we let  $r(x^*)=(L^1(x^*), \ldots, L_d(x^*)) \in N(r)$ and 
$${\bf L}^* := \Big(\frac{L^1(x)}{L^1(x^*)}, \ldots, \frac{L_d(x)}{L_d(x^*)}\Big),$$
be the corresponding label, so that $(\Delta, {\bf L}^*)$ is monotone. We want to show that one can choose $x^*$ so that $s_{(\Delta, {\bf L}^*)}\equiv a_0^*$ is  a constant function. By the last relation in \eqref{extremal-affine-system}, 
$$a_0^*= 2\frac{\int_{\partial \Delta} d\sigma_{{\bf L}^*}}{{\rm Vol}(\Delta)}$$
whereas the other $m$ relations  give
$$\int_{\partial \Delta} x_i d\sigma_{{\bf L}^*} = \frac{\int_{\partial \Delta} d\sigma_{{\bf L}^*}}{{\rm Vol}(\Delta)} \int_{\Delta} x_i dv.$$
We thus obtain a linear system for the unknown $x^*$:
\begin{equation}\label{system}
\sum_{j=1}^d L_j(x^*) \int_{F_j} x_i d\sigma_{L_j}  = \Big(\frac{\int_{\Delta} x_i dv}{\int_{\Delta} dv}\Big)\Big(\sum_{j=1}^d L_j(x^*) \int_{F_j}d\sigma_{L_j} \Big), \ \ i=1, \ldots, m.
\end{equation}
In order to show that \eqref{system} has a unique solution we shall use the following elementary 
\begin{lemma}\label{l:mass} Let $f(x)=\langle u, x\rangle + \lambda$  be an affine-linear function on $\R^m$. Then
$$\sum_{j=1}^d  \Big(\int_{F_j} f d\sigma_{L_j} \Big) u_j = - \Big(\int_{\Delta} dv \Big) u  \in (\R^m)^*.$$
\end{lemma} 
The proof of  Lemma~\ref{l:mass} is left as an exercise, and can be obtained easily starting from a triangulation of $\Delta$ as the union of simplexes with a common vertex at the interior of $\Delta$.

\smallskip
We now suppose, without loss of generality, that $0 \in \Delta^0$, and the initial labelling ${\bf L}$ satisfies $L_j(0)=1,$ i.e. 
$$L_j(x)=\langle u_j, x \rangle + 1, \ \ j=1, \ldots, d.$$
By Lemma~\ref{l:mass}, 
\begin{equation}\label{star}
\begin{split}
x^*_i &= -\frac{1}{{\rm Vol}(\Delta)}\sum_{j=1}^d \langle u_j, x^*\rangle \int_{F_j} x_i d\sigma_{L_j} \\
       &= -\frac{1}{{\rm Vol}(\Delta)}\sum_{j=1}^d L_j(x^*) \int_{F_j} x_i d\sigma_{L_j} + \frac{1}{{\rm Vol}(\Delta)}\int_{\partial \Delta} x_i d\sigma_{\bf L}\\
       &= \frac{1}{{\rm Vol}(\Delta)}\Big(\int_{\partial \Delta} x_i d\sigma_{\bf L} -\int_{\partial \Delta} x_i d\sigma_{\bf L^*}\Big).  \\
       0 & = \langle u_j, x^*\rangle \int_{F_j}  d\sigma_{L_j} = \sum_{j=1}^d L_j(x^*) \int_{F_j} d\sigma_{L_j} -\int_{\partial \Delta}  d\sigma_{\bf L}\\
          &= \int_{\partial \Delta} d\sigma_{\bf L} -\int_{\partial \Delta}  d\sigma_{\bf L^*}
       \end{split}
       \end{equation}
and \eqref{system} becomes 
$$x_i^*=\frac{1}{{\rm Vol}(\Delta)^2}\Big( \int_{\Delta} dv\int_{\partial \Delta} x_i d\sigma_{\bf L} - \int_{\Delta} x_i dv \int_{\partial \Delta} d\sigma_{\bf L}\Big).$$


\end{proof}

\end{document}